\documentclass[12pt]{amsart} 
\usepackage[mathscr]{eucal} 
\usepackage{amsmath,amsfonts} 
\usepackage{amssymb}
\usepackage{color}
\usepackage{graphicx}
\usepackage{algorithm}
\usepackage{algorithmic}
\usepackage{hyperref}
\parskip=\smallskipamount 
\newtheorem{theorem}{Theorem}[section] 
\newtheorem{proposition}[theorem]{Proposition} 
 
\newtheorem{lemma}[theorem]{Lemma} 
\theoremstyle{definition} 
\newtheorem{definition}[theorem]{Definition} 
\newtheorem{example}[theorem]{Example}

\newtheorem{remark}[theorem]{Remark} 

\makeatletter 
\@addtoreset{equation}{section} 
\makeatother

\newcommand{\CC}{{\mathbb C}} 
\newcommand{\NN}{{\mathbb N}}

\newcommand{\RR}{{\mathbb R}} 
 
\newcommand{\cA}{{\mathcal A}} 
\newcommand{\cB}{{\mathcal B}} 
\newcommand{\cC}{{\mathcal C}} 
 
\newcommand{\cE}{{\mathcal E}} 
\newcommand{\cF}{{\mathcal F}} 
\newcommand{\cG}{{\mathcal G}} 
\newcommand{\cH}{{\mathcal H}} 
\newcommand{\cI}{{\mathcal I}}
\newcommand{\cJ}{{\mathcal J}} 
\newcommand{\cK}{{\mathcal K}} 
\newcommand{\cL}{{\mathcal L}} 
\newcommand{\cM}{{\mathcal M}}

\newcommand{\cS}{{\mathcal S}} 
 
\newcommand{\cU}{{\mathcal U}} 
\newcommand{\cV}{{\mathcal V}} 
\newcommand{\cW}{{\mathcal W}} 
\newcommand{\cX}{{\mathcal X}}
\newcommand{\cY}{{\mathcal Y}}

\newcommand{\bH}{\mathbf{H}}

\newcommand{\fC}{\mathfrak{C}}
\newcommand{\fH}{\mathfrak{H}}
\newcommand{\fM}{\mathfrak{M}}

\newcommand{\dom}{\operatorname{Dom}}

\newcommand{\ran}{\operatorname{Ran}} 
 
\newcommand{\ra}{\rightarrow}

\let\phi=\varphi 
\newcommand{\iac}{\mathrm{i}}

\newcommand{\de}{\mathrm{d}}

\newcommand{\supp}{\operatorname{supp}}
\newcommand{\Lloc}{\cB_{\mathrm{loc}}}
\newcommand{\Lcoh}{\cL_{\mathrm{coh}}}
\newcommand{\essran}{\operatorname{ess-ran}}
\newcommand{\esssup}{\operatorname{ess-sup}}

\newcommand{\nr}[1]{\vspace{0.1ex}\noindent\hspace*{12mm}\llap{\textup{(#1)}}} 
 
\begin{document} 
\title[Representing Locally Hilbert Spaces]{Representing Locally Hilbert Spaces and Functional Models for Locally Normal Operators}

 \date{\today}
  
\author[A. Gheondea]{Aurelian Gheondea} 
\address{Institutul de Matematic\u a al Academiei Rom\^ane, Calea Grivi\c tei 21,  010702 Bucure\c sti, Rom\^ania \emph{and}
Department of Mathematics, Bilkent University, 06800 Bilkent, Ankara, 
Turkey} 
\email{A.Gheondea@imar.ro \textrm{and} aurelian@fen.bilkent.edu.tr} 

\begin{abstract} The aim of this article is to explore in all remaining aspects the spectral theory of locally 
normal operators. In a previous article we proved the spectral theorem in terms of locally spectral measures. 
Here we prove the spectral theorem in terms of projective limits of certain 
multiplication operators with locally $L^\infty$ functions. In order to do this, 
we first investigate strictly inductive systems of measure spaces and point out the concept of representing 
locally Hilbert space for which we obtain a functional model as a strictly inductive limit of 
$L^2$ type spaces.
Then, we first obtain a functional model for locally normal operators on representing locally Hilbert spaces 
combined with a spectral 
multiplicity model on a pseudo-concrete functional model for the underlying locally Hilbert space, under the 
technical condition on the directed set to be sequentially finite.
Finally, under the same technical condition on the directed set, we 
derive the spectral theorem for locally normal operators in terms of projective limits of certain 
multiplication operators with locally  $L^\infty$ functions. As a consequence of 
the main result we sketch the direct integral representation of locally normal operators under the same technical 
assumptions and the separability of the locally Hilbert space. Examples of 
strictly inductive systems of measure spaces involving the Hata tree-like self-similar set, which justify the 
technical condition on a relevant case and which may open a connection with analysis on fractal sets, 
are included.
\end{abstract} 

\maketitle 

\section*{Introduction}

This article is aimed at complete 
the study of the spectral theory of locally normal operators as a continuation
of our investigations from \cite{Gheondea1}, \cite{Gheondea2}, 
and \cite{Gheondea3} on operator theory on locally Hilbert spaces. 
On the one hand, these investigations are partially motivated  by the theory 
of certain locally convex $*$-algebras which was initiated by G.~Allan \cite{Allan}, C.~Apostol \cite{Apostol}, 
A.~Inoue in \cite{Inoue}, and K.~Schm\"udgen \cite{Schmudgen}, and continued 
by N.C.~Phillips \cite{Phillips},  and by the theory of Hilbert modules over locally convex 
$*$-algebras initiated by the works of A.~Mallios \cite{Mallios} and D.V.~Voiculescu \cite{Voiculescu}. On the 
other hand, a strong motivation comes from
the analysis on fractal sets, cf.\ J.~Kigami \cite{Kigami}, differential equations on fractal sets, cf.\ 
R.S.~Strichartz \cite{Strichartz}, and spectral theory on the Hata tree-like self-similar set, cf.\ A.~Brzoska
\cite{Brzoska}. More precisely, following J.~Kigami \cite{Kigami},
one of the main question in the theory of self-similar sets in the sense of J.E.~Hutchinson \cite{Hutchinson}, 
which are the most studied fractal type sets, is the definition and the spectral properties of the Laplace operator 
on such a set. The first challenging question is the definition of the Laplace operator on a given self-similar set 
$X$. There are at least two approaches to this problem. The first was the probabilistic approach which, roughly 
speaking, consists in taking a sequence of random walks on the graphs which approximate $X$ and show that 
those random walks converge to a diffusion process on $X$. The second was the analytic approach which, 
roughly speaking, means that one has to study the structures of harmonic functions, Green's functions and 
Dirichlet forms, solutions of Poisson's equations. Both ways are tedious and need a long sequence of results 
to be proven, see \cite{Kigami} and the rich bibliography cited there. Our hope is that 
the study of operators on locally Hilbert spaces will open a new method to approach these problems. In this 
article we only considered an example generated by the Hata tree-like self-similar set. This example is used to 
show a couple of variants of organising it as a strictly inductive system of measure spaces and justifies the 
technical condition in the main theorems.

Previously we obtained an operator model for locally Hilbert $C^*$-modules in \cite{Gheondea2} and then 
studied
locally Hilbert spaces from the topological point of view in \cite{Gheondea1}. From a different point of view,
operator theory in locally Hilbert spaces was considered by A.A.~Dosiev \cite{Dosi1}, who called them 
quantised domains of Hilbert spaces, and
by D.J.~Karia and Y.M.~Parmar in \cite{KariaParmar}. In \cite{Gheondea3} we 
obtained the First Spectral Theorem, that is, in terms of spectral measures, 
and the Borel Functional Calculus Theorem for locally normal operators. For the special case of sequentially 
ordered locally Hilbert spaces, an analysis of partially defined locally operators has been performed by 
E.-C.~Cismas in \cite{Cismas}.
In this article we obtain a generalisation of the Second Spectral Theorem for normal operators, that is, in the
functional model representation, see J.B.~Conway \cite{Conway}, for the case of locally normal operators.  

Our main results are Theorem~\ref{t:st2} and Theorem~\ref{t:st3} in which we obtain two variants of the 
Second Spectral Theorem on locally Hilbert spaces, in the sense of multiplication operators with locally 
$L^\infty$ functions, under the condition on the directed set to be sequentially finite. 
From here, we derive in Theorem~\ref{t:st4} the Third Spectral Theorem for locally normal operators, that is, in 
terms of direct integrals of locally Hilbert spaces, under the same conditions on the directed set and the 
separability of the locally Hilbert space. Locally normal operators generate Abelian locally 
von Neumann algebras, which are operator algebra types of locally $W^*$-algebras, see
M.~Fragoulopoulou, \cite{Fragoulopoulou} and M.~Joi\c ta, \cite{Joita1}, \cite{Joita2}. Also, A.A.~Dosi in 
\cite{Dosi3} made a study of Abelian $W^*$-algebras. Recently, C.J. Kulkarni and S.K. Pamula 
\cite{KulkarniPamula} considered direct integral representations of a special type of Abelian locally von Neumann 
algebras.  None of these results cover our main theorems.
Our approach for proving Theorem~\ref{t:st2} and Theorem~\ref{t:st3} 
is made by considering strictly inductive systems of measure spaces and points out 
the importance of representing locally Hilbert space, which are called commutative quantised domains of Hilbert 
spaces in \cite{Dosi1}, for which we obtain a functional model as a strictly 
inductive limit of $L^2$ type spaces. Then, the functional model of locally normal operators is obtained
as multiplication operators with locally $L^\infty$ type functions, in a sense which we make precise, by exploiting 
the structure of Abelian von Neumann algebras. In the following we briefly describe the contents of this article.

We briefly 
review in Section~\ref{s:ipl} the basics on inductive and projective limits to a degree of generality which is 
required by this article, a more general point of view compared to what we did in \cite{Gheondea1} 
and \cite{Gheondea2}. 
One of our main concerns is to make a clear distinguish between projective limits,
which have good properties, and inductive limits, which miss some of the good properties, and somehow 
motivates the limitation on using the latter in a greater generality, e.g.\ see Y.~Komura
\cite{Komura}. In addition, we point out the importance of
coherence, when approaching linear maps between inductive or projective limits. Although this tool is 
considered to be classical mathematics and scattered in many textbooks or monographs, e.g.\ see 
A.~Grothendieck \cite{Grothendieck} and G.~K\"othe \cite{Kothe}, 
the generality that we use in this article is seldom encountered, e.g.\ see J.L. Taylor's lecture notes 
\cite{Taylor}. In addition, in this article we need a more general approach to inductive limits and projective limits, 
in the sense of categories, see N. Bourbaki \cite{Bourbaki} and S.~Mac Lane \cite{MacLane}, but in a more 
concrete fashion, that starting with inductive systems and projective systems. 

Section~\ref{s:lhs} starts by recalling a few basic things about locally Hilbert spaces and their topological and 
geometrical properties from \cite{Gheondea1}. 
Subsection~\ref{ss:lbo} starts by reviewing the basic definitions and facts on locally bounded operators.  Two of the most useful facts that we 
get here are Proposition~\ref{p:lbo2}, 
which shows when and how we can assemble a locally bounded operator
from its many pieces, and Proposition~\ref{p:mat} which provides a block-matrix characterisation of locally 
bounded operators in terms of any fixed component. After recalling the basic terminology and facts on locally
$C^*$-algebras in Subsection~\ref{ss:lcsa}, we provide two relevant examples, one in terms of continuous
functions and the other in terms of locally $L^\infty$-functions. 

Section~\ref{s:rlhs} is meant to reformulate the definition of locally Hilbert spaces with one 
more property, namely that of commutation of the orthogonal projections onto the 
components $\cH_\lambda$, from which we get the concept of a representing locally 
Hilbert space. This concept is firstly justified by the fact that, in the original proof of Inoue \cite{Inoue} of 
the Gelfand-Naimark's Theorem, this property is implicit. This fact was observed and used for other purposes 
in \cite{Dosi2} as well. Our goal is to obtain a rather general functional 
model for locally Hilbert spaces and for this reason we consider strictly inductive systems of measurable spaces 
and strictly inductive systems of measure spaces, in connection with their inductive limits. 
Thus, this additional property of representing holds for the case of 
strictly inductive systems of measure spaces and hence, it makes the first step in the direction of functional 
models and, in the same time, provides a second justification for introducing the concept. 
The results presented in 
subsections \ref{ss:sisms} and \ref{ss:sismes} are related to previous investigations on inductive and projective 
systems of measure spaces as in N.~Dinculeanu~\cite{Dinculeanu}, Z.~Frolik~\cite{Frolik}, and 
N.D.~Macheras~\cite{Macheras}, as well as in the recent preprint \cite{KulkarniPamula}, but we have a more 
specific framework which should be made clear. Here the main result is contained in Proposition~\ref{p:eltwo} which 
says that the locally Hilbert space
which is the inductive limit of $L^2(X_\lambda,\mu_\lambda)$, for a strictly inductive system of measure spaces
$((X_\lambda,\Omega_\lambda,\mu_\lambda))_{\lambda\in\Lambda}$, can be completed to a concrete
Hilbert space $L^2(X,\widetilde\mu)$, for a measure space $(X,\widetilde\Omega,\widetilde\mu)$ that is
precisely described in Proposition~\ref{p:wimu}. 

As we mentioned before, our more general objective is to relate spectral theory on locally Hilbert spaces with the 
theory of analysis on self-similar (fractal) sets. For this reason, three relevant examples of strictly inductive 
systems of measure spaces obtained from the Hata tree-like self-similar set are presented in 
Subsection~\ref{ss:hata}. 
In order to do this, we explore with more details the topological structure of this self-similar set. In addition, the 
three variants of organising the Hata tree-like self-similar set as an inductive limit of a strictly inductive system of 
measure spaces share certain extra properties that are used to show that the technical condition on the directed 
set that we impose in the main results are natural and useful in applications. This direction of research is 
susceptible to be followed in view of a spectral theory on self-similar sets by means of locally Hilbert spaces.

In Section~\ref{s:st} we consider the basic concepts, like spectrum and resolvent sets. 
Subsection~\ref{ss:lno} is dedicated to a first encounter of locally normal operators and their basic properties. In 
Example~\ref{ex:lnop} we show that our previous efforts to develop functional models prove worthwhile,
allowing us to obtain the analogue, in its local version, of the multiplication operators with $L^\infty$ functions on 
$L^2$ type spaces. This clarifies what kind of functional model for locally normal operators we 
are looking for. Then we briefly review the First Spectral Theorem for locally normal operators, that is, in terms of 
spectral measures, previously obtained in \cite{Gheondea3}. An important fact is Example~\ref{ex:psconc} which 
shows a functional model that contains information on the spectral multiplicity and points out a one-sided 
connection between the First Spectral Theorem and the Second Spectral Theorem.

We reach our goal in two steps. First, under the technical condition on the
directed set to be sequentially finite,
in Theorem~\ref{t:st2} we show that any locally normal operator on a representing locally Hilbert space
is unitarily equivalent to a projective limit of operators of multiplication by a locally $L^\infty$ function on a 
representing pseudo-concrete locally Hilbert space, which contains the spectral multiplicity aspect as well. The 
technical tool here is the general structure theory of commutative von Neumann algebras, where the fact that 
the locally Hilbert space is representing is essential. The technical condition on the directed set to be 
sequentially finite is used in order to allow us to
use mathematical induction in the construction of the strictly inductive system of measure spaces. 
Then, in Theorem~\ref{t:st3} we transfer this into a genuine functional model on a concrete representing 
locally Hilbert space. In Subsection~\ref{ss:ttst} we sketch the direct integral representation, for locally normal 
operators, which is usually called the Third Spectral Theorem, as a consequence of 
Theorem~\ref{t:st2}, hence under the same assumptions, and, in addition, under the assumption of
separability of all Hilbert spaces $\cH_\lambda$. In Subsection~\ref{ss:last} we make some 
final comments on the relations between the three spectral theorems on locally normal operators that we present 
in this article, which are quite different when compared to the classical spectral theorems on normal operators, and further research questions that these ideas may trigger.
\medskip

\textbf{Acknowledgements.} We thank Daniel Belti\c t\u a for many conversations on the topics of this manuscript 
and useful observations. Thanks are also due to L\'aszlo Zsid\'o for an exchange of messages 
on Theorem~\ref{t:abvn} and to Chaitanya J.~Kulkarni and Santhosh Kumar Pamula for very useful 
observations which improved the form of Lemma~\ref{l:letexem}. 
During the Spring semester in 2024 I enjoyed the supervision of 
a Senior Project of Ziya Ka\u gan  Akman on analysis on self-similar sets and this experience was useful for me 
in Subsection~\ref{ss:hata}.   

\section{Inductive and Projective Limits}\label{s:ipl}

In this section and throughout this article, we use inductive limits, projective limits, and coherent maps. 
We briefly recall the notation, terminology, and some basic facts for the reader's convenience. In the following
we fix a category $\fC$.

\subsection{Projective Limits.}\label{ss:pllcs}
A  \emph{preorder} on a nonempty set $\Lambda$ is a relation, usually denoted by $\leq$, 
on $\Lambda$ that is reflexive and transitive. A \emph{directed set} (sometimes called \emph{upward directed set})
is a nonempty set $\Lambda$ 
together with a preorder $\leq$ on it with the property that for any $\lambda,\mu\in\Lambda$ there exists 
$\epsilon\in\Lambda$ such that $\lambda\leq \epsilon$ and $\mu\leq\epsilon$. If $\Lambda$ is a directed set,
a family $(O_\lambda)_{\lambda\in\Lambda}$ will be called a \emph{net} indexed on $\Lambda$.

A \emph{projective system} in the category $\fC$ is a pair 
$((\cV_\alpha)_{\alpha\in A};(\phi_{\alpha,\beta})_{\alpha\leq\beta})$
subject to the following properties.
\begin{itemize}
\item[(ps1)] $(A;\leq)$ is a directed set.
\item[(ps2)] $(\cV_\alpha)_{\alpha\in A}$ is a net of objects in the category $\fC$.
\item[(ps3)] $\{\phi_{\alpha,\beta}\mid \phi_{\alpha,\beta}\colon 
\cV_\beta\ra \cV_\alpha,\ \alpha,\beta\in A,\ \alpha\leq\beta\}$ is a family of morphisms in the category $\fC$ 
such that $\phi_{\alpha,\alpha}$ is the identity morphism on $\cV_\alpha$, for all $\alpha\in A$.
\item[(ps4)] The following {transitivity} condition holds
\begin{equation}\phi_{\alpha,\gamma}
=\phi_{\alpha,\beta}\circ\phi_{\beta,\gamma},\mbox{ for all } \alpha,\beta,\gamma\in A,
\mbox{ such that } \alpha\leq\beta\leq\gamma.
\end{equation}
\end{itemize}

Given a projective system $((\cV_\alpha)_{\alpha\in A};(\phi_{\alpha,\beta})_{\alpha\leq\beta})$ in the category 
$\fC$, a \emph{projective limit} associated to it is a pair $(\cV,(\phi_\alpha)_{\alpha\in A}$, where $\cV$ is an 
object of $\fC$ and $\phi_\alpha\colon \cV\ra\cV_\alpha$, for all $\alpha\in A$, are morphisms in the category 
$\fC$, subject to the following conditions.

\begin{itemize}
\item[(pl1)] For any $\alpha,\beta\in A$ such that $\alpha\leq \beta$ we have 
$\phi_\alpha=\phi_{\alpha,\beta}\circ \phi_\beta$. (compatibility)
\item[(pl2)] For any object $\cU$ of $\fC$ and any net $(\psi_\alpha)_{\alpha\in A}$, $\psi_\alpha\colon \cU\ra
\cV_\alpha$, of morphisms in $\fC$, such that the compatibility condition
\begin{equation*}
\psi_\alpha=\phi_{\alpha,\beta}\circ \psi_\beta,\mbox{ for all }\alpha\leq\beta,
\end{equation*} holds,
there exists a unique morphism $\psi \colon \cU\ra \cV$ such that
$\phi_\alpha\circ \psi=\psi_\alpha$, for all $\alpha\in A$. (universality)
\end{itemize}
The projective limit may or may not exist in the given category but, due to universality, if exists it is unique up to 
an isomorphism in $\fC$.

For the special category of locally convex spaces, 
the projective limit of a projective system always exists and, if all the morphisms are  Hausdorff locally convex 
spaces  then the projective limit is always
Hausdorff and, if all locally convex spaces of the projective system are complete, then the projective limit 
is complete, e.g.\ see \cite{Taylor}.

\subsection{Inductive Limits.}\label{ss:illcs}
An \emph{inductive system} in the category $\fC$ is a pair 
$((\cX_\alpha)_{\alpha\in A};(\chi_{\beta,\alpha})_{\alpha\leq\beta})$ subject to the following conditions.
\begin{itemize}
\item[(is1)] $(A;\leq)$ is a directed set. 
\item[(is2)] $(\cX_\alpha)_{\alpha\in A}$ is a net of objects of $\fC$.
\item[(is3)] $\{\chi_{\beta,\alpha}\colon\cX_\alpha\ra\cX_\beta\mid \alpha,\beta\in A,\ \alpha\leq\beta\}$ is a family of morphisms such that $\chi_{\alpha,\alpha}$ is the identity morphism on $\cX_\alpha$, for all $\alpha\in A$.
\item[(is4)] The following transitivity condition holds.
\begin{equation}\chi_{\delta,\alpha}=\chi_{\delta,\beta}\circ\chi_{\beta,\alpha},\mbox{ for all } \alpha,
\beta,\gamma\in A\mbox{ with }
\alpha\leq\beta\leq\delta.
\end{equation}
\end{itemize}

Given an inductive system $((\cX_\alpha)_{\alpha\in A};(\chi_{\beta,\alpha})_{\alpha\leq\beta})$ in the category 
$\fC$, an \emph{inductive limit} in the category $\fC$ is a pair $(\cX,(\chi_\alpha)_{\alpha\in A})$, where $\cX$
is an object in the category $\fC$ and $\chi_\alpha\colon \cX_\alpha\ra \cX$ are morphisms in the category 
$\fC$, subject to the following conditions.

\begin{itemize} 
\item[(il1)] For all $\alpha,\beta\in A$ such that $\alpha\leq \beta$ we have  
$\chi_\beta\circ \chi_{\beta,\alpha}=\chi_\alpha$. (compatibility)
\item[(il2)] For any object $\cY\in\fC$ and any net $\kappa_\alpha\colon \cX_\alpha\ra \cY$,  $\alpha\in A$, 
of morphisms in $\cF$ such that the following compatibility condition
\begin{equation*}
\kappa_\beta\circ \chi_{\beta,\alpha}=\kappa_\alpha,\mbox{ for all }\alpha\leq\beta,
\end{equation*} holds,
there exists a unique morphism $\kappa\colon \cX\ra\cY$ such that $\kappa\circ \chi_\alpha=\kappa_\alpha$ 
for all $\alpha\in A$. (universality)
\end{itemize}

An inductive limit associated to an inductive system in a given category $\fC$ may or may not exist but, due to
the universality property, if it exists it is unique up to an isomorphism in the category $\fC$. However,
there are some important differences between inductive limits and projective limits. For example, a case that 
interests us in this article, in the category of locally convex spaces, although the inductive limit exists,
under the assumption that all 
locally convex spaces $\cX_\alpha$, $\alpha\in A$ are Hausdorff, 
the inductive limit topology may
not be Hausdorff, unless a certain subspace $\cX_0$ is closed in $\bigoplus_{\alpha\in A}\cX_\beta$.
Also, in general, the inductive limit of an inductive system of complete locally convex spaces
may not be complete. See \cite{Taylor} for the construction and \cite{Komura} for some anomalies.

\subsection{Coherent Maps between Inductive Limits.}\label{ss:cm}
We consider the category of locally convex spaces.
Let $(\cX;(\chi_\alpha)_{\alpha\in A})$, $\cX=\varinjlim_{\alpha\in A}\cX_\alpha$, and 
$(\cY;(\kappa_\alpha)_{\alpha\in A})$, $\cY=\varinjlim_{\alpha\in A}\cY_\alpha$, 
be two inductive limits of locally convex spaces indexed by the same directed set $A$. In addition, we 
assume that all canonical maps $\chi_\alpha$ and $\kappa_\alpha$ are injective.
 
A linear map $g\colon \cX\ra\cY$ is called
\emph{coherent} if
\begin{itemize}
\item[(cim)] There exists $(g_\alpha)_{\alpha\in A}$ a net of linear maps 
$g_\alpha\colon\cX_\alpha\ra\cY_\alpha$, $\alpha\in A$, such that 
$g\circ\chi_\alpha= \kappa_\alpha\circ g_\alpha$ for all $\alpha\in A$.
\end{itemize}
In terms of the underlying inductive systems 
$((\cX_\alpha)_{\alpha\in A};(\chi_{\beta,\alpha})_{\alpha\leq\beta})$ and 
$((\cY_\alpha)_{\alpha\in A};(\kappa_{\beta,\alpha})_{\alpha\leq\beta})$, (cim) is equivalent with
\begin{itemize}
\item[(cim)$^\prime$] There exists $(g_\alpha)_{\alpha\in A}$ a net of linear maps 
$g_\alpha\colon\cX_\alpha\ra\cY_\alpha$, $\alpha\in A$, such that $\kappa_{\beta,\alpha}\circ
g_\alpha=g_\beta\circ \chi_{\beta,\alpha}$, for all $\alpha,\beta\in A$ with $\alpha\leq\beta$.
\end{itemize}
There is a one-to-one correspondence between the class of all coherent linear maps 
$g\colon\cX\ra\cY$ and the class of all nets $(g_\alpha)_{\alpha\in A}$ as in (cim) or, 
equivalently, as in (cim)$^\prime$.

Let $\Lcoh(\cX,\cY)$ denote the set of all coherent linear maps $g\colon\cX\ra\cY$. It is easy to see
that $\Lcoh(\cX,\cY)$ is a linear space: if $f,g\in\Lcoh(\cX,\cY)$ corresponds by (cim) to the nets 
$(f_\alpha)_{\alpha\in A}$ and, respectively, $(g_\alpha)_{\alpha\in A}$, and $s,t\in\CC$, then 
$sf+tg$ is a coherent linear map corresponding to the net $(sf_\alpha+tg_\alpha)_{\alpha\in A}$.  
For each $\beta\in A$ let $\pi_\beta\colon\Lcoh(\cX,\cY)\ra \cL(\cX_\beta,\cY_\beta)$ be defined 
by $\pi_\beta(f)=f_\beta$, where $f\in\Lcoh(\cX,\cY)$ corresponds by (cim) to the net 
$(f_\alpha)_{\alpha\in A}$. Since $\pi_\beta$ is linear, its range $\pi_\beta(\Lcoh(\cX,\cY))$ is 
a linear subspace of $\cL(\cX,\cY)$.  
Then $\Lcoh(\cX,\cY)$ is naturally identified with the projective limit 
$\varprojlim_{\alpha\in A}\pi_\alpha(\Lcoh(\cX,\cY))$, where the canonical projections 
are the linear maps $\pi_\beta\colon\Lcoh(\cX,\cY)\ra\pi_\beta(\Lcoh(\cX,\cY))$. For this reason,
if $f\in\Lcoh(\cX,\cY)$ corresponds by (cim) to the net $(f_\alpha)_{\alpha\in A}$, we use the 
notation
\begin{equation}\label{e:cohproj} f=\varprojlim_{\alpha\in A}f_\alpha.
\end{equation}

\subsection{Coherent Maps between Projective Limits}\label{ss:cmpl}
We consider again the category of locally convex spaces.
Let $(\cV;(\phi_\alpha)_{\alpha\in A})$, $\cV=\varprojlim_{\alpha\in A}\cV_\alpha$, and 
$(\cW;(\psi_\alpha)_{\alpha\in A})$, $\cW=\varprojlim_{\alpha\in A}\cW_\alpha$, 
be two projective limits of locally convex spaces indexed by the same directed set $A$. 
A linear map $f\colon \cV\ra\cW$ is called \emph{coherent} if
\begin{itemize}
\item[(cpm)] There exists $(f_\alpha)_{\alpha\in A}$ a net of linear maps $f_\alpha\colon\cV_\alpha\ra\cW_\alpha$, $\alpha\in A$, such that $\psi_\alpha\circ f=f_\alpha\circ \phi_\alpha$ for all
$\alpha\in A$.
\end{itemize}
In terms of the underlying projective systems 
$((\cV_\alpha)_{\alpha\in A};(\phi_{\alpha,\beta})_{\alpha\leq\beta})$ and 
$((\cW_\alpha)_{\alpha\in A};(\psi_{\alpha,\beta})_{\alpha\leq\beta})$, (cpm) is equivalent with
\begin{itemize}
\item[(cpm)$^\prime$] There exists $(f_\alpha)_{\alpha\in A}$ a net of linear maps 
$f_\alpha\colon\cV_\alpha\ra\cW_\alpha$, $\alpha\in A$, such that $\psi_{\alpha,\beta}\circ
f_\beta=f_\alpha\circ \phi_{\alpha,\beta}$, for all $\alpha,\beta\in A$ with $\alpha\leq\beta$.
\end{itemize}
There is a one-to-one correspondence between the class of all coherent linear maps 
$f\colon\cV\ra\cW$ and the class of all nets $(f_\alpha)_{\alpha\in A}$ as in (cpm) or, 
equivalently, as in (cpm)$^\prime$.

Let $\cL_\mathrm{coh}(\cV,\cW)$ denote the set of all coherent linear maps $f\colon \cV\ra \cW$. Clearly,
$\cL_\mathrm{coh}(\cV,\cW)$ is a linear space and, for any $f\in\cL_\mathrm{coh}(\cV,\cW)$, if $(f_\alpha)_{\alpha\in A}$ is associated to $f$ as in (cpm), we denote this by $f=\varprojlim_{\alpha\in A}f_\alpha$.
For any $f=\varprojlim_{\alpha\in A}f_\alpha\in\cL_{\mathrm{coh}}(\cV,\cW)$, we have that $f$ is
continuous with respect to the projective limit topologies on $\cV$ and respectively $\cW$,
if and only if $f_\alpha\colon\cV_\alpha\ra\cW_\alpha$ is continuous for all $\alpha\in A$.

\section{Locally Hilbert Spaces and Their Linear Operators}\label{s:lhs}

Locally Hilbert spaces have implicitly shown up in \cite{Inoue} and, from 
a different perspective, they were called quantised domains in Hilbert spaces in \cite{Dosi1}. 

\subsection{General Properties}\label{ss:gp}
A \emph{locally Hilbert space} is an \emph{inductive limit}
\begin{equation}\label{e:injlim} 
\cH=\varinjlim\limits_{\lambda\in\Lambda}\cH_\lambda=\bigcup_{\lambda\in\Lambda}
\cH_\lambda,
\end{equation} of a 
\emph{strictly inductive system of 
Hilbert spaces} $(\cH_\lambda)_{\lambda\in
\Lambda}$, that is,
\begin{itemize}
\item[(lhs1)] $(\Lambda,\leq)$ is a directed set; 
\item[(lhs2)] $(\cH_\lambda, \langle\cdot,\cdot\rangle_{\cH_\lambda})_{\lambda\in\Lambda}$ 
is a net of Hilbert spaces;
\item[(lhs3)] for each $\lambda,\mu\in\Lambda$ with $\lambda\leq \mu$  we have 
$\cH_\lambda\subseteq\cH_\mu$; 
\item[(lhs4)] for each $\lambda,\mu\in\Lambda$ with $\lambda\leq\mu$ the inclusion map 
$J_{\mu,\lambda}\colon \cH_\lambda\ra\cH_\mu$ is isometric, that is, 
\begin{equation}\langle x,y\rangle_{\cH_\lambda}=\langle x,y\rangle_{\cH_\mu},\mbox{ for all }
x,y\in\cH_\lambda.\end{equation}
\end{itemize} 
For each $\lambda\in\Lambda$, letting $J_\lambda\colon\cH_\lambda\ra\cH$ be the inclusion
of $\cH_\lambda$ in $\bigcup\limits_{\lambda\in\Lambda}\cH_\lambda$, the \emph{inductive limit
topology} on $\cH$ is the strongest which makes the linear maps $J_\lambda$ 
continuous for all $\lambda\in\Lambda$. 

On $\cH$ a canonical inner product $\langle\cdot,\cdot\rangle_\cH$ can be defined as follows: 
\begin{equation}\label{e:lip}
\langle h,k\rangle_\cH=\langle h,k\rangle_{\cH_\lambda},\quad h,k\in\cH,
\end{equation} where $\lambda\in\Lambda$ is any index for which $h,k\in\cH_\lambda$. 
With notation as before, it
follows that the definition of the inner product as in \eqref{e:lip}
is correct and, for each $\lambda\in\Lambda$, 
the inclusion map $J_\lambda\colon(\cH_\lambda,\langle\cdot,\cdot\rangle_{\cH_\lambda})\ra
(\cH,\langle\cdot,\cdot\rangle_{\cH})$ is isometric. This implies that, letting $\|\cdot\|_\cH$ 
denote the norm induced by the inner product $\langle\cdot,\cdot\rangle_\cH$ on $\cH$, the
\emph{norm topology} on $\cH$ is weaker than the inductive limit topology of $\cH$. 
Since the norm topology is Hausdorff, it follows that the inductive limit topology on $\cH$ 
is Hausdorff as well. In the following we let $\widetilde\cH$ denote the completion of the
inner product space $(\cH,\langle\cdot,\cdot\rangle)$ to a Hilbert space. 

\begin{remark}\label{r:lch} Let $\fH$ be the category of Hilbert spaces with morphisms isometric operators.
Given a strictly inductive system of Hilbert spaces $(\cH_\lambda)_{\lambda\in
\Lambda}$, the Hilbert space $\widetilde\cH$, which is the Hilbert space completion of $\cH$ defined at 
\eqref{e:injlim} endowed with the inner product defined at \eqref{e:lip}, 
is its injective limit in the category $\fH$. More precisely, 
let $j_\lambda\colon \cH_\lambda\hookrightarrow\widetilde\cH$, for $\lambda\in\Lambda$, 
be the embedding operator. The following statements hold.
\begin{itemize}
\item[(h1)] $\widetilde\cH\in\fH$.
\item[(h2)] For any $\nu\geq \lambda$ we have $j_\nu\circ J_{\nu,\lambda}=j_\lambda$.
\item[(h3)] For any Hilbert space $\cG$ and  any net $(\phi_\lambda)_{\lambda\in\Lambda}$, where 
$\phi_\lambda\colon \cH_\lambda\ra \cG$ is an isometry, for all $\lambda\in \Lambda$, there exists a unique 
isometry $\phi\colon \widetilde\cH\ra \cG$ such that $\phi\circ j_\lambda=\phi_\lambda$, for all 
$\lambda\in\Lambda$.
\end{itemize}
We make this observation in order to point out that the locally Hilbert space of a strictly inductive system of 
Hilbert spaces is the inductive limit in the larger category of locally convex spaces and not in the category of 
Hilbert spaces.
\end{remark}

In addition, on the locally Hilbert space $\cH$ 
we consider the \emph{weak topology} as well, that is, the locally convex 
topology induced by the family of seminorms $\cH\ni h\ra |\langle h,k\rangle|$, indexed by 
$k\in\cH$.

\begin{remark}\label{r:wp}
Clearly, the weak topology on any locally Hilbert space is Hausdorff separated as 
well. On the other hand, there 
is a weak topology on the Hilbert space $\widetilde\cH$, determined by all linear functionals 
$\widetilde\cH\ni h\mapsto \langle h,k\rangle$, for $k\in\widetilde\cH$, 
and this induces a topology on 
$\cH$, determined by all linear functionals $\cH\ni h\mapsto \langle h,k\rangle$, 
for $k\in\widetilde\cH$, different than the weak topology on $\cH$; in general, 
the weak topology of $\cH$ is weaker 
than the topology induced by the weak topology of $\widetilde\cH$ on $\cH$.
\end{remark}

For an arbitrary nonempty subset $\cS$ of a locally Hilbert space $\cH$ we denote, as usual, the 
\emph{orthogonal companion} of 
$\cS$ by $\cS^\perp=\{k\in\cH\mid \langle h,k\rangle=0\mbox{ for all }h\in\cH\}$. 
If $\cL$ is a subspace of $\cH$ and we denote
by $\overline{\cL}^\mathrm{w}$ its weak closure, then
$\cL^\perp$ is weakly closed and $\cL^\perp={\overline{\cL}^\mathrm{w}}^\perp$, cf.\ Lemma~2.2
in \cite{Gheondea1}.
In particular, a subspace $\cL$ of the  inner product space $\cH$ is weakly closed if and only if
$\cL=\cL^{\perp\perp}$, cf.\ Proposition~2.3 in \cite{Gheondea1}.

For two linear subspaces $\cS$ and $\cL$ of $\cH$, which are mutually orthogonal, denoted 
$\cS\perp \cL$,
we denote by $\cS\oplus\cL$ their algebraic sum. Also, a linear operator $T\colon \cH\ra\cH$
is called \emph{projection} if $T^2=T$ and \emph{Hermitian} if 
$\langle Th,k\rangle=\langle h,Tk\rangle$ for all $h,k\in\cH$. It is easy to see that any Hermitian 
projection $T$ is \emph{positive} in the sense $\langle Th,h\rangle\geq 0$ for all $h\in\cH$ 
and that $T$ is a Hermitian projection if and only if $I-T$ is the same.

For an arbitrary subspace $\cS$ of the locally Hilbert space $\cH$, the \emph{weak topology of} 
$\cS$ is defined as the locally convex topology generated by the family of seminorms $\cS\ni h
\mapsto |\langle h,k\rangle|$, indexed on $k\in\cS$.  The weak topology of $\cS$ is
weaker than the topology on $\cS$ induced by the weak topology of $\cH$.

On the other hand, a linear functional $\phi\colon 
\cS\ra\CC$ is continuous with respect to the weak topology of $\cS$ if and only if there exists a 
unique $k_\phi\in\cS$ such that $\phi(h)=\langle h,k_\phi\rangle$ for all $h\in\cS$, cf.\ 
Proposition~2.1 in \cite{Gheondea1}.

The next proposition provides equivalent characterisations for orthocomplementarity of subspaces
of a locally Hilbert space and it is actually more general, in the context of inner product spaces,
cf.\ Proposition~2.4 in \cite{Gheondea1}.
However, we state it as it is, since this is the only case when we use it.

\begin{proposition}\label{p:orthocomplement}
Let $\cS$ be a linear subspace of $\cH$. The following assertions are equivalent.
\begin{itemize}
\item[(i)] The weak topology of $\cS$ coincides with the topology induced on $\cS$
by the weak topology of $\cH$, in particular $\cS$ is weakly closed in $\cH$.
\item[(ii)] For each $h\in\cH$ the functional $\cS\ni y\mapsto \langle y,h\rangle$ is continuous with 
respect to the weak topology of $\cS$.
\item[(iii)] $\cH=\cS\oplus\cS^\perp$.
\item[(iv)] There exists a Hermitian projection $P\colon\cH\ra\cH$ such that $\ran(P)=\cS$.
\end{itemize}
\end{proposition}

For a given locally Hilbert space $\cH=\varinjlim_{\lambda\in\Lambda}\cH_\lambda$, it is important
to understand the geometry of the components $\cH_\lambda$ and their orthogonal 
complements $\cH_\lambda^\perp$ within $\cH$. For the proof of the next lemma we refer
to Lemma~3.1 in \cite{Gheondea1}.

\begin{lemma}\label{l:perp}
For each $\lambda\in\Lambda$ we have $\cH=\cH_\lambda\oplus \cH_\lambda^\perp$, in particular there exists a unique Hermitian projection $P_\lambda\colon\cH\ra\cH$ such that $\ran(P_\lambda)=\cH_\lambda$.
\end{lemma}

With respect to the decomposition provided by Lemma~\ref{l:perp}, the underlying locally Hilbert 
space structure of $\cH_\lambda^\perp$ can be explicitly described. In the following proposition and throughout 
this article, for each $\lambda\in\Lambda$ we use the notation
\begin{equation}\label{e:lal}
\Lambda_\lambda:=\{\alpha\in\Lambda \mid \alpha\leq \lambda\}.
\end{equation}
For the proof of the next
proposition we refer to Proposition~3.2 in \cite{Gheondea1}.

\begin{proposition}\label{l:lhorth} Let $\cH=\varinjlim\limits_{\lambda\in\Lambda}\cH_\lambda$ and,
for a fixed but
arbitrary $\lambda\in\Lambda$, let us 
denote by $\Lambda^\lambda=\{\mu\in\Lambda\mid \lambda\leq \mu\}$ the 
branch of $\Lambda$ defined by $\lambda$. Then,
with respect to the induced order relation $\leq$, $\Lambda^\lambda$ is a directed set, 
$(\cH_{\mu}\ominus\cH_\lambda)_{ \mu\in\Lambda_\lambda}$ 
is a strictly inductive system of Hilbert spaces and, modulo a canonical identification
of $\varinjlim\limits_{\mu\in\Lambda_\lambda} (\cH_\mu\ominus\cH_\lambda)$ with a subspace
of $\cH$, we have
\begin{equation}\label{e:ilperp} \cH_\lambda^\perp=\varinjlim_{\mu\in\Lambda^\lambda} 
(\cH_\mu\ominus\cH_\lambda).
\end{equation}
\end{proposition}

\subsection{Locally Bounded Operators.}\label{ss:lbo}
With notation as in Subsection~\ref{ss:gp}, let $\cH=\varinjlim_{\lambda\in A}\cH_\lambda$ and 
$\cK=\varinjlim_{\lambda\in A}\cK_\lambda$ be two locally Hilbert spaces generated by strictly
inductive systems of Hilbert spaces 
$((\cH_\lambda)_{\lambda\in\Lambda},(J_{\nu,\lambda}^\cH)_{\lambda\leq\nu})$ 
and, respectively, 
$((\cK_\lambda)_{\lambda\in\Lambda},(J_{\nu,\lambda}^\cK)_{\lambda\leq\nu})$,
indexed on the same directed set $\Lambda$. 
 A linear map 
$T\colon \cH\ra\cK$ is called a \emph{locally bounded operator} if it is a continuous double
coherent linear map, in the sense defined in Subsection~\ref{ss:cm}, 
more precisely,
\begin{itemize}
\item[(lbo1)] There exists a net of operators $(T_\lambda)_{\lambda\in\Lambda}$, 
with $T_\lambda\in\cB(\cH_\lambda,\cK_\lambda)$ such that 
$TJ^\cH_\lambda=J_\lambda^\cK T_\lambda$ for all $\lambda\in\Lambda$. 
\item[(lbo2)] The net of operators $(T_\lambda^*)_{\lambda\in\Lambda}$ 
is coherent as well, that is, $T_\nu^* J_{\nu,\lambda}^\cK=J_{\nu,\lambda}^\cH T_\lambda^*$,
for all $\lambda,\nu\in \Lambda$ such that $\lambda\leq\nu$.
\end{itemize}
We denote by $\Lloc(\cH,\cK)$ the collection of all locally bounded operators 
$T\colon\cH\ra\cK$. We observe that $\Lloc(\cH,\cK)$ is a vector space and that
there is a canonical embedding
\begin{equation}\label{e:lbl} 
\Lloc(\cH,\cK)\subseteq\varprojlim_{\lambda\in\Lambda}\cB(\cH_\lambda,\cK_\lambda).
\end{equation}

The correspondence between $T\in\Lloc(\cH,\cK)$ and the net of operators 
$(T_\lambda)_{\lambda\in\Lambda}$ as in (lbo1) and (lbo2) is unique. Given $T\in\Lloc(\cH,\cK)$,
for arbitrary $\lambda\in\Lambda$ we have $T_\lambda h=Th$, for all $h\in\cH_\lambda$, with the
observation that $Th\in\cK_\lambda$. Conversely, if $(T_\lambda)_{\lambda\in\Lambda}$ is a 
net of operators $T_\lambda\in\cB(\cH_\lambda,\cK_\lambda)$ satisfying (lbo2), then letting
$Th=T_\lambda h$ for arbitrary $h\in\cH$, where $\lambda\in\Lambda$ is such that 
$h\in\cH_\lambda$, it follows that $T$ is a locally bounded operator: 
this definition is correct by (lbo2). In view of \eqref{e:cohproj},
we will use the notation 
\begin{equation}\label{e:tpl} 
T=\varprojlim\limits_{\lambda\in\Lambda}T_\lambda.\end{equation}

On the other hand, in order to perform operator theory with locally bounded operators, 
we will need to assemble a net
of bounded operators acting between component spaces into a locally bounded 
operator acting between the corresponding locally Hilbert spaces. 
The following result tells us which additional properties this net of bounded operators 
must have in order to produce a locally bounded operator, see Proposition~1.12 in \cite{Gheondea3}.

\begin{proposition}\label{p:lbo2} Let
$(T_\lambda)_{\lambda\in\Lambda}$ be a net with $T_\lambda\in\cB(\cH_\lambda,\cK_\lambda)$ 
for all $\lambda\in\Lambda$. The following assertions are equivalent.

\nr{1} For every $\lambda,\nu\in\Lambda$ such that $\lambda\leq\nu$ we have:
\begin{equation}\label{e:temuha}
T_\nu|_{\cH_\lambda}=J_{\nu,\lambda}^\cK T_\lambda,\mbox{ and }
T_\nu P_{\lambda,\nu}^\cH=P_{\lambda,\nu}^\cK T_\nu,\end{equation} 
where $P^\cH_{\lambda,\nu}$ is the orthogonal projection of $\cH_\nu$ onto its subspace 
$\cH_\lambda$.

\nr{2} For every $\lambda,\nu\in\Lambda$ such that $\lambda\leq\nu$, with respect to the
decompositions
\begin{equation*}\cH_\nu=\cH_\lambda\oplus(\cH_\nu\ominus\cH_\lambda),\quad
\cK_\nu=\cK_\lambda\oplus(\cK_\nu\ominus\cK_\lambda),
\end{equation*} the operator $T_\nu$ has the following block matrix representation
\begin{equation}\label{e:temule} 
T_\nu=\left[\begin{matrix} T_\lambda & 0 \\ 0 & T_{\lambda,\nu}\end{matrix}\right],
\end{equation}
for some bounded linear operator $T_{\lambda,\nu}\colon\cH_\nu\ominus\cH_\lambda\ra
\cK_\nu\ominus\cK_\lambda$.

\nr{3} There exists an operator $T\in\Lloc(\cH,\cK)$ such that 
$T|_{\cH_\lambda}=J^\cK_\lambda T_\lambda$ for all $\lambda\in\Lambda$.

In addition, if any of these assertions holds (hence all of them hold), 
the operator $T\in\Lloc(\cH,\cK)$ as in \emph{(3)} is uniquely determined by 
$(T_\lambda)_{\lambda\in\Lambda}$ and
\begin{equation*} T=\varprojlim_{\lambda\in\Lambda}T_\lambda.
\end{equation*}
\end{proposition}

As a consequence of the previous proposition, we can introduce the adjoint operation on 
$\Lloc(\cH,\cK)$. Let $T=\varprojlim_{\lambda\in\Lambda}T_\lambda\in\Lloc(\cH,\cK)$ and
hence the net $(T_\lambda)_{\lambda\in\Lambda}$ satisfies the conditions \eqref{e:temuha}. Then,
consider the net of bounded operators $(T_\lambda^*)_{\lambda\in\Lambda}$, 
$T_\lambda^*\in\cB(\cH_\lambda,\cK_\lambda)$ for all $\lambda\in\Lambda$ and observe that,
for all $\lambda,\nu\in\Lambda$ with $\nu\geq \lambda$, we have
\begin{equation}
T_\nu^*|_{\cK_\lambda}=J_{\nu,\lambda}^\cH T_\lambda^*,\mbox{ and }
T_\nu^* P_{\lambda,\nu}^\cK=P_{\lambda,\nu}^\cH T_\nu^*,\end{equation}
hence, by Proposition~\ref{p:lbo2} there exists a unique operator $T^*\in\Lloc(\cK,\cH)$ such that
\begin{equation}\label{e:temuhadj}
T^*=\varprojlim_{\lambda\in\Lambda}T_\lambda^*.
\end{equation}

Given three locally Hilbert spaces $\cH=\varinjlim_{\lambda\in\Lambda}\cH_\lambda$,
$\cK=\varinjlim_{\lambda\in\Lambda}\cK_\lambda$, and 
$\cG=\varinjlim_{\lambda\in\Lambda}\cG_\lambda$, indexed on the same directed set $\Lambda$, we
observe that the composition of locally bounded operators yields locally bounded operators,
more precisely, whenever $T\in\Lloc(\cH,\cK)$ and $S\in\Lloc(\cK,\cG)$ it follows that 
$ST\in\Lloc(\cH,\cG)$ and usual algebraic properties as associativity and distributivity with
respect to addition and multiplication hold. Moreover, for each $\lambda\in\Lambda$ 
we have
\begin{equation*} (ST)_\lambda =S_\lambda T_\lambda
\end{equation*}
and hence
\begin{equation}\label{e:setep} ST=\varprojlim_{\lambda\in\Lambda} S_\lambda T_\lambda.
\end{equation}
In addition, composition of locally bounded
operators behaves as usual with respect to the adjoint operation, that is,
\begin{equation}\label{e:dadj}
(ST)^*=T^*S^*,\quad T\in\Lloc(\cH,\cK),\ S\in\Lloc(\cK,\cG).
\end{equation} 

Once we have the concept of adjoint of a locally bounded operator 
$T=\varprojlim_{\lambda\in\Lambda}T_\lambda\in\Lloc(\cH,\cK)$ and the concept of composition of locally 
bounded operators, it is natural to call it \emph{locally isometric} if 
$T^*T=I_{\cH}$. It is easy to see that $T$ is locally isometric if and only if 
$T_\lambda\in\cB(\cH_\lambda,\cK_\lambda)$ is isometric for all $\lambda\in \Lambda$. 
Also, $T$ is called \emph{locally unitary} if both $T$ and $T^*$ are locally isometric and then, it is clear that $T$ 
is locally unitary if and only if $T_\lambda\in\cB(\cH_\lambda,\cK_\lambda)$ is unitary for all 
$\lambda\in\Lambda$.

With respect to the pre-Hilbert spaces $\cH$ and $\cK$, a locally bounded operator $T\in \Lloc(\cH,\cK)$
is not, in general, a bounded operator. However, when considering $\cH$  and $\cK$ as dense linear 
subspaces of its Hilbert space completions $\widetilde\cH$ and, respectively, $\widetilde\cK$, it follows that
both $T$ and its adjoint $T^*$ are densely defined hence, they are closable and, letting $\widetilde T$ denote
the closure of $T$ then the closure of $T^*$ is exactly $\widetilde {T^*}$, the closure of $T^*$. So, when
dealing with locally bounded operators we deal with a collection of closed and densely defined operators
which have a common core for them, and a common core for all their adjoint operators.

In the following we recall, see Proposition~1.14 in \cite{Gheondea3}, an operator theoretic characterisation 
of locally bounded operators. Let $T\colon\cH\ra\cK$ be a linear operator and, for arbitrary 
$\lambda\in\Lambda$, by 
Lemma~\ref{l:perp}, with respect to the decompositions,
\begin{equation}\label{e:decomp} \cH=\cH_\lambda\oplus\cH_\lambda^\perp,\quad 
\cK=\cK_\lambda\oplus\cK_\lambda^\perp,
\end{equation} $T$ has the following block matrix representation
\begin{equation}\label{e:mat} T=\left[\begin{matrix} T_{11} & T_{12} \\ T_{21} & T_{22}\end{matrix}\right].
\end{equation}

\begin{proposition}\label{p:mat} A linear operator $T\colon\cH\ra\cK$ is locally bounded if and only 
if, for every $\lambda\in\Lambda$, its matrix representation \eqref{e:mat} is diagonal, i.e.\ 
$T_{12}=0$ and $T_{21}=0$, and $T_{11}$ is bounded.

In addition, if $T=\varprojlim_{\nu\in\Lambda}
{T_\nu}\in\Lloc(\cH,\cK)$ then, for every $\lambda\in\Lambda$, with 
respect to the block matrix representation \eqref{e:mat}, we have $T_{11}=T_\lambda$ and, with
respect to the locally Hilbert spaces 
$\cH_\lambda^\perp$ and $\cK_\lambda^\perp$ as in \eqref{e:ilperp}, 
$T_{22}\in\Lloc(\cH_\lambda^\perp,\cK_\lambda^\perp)$.
\end{proposition}

Also, as a consequence of coherence, see Subsection~\ref{ss:cm}, 
any locally bounded operator $T\colon\cH\ra\cK$ 
is continuous with respect to the inductive limit topologies of $\cH$ and $\cK$. However, 
in general, a locally bounded operator $T\colon\cH\ra\cK$
may not be continuous with respect to the norm topologies of $\cH$ and $\cK$. An arbitrary
linear operator $T\in\Lloc(\cH,\cK)$ is continuous with respect to the norm topologies of $\cH$ 
and $\cK$ if and only if, with respect to the notation as in (lbo1) and (lbo2), 
$\sup_{\lambda\in\Lambda}\|T_\lambda\|_{\cB(\cH_\lambda,\cK_\lambda)}<\infty$. In this case,
the operator $T$ uniquely extends to an operator $\widetilde T\in\cB(\widetilde\cH,\widetilde\cK)$, 
where
$\widetilde\cH$ and $\widetilde\cK$ are the Hilbert space completions of $\cH$ and $\cK$, respectively, 
and $\|\widetilde T\|
=\sup_{\lambda\in\Lambda}\|T_\lambda\|_{\cB(\cH_\lambda,\cK_\lambda)}$.

As a consequence of \eqref{e:lbl}, $\Lloc(\cH,\cK)$ has a natural locally convex topology
induced by the projective limit locally convex topology of 
$\varprojlim_{\lambda\in\Lambda}\cB(\cH_\lambda,\cK_\lambda)$, more precisely, 
generated by the seminorms $\{q_\lambda\}_{\lambda\in\Lambda}$ defined by
\begin{equation}\label{e:qmt}
q_\nu(T)=\|T_\nu\|_{\cB(\cH_\nu,\cK_\nu)},\quad T
=(T_\lambda)_{\lambda\in\Lambda}\in\varprojlim_{\lambda\in\Lambda}
\cB(\cH_\lambda,\cK_\lambda).
\end{equation}
With respect to the embedding \eqref{e:lbl}, $\Lloc(\cH,\cK)$ is closed
in $\varprojlim_{\lambda\in\Lambda}\cB(\cH_\lambda,\cK_\lambda)$, hence complete.
               
The locally convex space $\Lloc(\cH,\cK)$ can be organised as a projective limit of locally 
convex spaces, in view of \eqref{e:lbl}, more precisely, letting 
$\pi_\nu\colon \varprojlim_{\lambda\in\Lambda} \cB(\cH_\lambda,\cK_\lambda)\ra 
\cB(\cH_\nu,\cK_\nu)$, for $\nu\in\Lambda$, 
be the canonical projection, then
\begin{equation}\label{e:pll}
\Lloc(\cH,\cK)=\varprojlim_{\lambda\in\Lambda} \pi_\lambda (\Lloc(\cH,\cK)).
\end{equation}

\subsection{Locally $C^*$-Algebras.}\label{ss:lcsa}
A $*$-algebra $\cA$ is called a \emph{locally $C^*$-algebra}
if it has a complete Hausdorff locally convex topology which is induced by a family of 
$C^*$-seminorms, that is,
seminorms $p$ with the property $p(a^*a)=p(a)^2$ for all $a\in\cA$, see \cite{Inoue}. 
Any $C^*$-seminorm $p$ has also the properties $p(a^*)=p(a)$ and $p(ab)\leq p(a)p(b)$ 
for all $a,b\in\cA$, cf.\ \cite{Sebestyen}. 
Locally $C^*$-algebras have been also called \emph{$LMC^*$-algebras} 
\cite{Schmudgen}, \emph{$b^*$-algebras} \cite{Allan}, \emph{pro $C^*$-algebras} 
\cite{Voiculescu}, \cite{Phillips}, and \emph{multinormed $C^*$-algebras} \cite{Dosi1}. 

If $\cA$ is a locally $C^*$-algebra, let $S(\cA)$ denote the collection of all continuous 
$C^*$-seminorms and note that $S(\cA)$ is a directed set with respect to the partial order defined as
follows:
$p\leq q$ if $p(a)\leq q(a)$ for all $a\in\cA$. If $p\in S(\cA)$ then 
\begin{equation}\label{e:cisp} 
\cI_p=\{a\in\cA\mid p(a)=0\}
\end{equation} is a
closed two sided $*$-ideal of $\cA$ and $\cA_p=\cA/\cI_p$ becomes a $*$-algebra with respect 
to the $C^*$-norm $\|\cdot\|_p$ induced by $p$, more precisely,
\begin{equation}\label{e:tlp} 
\|a+\cI_p\|_p=p(a),\quad a\in\cA.
\end{equation} 
It is proven by C.~Apostol in \cite{Apostol} that the $C^*$-norm  is complete and hence $\cA_p$ is a $C^*$-
algebra already and no completion is needed, as assumed in \cite{Schmudgen} and \cite{Phillips}.
Letting $\pi_p\colon \cA\ra\cA_p$ denote the 
canonical projection, for any $p,q\in S(\cA)$ such that $p\leq q$ there exists a canonical 
$*$-epimorphism of $C^*$-algebras $\pi_{p,q}\colon \cA_q\ra \cA_p$ such that 
$\pi_p=\pi_{p,q}\circ \pi_q$, with respect to which 
$(\cA_p)_{p\in S(\cA)}$ becomes a projective system of $C^*$-algebras such that 
\begin{equation}\label{e:alim}
\cA=\varprojlim\limits_{p\in S(\cA)} \cA_p,\end{equation}
see \cite{Schmudgen}, \cite{Phillips}. This projective limit is taken
in the category of locally convex $*$-algebras and hence all the morphisms are continuous
$*$-morphisms of locally convex $*$-algebras.

\begin{remark}\label{r:bac}
Letting $b(\cA)=\{a\in\cA\mid \sup_{p\in S(\cA)} p(a)<+\infty\}$, it follows that 
$\|a\|=\sup_{p\in S(\cA)} p(a)$ is a $C^*$-norm on the $*$-algebra $b(\cA)$ and, with respect to
this norm, $b(\cA)$ is a $C^*$-algebra, dense in $\cA$, see \cite{Apostol}. The elements
of $b(\cA)$ are called \emph{bounded}.

The $C^*$-algebra $b(\cA)$, together with the system of $*$-morphisms $\{\pi_p\}_{p\in S(\cA)}$, where, with an 
abuse of notation, we denote by the same symbol $\pi_p\colon b(\cA)\ra \cA_p$ 
the canonical projection  $\pi_p\colon \cA\ra\cA_p=\cA/\cI_p$ restricted to $b(\cA)$,
is the projective limit of the projective system of 
$C^*$-algebras $(\cA_p,\pi_p)_{p\in S(\cA)}$ in the category $\fC^*$ of all $C^*$-algebras. Indeed, one can 
show that the following statements hold.
\begin{itemize}
\item[(cs1)] $b(\cA)\in\fC^*$.
\item[(cs2)] For any $p,q\in S(\cA)$, with $p\leq q$ we have $\pi_p=\pi_{p,q}\circ \pi_q$.
\item[(cs3)] For any $C^*$-algebra $\cB$ and $\{\phi_p\}_{p\in S(\cA)}$, where $\phi_p\colon \cB\ra \cA$ are $*$-
morphisms for all $p\in S(\cA)$, there exists a unique $*$-morphism $\phi\colon \cB\ra b(\cA)$ such that
$\pi_p\circ \phi=\phi_p$ for all $p\in S(\cA)$.
\end{itemize}
\end{remark}

\subsection{The Locally $C^*$-Algebra $\Lloc(\cH)$.} \label{ss:lcsab}
Let $\cH=\varinjlim\limits_{\lambda\in\Lambda}\cH_\lambda$
be a locally Hilbert space and $\Lloc(\cH)$ be the locally convex space of all locally 
bounded operators $T\colon \cH\ra\cH$,
as in Subsection~\ref{ss:lbo}.
$\Lloc(\cH)$ has a natural product and a natural involution $*$, 
with respect to which it is a $*$-algebra, see \eqref{e:temuhadj} and \eqref{e:setep}.
For each $\mu\in\Lambda$, consider the $C^*$-algebra 
$\cB(\cH_\mu)$ of all bounded linear operators in $\cH_\mu$ and 
$\pi_\mu\colon\Lloc(\cH)\ra\cB(\cH_\mu)$ be the canonical map: 
\begin{equation}\label{e:pimu}
\pi_\mu(T)=T_\mu,\quad T=\varprojlim\limits_{\lambda\in\Lambda}T_\lambda\in\Lloc(\cH).
\end{equation} Let $\Lloc(\cH_\mu)$ denote the range of $\pi_\mu$ and note that it is 
a $C^*$-subalgebra of $\cB(\cH_\mu)$.
It follows that $\pi_\mu\colon\Lloc(\cH)\ra\Lloc(\cH_\mu)$ is a $*$-morphism of $*$-algebras 
and, for each 
$\lambda,\mu\in \Lambda$ with $\lambda\leq\mu$, 
there is a unique $*$-epimorphism of $C^*$-algebras 
$\pi_{\lambda,\mu}\colon \Lloc(\cH_\mu)\ra\Lloc(\cH_\lambda)$, such that 
$\pi_\lambda=\pi_{\lambda,\mu}\pi_\mu$. More precisely,
$\pi_{\lambda,\mu}$ is the compression of $\cH_\mu$ to $\cH_\lambda$,
\begin{equation}\label{e:plmh} 
\pi_{\lambda,\mu}(S)=J_{\mu,\lambda}^*SJ_{\mu,\lambda},\quad S\in\Lloc(\cH_\mu).
\end{equation}
Then $((\Lloc(\cH_\lambda))_{\lambda\in\Lambda},
(\pi_{\lambda,\mu})_{\lambda,\mu\in\Lambda,\, \lambda\leq \mu})$ is
a projective system of $C^*$-algebras, in the sense that,
\begin{equation}\label{e:ple} 
\pi_{\lambda,\eta}=\pi_{\lambda,\mu}\circ\pi_{\mu,\eta},\quad \lambda,\mu,\eta\in\Lambda,\
\lambda\leq\mu\leq\eta,
\end{equation} and, in addition,
\begin{equation}\label{e:coh}
\pi_\mu(S)P_{\lambda,\mu}=P_{\lambda,\mu}\pi_\mu(S),\quad \lambda,\mu\in\Lambda,\ 
\lambda\leq\mu,\ S\in\Lloc(\cH_\mu),
\end{equation}
such that
\begin{equation}\label{e:loclhs} \Lloc(\cH)=\varprojlim_{\lambda\in\Lambda} \Lloc(\cH_\lambda),
\end{equation}
where, the projective limit is considered in the category of locally convex $*$-algebras. 
In particular, $\Lloc(\cH)$ is a locally $C^*$-algebra.

For each $\lambda\in \Lambda$, letting $p_\lambda\colon\Lloc(\cH)\ra\RR$ be defined by 
\begin{equation}\label{e:psmu}
 p_\lambda(T)=\|T_\lambda\|_{\cB(\cH_\lambda)},\quad T=\varprojlim_{\nu\in\Lambda}
T_\nu\in\Lloc(\cH),
\end{equation} then $p_\lambda$ is a $C^*$-seminorm on $\Lloc(\cH)$. Then $\Lloc(\cH)$ 
becomes a unital locally $C^*$-algebra with the topology induced by  
$\{p_\lambda\}_{\lambda\in\Lambda}$. 

The $C^*$-algebra $b(\Lloc(\cH))$, see Remark~\ref{r:bac}, 
coincides with the set of all locally bounded operators 
$T=\varprojlim_{\lambda\in\Lambda}T_\lambda$ such that $(T_\lambda)_{\lambda\in\Lambda}$
is uniformly bounded, i.e.\ $\sup_{\lambda\in\Lambda}\|T_\lambda\|<\infty$, 
equivalently, those locally bounded operators $T\colon \cH\ra\cH$ which are bounded with respect
to the canonical norm $\|\cdot\|_\cH$ on the pre-Hilbert space 
$(\cH,\langle\cdot,\cdot\rangle_\cH)$. In particular $b(\cA)$ is a $C^*$-subalgebra of 
$\cB(\widetilde\cH)$, where $\widetilde\cH$ denotes the completion of 
$(\cH,\langle\cdot,\cdot\rangle_\cH)$ to a Hilbert space.

By taking the closure, any locally bounded operator $T\in\Lloc(\cH)$ uniquely extends
to a closed and densely defined operator $\widetilde T$ on the Hilbert space $\widetilde\cH$.

\section{Representing Locally Hilbert Spaces}\label{s:rlhs}

\subsection{Why Representing Locally Hilbert Spaces?}\label{ss:wrlhs} The original motivation for considering
locally Hilbert spaces comes from A.~Inoue's construction in the generalisation of the Gelfand-Naimark 
Theorem for locally $C^*$-algebras, see \cite{Inoue}. 
A first observation is that, in that construction, the locally Hilbert space has one 
additional property which was not included in the definition but which becomes necessary 
in order to obtain functional models.

\begin{definition}\label{d:rep}
A locally Hilbert space $\cH=\varinjlim_{\lambda\in\Lambda}\cH_\lambda $ is called \emph{representing}
if, letting $P_\lambda\colon\cH\ra\cH$ denote the Hermitian projection onto $\cH_\lambda$, for arbitrary 
$\lambda\in\Lambda$, as in Lemma~\ref{l:perp}, we have
\begin{itemize}
\item[(lch5)] $P_\lambda P_\nu=P_\nu P_\lambda$ for all $\lambda,\nu\in\Lambda$.
\end{itemize}
In this case, the strictly inductive system of Hilbert spaces $(\cH_\lambda)_{\lambda\in\Lambda}$ is called
\emph{representing} as well.
\end{definition}

Representing locally Hilbert spaces have shown up already for different purposes 
in \cite{Dosi2} under the name of commutative quantised domains of Hilbert spaces. 
Our motivation is strongly related to spectral theory. We firstly relate this property with the 
orthogonal projections between the components Hilbert spaces.

\begin{lemma}\label{l:repof} A locally Hilbert space is representing if and only if, for any 
$\lambda,\nu\in\Lambda$ and any $\epsilon\in\Lambda$ such that $\lambda\leq\epsilon$ and 
$\nu\leq\epsilon$ (since $\Lambda$ is directed, we always can find such an $\epsilon$), letting 
$P_{\lambda,\epsilon},P_{\nu,\epsilon}\in\cB(\cH_\epsilon)$ 
be the orthogonal projections onto $\cH_\lambda$ and $\cH_\nu$, respectively,  we have
\begin{equation}\label{e:pale}
P_{\lambda,\epsilon}P_{\nu,\epsilon}=P_{\nu,\epsilon}P_{\lambda,\epsilon}.
\end{equation}
\end{lemma}

\begin{proof} We first assume that the locally Hilbert space $\cH$ is representing and
fix $\lambda,\nu\in\Lambda$ and $\epsilon\in\Lambda$ such that $\lambda\leq\epsilon$ and $\nu\leq\epsilon$.
Then, for any $h\in\cH_\epsilon$ we have $P_\lambda h=P_{\lambda,\epsilon}h$ and 
$P_\nu h=P_{\nu,\epsilon}h$ and hence,
\begin{equation*}
P_{\lambda,\epsilon}P_{\nu,\epsilon}h=P_{\lambda,\epsilon}P_\nu h=P_\lambda P_\nu h=
P_\nu P_\lambda h=P_{\nu,\epsilon}P_\lambda h
=P_{\nu,\epsilon}P_{\lambda,\epsilon} h,
\end{equation*}
where we have taken into account that $P_\lambda h$ and $P_\nu h$ belong to $\cH_\epsilon$.

Conversely, assume that for any $\lambda,\nu\in\Lambda$ and any $\epsilon\in\Lambda$ such that 
$\lambda,\nu\leq\epsilon$ \eqref{e:pale} holds. Fix $\lambda,\nu\in\Lambda$.
For any $h\in \cH$ there exists $\epsilon\in\Lambda$ such that 
$h\in\cH_\epsilon$. Since $\Lambda$ is directed, from (lhs3) and (lhs4), without loss of generality 
we can assume that $\lambda,\nu\leq\epsilon$. Then both $P_\lambda h$ and 
$P_\nu h$ belong to $\cH_\epsilon$ and hence
\begin{equation*}
P_\lambda P_\nu h=P_\lambda P_{\nu,\epsilon}h=P_{\lambda,\epsilon}P_{\nu,\epsilon}h=
P_{\nu,\epsilon}P_{\lambda,\epsilon}h=P_{\nu,\epsilon}P_\lambda h=P_\nu P_\nu h.
\end{equation*}
Since this holds for any $h\in\cH$, it follows that $P_\lambda P_\nu=P_\nu P_\lambda$, hence $\cH$
is representing.
\end{proof}

\begin{example}\label{e:rlhs} Here we show that 
the original construction of Inoue, in the proof of the generalisation of the 
Gelfand--Naimark's Theorem for locally $C^*$-algebras, cf.\ Theorem~5.1 in \cite{Inoue}, 
leads to a representing locally Hilbert space. Let $\Lambda$ be a directed
set and let $(\cG_\lambda)_{\lambda\in\Lambda}$ be a net of Hilbert spaces. 
For each $\lambda\in\Lambda$ consider the set $\Lambda_\lambda$ defined as in \eqref{e:lal}
 and the Hilbert space $\cH_\lambda$
\begin{equation}\label{e:cal}
\cH_\lambda:=\bigoplus_{\alpha\in\Lambda_\lambda} \cG_\alpha.
\end{equation}

It is easy to see that $(\cH_\lambda)_{\lambda\in\Lambda}$ is a strictly inductive system of Hilbert spaces
and hence it gives rise to a locally Hilbert space $\cH=\varinjlim_{\lambda\in\Lambda}\cH_\lambda$. 
In the following we show that $\cH$ is representing. To this end, we use Lemma~\ref{l:repof}, so
let $\lambda,\nu,\epsilon\in\Lambda$ be
such that $\lambda\leq\epsilon$ and $\nu\leq\epsilon$. Then, with notation as in \eqref{e:lal},
$\Lambda_\lambda\subseteq\Lambda_\epsilon$ and $\Lambda_\nu\subseteq\Lambda_\epsilon$, hence
we have the following decomposition in mutually disjoint subsets
\begin{equation}\label{e:leca}
\Lambda_\epsilon=(\Lambda_\lambda\setminus\Lambda_\nu)\cup (\Lambda_\lambda\cap\Lambda_\nu)
\cup (\Lambda_\nu\setminus\Lambda_\lambda)\cup 
(\Lambda_\epsilon \setminus(\Lambda_\lambda\cup \Lambda_\nu)),
\end{equation}
which yields the following decomposition in terms of mutually orthogonal subspaces
\begin{equation}\label{e:lecah}
\cH_\epsilon=\bigoplus_{\alpha\in\Lambda_\lambda\setminus\Lambda_\nu}\cG_\alpha \oplus
\bigoplus_{\alpha\in\Lambda_\lambda\cap\Lambda_\nu}\cG_\alpha
 \oplus \bigoplus_{\alpha\in\Lambda_\nu\setminus\Lambda_\lambda}\cG_\alpha
\oplus \bigoplus_{\alpha\in\Lambda_\epsilon\setminus(\Lambda_\lambda\cup\Lambda_\nu)}\cG_\alpha.
\end{equation}
We observe that
\begin{equation*}
\cH_\lambda=\bigoplus_{\alpha\in\Lambda_\lambda\setminus\Lambda_\nu}\cG_\alpha \oplus
\bigoplus_{\alpha\in\Lambda_\lambda\cap\Lambda_\nu}\cG_\alpha,
\end{equation*}
that, similarly,
\begin{equation*}
\cH_\nu=\bigoplus_{\alpha\in\Lambda_\nu\setminus\Lambda_\lambda}\cG_\alpha \oplus
\bigoplus_{\alpha\in\Lambda_\lambda\cap\Lambda_\nu}\cG_\alpha,
\end{equation*}
and that
\begin{equation*}
\cH_\lambda\cap\cH_\nu = \bigoplus_{\alpha\in\Lambda_\lambda\cap\Lambda_\nu}\cG_\alpha.
\end{equation*}
Letting $P_{\lambda,\epsilon}\in\cB(\cH_\epsilon)$ denote the orthogonal projection of $\cH_\epsilon$
onto $\cH_\lambda$, from \eqref{e:lecah} we obtain
\begin{equation*}
P_{\lambda,\epsilon}P_{\nu,\epsilon}=P_{\cH_\lambda\cap\cH_\nu}=P_{\nu,\epsilon}P_{\lambda,\epsilon}.
\end{equation*}
Since this holds for any $\epsilon\in\Lambda$ such that $\lambda\leq \epsilon$ and $\nu\leq\epsilon$,
by Lemma~\ref{l:repof}
it follows that $\cH$ is representing.
\end{example}

\subsection{Strictly Inductive Systems of Measurable Spaces.}\label{ss:sisms}
Let us recall that a measure space $(Y;\Sigma;m)$ is called \emph{decomposable}
if $Y$ has a measurable partition $\{Y_j\}_{j\in\cJ}$ with $m(Y_j)<\infty$ for all $j\in\cJ$ and such that.
\begin{itemize}
\item[(dm1)] If $E\subseteq Y$ has the property that $E\cap Y_j\in\Sigma$ for all $j\in\cJ$ then $E\in\Sigma$.
\item[(dm2)] $m(E)=\sum_{j\in\cJ} m(E\cap Y_j)$, in the sense of summability.
\end{itemize}
Then, $(Y;\Sigma;m)$ is decomposable if and only if $Y$ has a measurable partition 
$\{Y_\alpha\}_{\alpha\in A}$ with $m(Y_\alpha)>0$ for all $\alpha\in A$ and such that:
\begin{itemize}
\item[(dm1)] If $E\subseteq Y$ has the property that $E\cap Y_\alpha\in \Sigma$ for all $\alpha\in A$ then 
$E\in\Sigma$.
\item[(dm2)$^\prime$] If $E\in\Sigma$ and $m(E\cap Y_\alpha)=0$ for all $\alpha\in A$ then $m(E)=0$.
\end{itemize}
Decomposable measure spaces share many properties which $\sigma$-finite measure spaces have, e.g.\
see \cite{Kelley}.
In general, the measure spaces $(X_\lambda,\Omega_\lambda,\mu_\lambda)$ which we will consider in the
following will be at least decomposable.

Given a directed set $(\Lambda;\leq)$, a net $((X_\lambda,\Omega_\lambda))_{\lambda
\in\Lambda}$ is called a \emph{strictly inductive system of measurable spaces} if the following properties hold.
\begin{itemize}
\item[(sim1)] For each $\lambda\in\Lambda$, $(X_\lambda,\Omega_\lambda)$ is a measurable space.
\item[(sim2)] For each $\lambda,\nu\in\Lambda$ with $\lambda\leq\nu$ we have 
$X_\lambda\subseteq X_\nu$ and 
$\Omega_\lambda=\{A\cap X_\lambda\mid A\in\Omega_\nu\}\subseteq\Omega_\nu$.
\end{itemize}
Given a strictly inductive system of measurable spaces 
$((X_\lambda,\Omega_\lambda))_{\lambda
\in\Lambda}$, we denote
\begin{equation}\label{e:lms}
X=\bigcup_{\lambda\in\Lambda} X_\lambda,\quad 
\Omega=\bigcup_{\lambda\in\Lambda}\Omega_\lambda.
\end{equation}
The pair $(X,\Omega)$ is called the \emph{inductive limit} of the strictly inductive system of measurable 
spaces, and we use the notation
\begin{equation}\label{e:ilm}
(X,\Omega)=\varinjlim_{\lambda\in\Lambda}(X_\lambda,\Omega_\lambda).
\end{equation}

In general, the inductive limit $(X,\Omega)$ of a strictly inductive system of measurable spaces is not
a measurable space, since $\Omega$ is not a $\sigma$-algebra. Actually,
$\Omega$ is a \emph{ Boolean ring} of subsets in $X$, that is, 
for any $A,B\in\Omega$ it follows 
$A\setminus B, A\cap B, A\cup B\in\Omega$, but, in general, not even a $\sigma$-ring, that is,
it may not be stable under countable unions or countable intersections. Moreover, $\Omega$ is a 
\emph{locally $\sigma$-ring} in the sense that, if $(A_n)_n$ is a sequence of subsets from $\Omega$ 
such that there exists $\lambda\in\Lambda$ with the property that $A_n\in\Omega_\lambda$ for all
$n\in\NN$, it follows that $\bigcup_{n\in\NN} A_n$ and $\bigcap_{n\in\NN} A_n$ are in $\Omega$.

$\Omega$ has a canonical extension to a $\sigma$-algebra. Letting 
\begin{equation}\label{e:tomega} \widetilde\Omega:=\{A\subseteq X\mid A\cap 
X_\lambda\in\Omega_\lambda\mbox{ for all }\lambda\in\Lambda\},
\end{equation}
then, by Proposition~3.1 in \cite{Gheondea3},
$\widetilde\Omega$ is a $\sigma$-algebra and $\Omega\subseteq\widetilde\Omega$.

\begin{remark} One may think that the pair $(X,\tilde\Omega)$, with definition as in \eqref{e:tomega}, 
should be called the inductive limit of the strictly inductive system of measurable spaces. From the category 
theory point of view, this is true, within the category $\fM$ of measurable spaces, where objects are measurable 
spaces and morphisms are measurable maps. More precisely, we consider 
$(X,\tilde\Omega,(j_\lambda)_{\lambda\in\Lambda})$, 
where $j_\lambda\colon X_\lambda\hookrightarrow X$ is the inclusion map, 
for each $\lambda\in\Lambda$. Also, for each 
$\lambda\leq \nu$, let $j_{\nu,\lambda}\colon X_\lambda\hookrightarrow X_\nu$ be the inclusion map. Then
$(((X_\lambda,\Omega_\lambda))_{\lambda\in\Lambda},(j_{\nu,\lambda})_{\lambda\leq \nu})$ is an inductive 
system and $(X,\tilde\Omega,(j_\lambda)_{\lambda\in\Lambda})$ is the inductive limit. Indeed, the following assertions hold true.
\begin{itemize}
\item[(il1)] $(X,\tilde\Omega)\in \fM$.
\item[(il2)] For any $\nu\geq \lambda$ we have $j_\nu\circ j_{\nu,\lambda}=j_\lambda$. (Compatibility)
\item[(il3)] For any $(Y,\Xi)\in\cM$ and $(\phi_\lambda)_{\lambda\in\Lambda}$, where 
$\phi_\lambda\colon (X_\lambda,\Omega_\lambda)\ra (Y,\Xi)$ is measurable, for each $\lambda\in\Lambda$, 
there exists uniquely a measurable map $\phi\colon (X,\tilde\Omega)\ra (Y,\Xi)$ such that 
$\phi\circ j_\lambda=\phi_\lambda$ for all $\lambda\in \Lambda$. (Universality)
\end{itemize}

However, as we will see later in Subsection~\ref{ss:crlhs}, 
in order to match these with locally Hilbert spaces and 
locally bounded operators, we have to work in a larger category, that of Boolean spaces, where objects are sets 
equipped with Boolean rings. This comes from the fact that locally Hilbert spaces are 
inductive limits in the larger category of locally convex spaces, not in the category of Hilbert spaces, see 
Remark~\ref{r:lch}.
\end{remark}

\subsection{Strictly Inductive Systems of Measure Spaces.}\label{ss:sismes}
Given a directed set $(\Lambda;\leq)$, a net $((X_\lambda,\Omega_\lambda,\mu_\lambda))_{\lambda
\in\Lambda}$ of measure spaces is called a \emph{strictly inductive system of measure spaces} if
$(X_\lambda,\Omega_\lambda)_{\lambda\in\Lambda}$ is a strictly inductive system of measurable spaces 
and, in addition,
\begin{itemize}
\item[(sim3)] For any $\lambda,\nu\in\Lambda$ with $\lambda\leq\nu$ and any $U\in\Omega_\lambda$ we 
have $\mu_\lambda(U)=\mu_\nu(U)$.
\end{itemize}
With notation as in \eqref{e:lms},
let also $\mu\colon\Omega\ra [0,+\infty]$ be defined by
\begin{equation}\label{e:lmsa}
\mu(\Delta)=\mu_\lambda (\Delta),\quad \Delta\in\Omega,
\end{equation}
where $\lambda\in\Lambda$ is such that $\Delta\in\Omega_\lambda$. 

Taking into account the assumptions (sim1)--(sim3), it is easy to see that $\mu$
is correctly defined, that is, its definition does not depend on $\lambda$. In addition, $\mu$ is nondecreasing, 
$\mu(\emptyset)=0$, and \emph{additive} in the sense that whenever $A,B\in\Omega$ are disjoint then 
$\mu(A\cup B)=\mu(A)+\mu(B)$. In general, $\mu$ is not $\sigma$-additive, but it
is \emph{locally $\sigma$-additive}, that is: 
for any mutually disjoint sequence $(A_n)_{n\in\NN}$ such that there exists 
$\lambda\in\Lambda$ with the property that $A_n\in\Omega_\lambda$ for all $n\in\NN$ we 
have
\begin{equation}\label{e:sap}
\mu(\bigcup_{n\in\NN} A_n)=\sum_{n\in\NN}\mu(A_n).
\end{equation}

The triple $(X,\Omega,\mu)$ is called the \emph{inductive limit} of the strictly inductive system of measure
spaces $(X_\lambda,\Omega_\lambda,\mu_\lambda)_{\lambda\in\Lambda}$ and we use the notation
\begin{equation}\label{e:ilsm}
(X,\Omega,\mu)=\varinjlim_{\lambda\in\Lambda} (X_\lambda,\Omega_\lambda,\mu_\lambda).
\end{equation}

\begin{example}\label{ex:locli}
We consider the notation as before. Let $L^\infty_{\mathrm{loc}}(X,\mu)$ denote
the set of all functions $\phi\colon X\ra \CC$, identified $\mu$-a.e., such that, 
for any $\lambda\in\Lambda$ we have $\phi|_{X_\lambda}\in L^\infty(X_\lambda,\mu_\lambda)$. With
natural addition, multiplication, multiplication with scalars, and involution, it is easy to see that
$L^\infty_{\mathrm{loc}}(X,\mu)$ is a commutative $*$-algebra. For each $\lambda\in \Lambda$ we
consider the seminorm
\begin{equation}\label{e:loclisem}
p_\lambda(\phi):= \|\phi|X_\lambda\|_\infty=\esssup_{x\in X_\lambda}|\phi(x)|,\quad 
\phi\in L^\infty_{\mathrm{loc}}(X,\mu).
\end{equation}
With respect to the topology defined by the family of seminorms 
$\{p_\lambda\}_{\lambda\in\Lambda}$, 
$L^\infty_{\mathrm{loc}}(X,\mu)$ is a locally $C^*$-algebra. Actually, we observe that we have 
a projective system of $C^*$-algebras made up by the $C^*$-algebras
$(L^\infty(X_\lambda,\mu_\lambda))_{\lambda\in\Lambda}$ and the projective system of $*$-morphisms
$(\Phi_{\lambda,\nu})_{\lambda\leq \nu}$ where, for $\lambda\leq \nu$ the $*$-morphism $\Phi_{\lambda,\nu}
\colon L^\infty(X_\nu,\mu_\nu)\ra L^\infty(X_\lambda,\mu_\lambda)$ is defined by $\phi\mapsto 
\phi|_{X_\lambda}$.
Then, the projective limit of this projective system of $C^*$-algebras, in the category of locally convex 
$*$-algebras, is canonically identified with $L^\infty_{\mathrm{loc}}(X,\mu)$.
\end{example}

Although an inductive limit of a strictly inductive system of measure spaces is not, in general, a measure 
space, in the following we show that it has a canonical extension to a measure space.

Consider the $\sigma$-algebra $\widetilde\Omega$, defined as in \eqref{e:lms}, and
define $\widetilde\mu\colon\widetilde\Omega\ra [0,+\infty]$ by
\begin{equation}\label{e:wimu}
\widetilde\mu(A):=\sup_{\lambda\in\Lambda} \mu_\lambda(A\cap X_\lambda),\quad 
A\in\widetilde\Omega.
\end{equation}

\begin{remark}\label{r:wimu}
Let $A\in\widetilde\Omega$ be arbitrary. Then, in view of the properties (sim1)--(sim3), the 
net $(\mu_\lambda(A\cap X_\lambda))_{\lambda\in\Lambda}$ is nondecreasing hence
\begin{equation*}
\widetilde\mu(A)=\begin{cases}+\infty, & \mbox{ if }(\mu_\lambda(A\cap X_\lambda))_{\lambda\in\Lambda}\mbox{ is unbounded,} \\
\lim\limits_{\lambda\in\Lambda}\mu_\lambda(A\cap X_\lambda), & \mbox{ if }
(\mu_\lambda(A\cap X_\lambda))_{\lambda\in\Lambda}\mbox{ is bounded}.
\end{cases}
\end{equation*}
\end{remark}

The following result was observed in \cite{KulkarniPamula} as well, 
with a different terminology and a slightly different proof.

\begin{proposition}\label{p:wimu}
$\widetilde\mu$ is a measure which extends $\mu$ defined at \eqref{e:lmsa}.
\end{proposition}

\begin{proof} Let $A\in\Omega$ be arbitrary. Then there exists $\lambda_0\in \Lambda$ such that
$A\in\Omega_{\lambda_0}$. For any $\lambda\in\Lambda$ there exists $\nu\in\Lambda$ such that
$\lambda_0,\lambda\leq\nu$, hence $\mu_{\lambda_0}(A)=\mu_\nu(A)$ and 
\begin{align*} \mu_\lambda(A\cap X_\lambda) & =\mu_\nu(A\cap X_\lambda) 
\leq \mu_\nu(A\cap X_\nu)=\mu_\nu(A)=\mu_{\lambda_0}(A). 
\end{align*}
This shows that 
\begin{equation*}
\widetilde\mu(A)=\sup_{\lambda\in\Lambda} \mu_\lambda(A)=\mu_{\lambda_0}(A)=\mu(A),
\end{equation*}
hence $\widetilde\mu$ extends $\mu$ defined at \eqref{e:lmsa}.

We now show that $\widetilde\mu$ is a measure. First we observe that, for arbitrary 
$A,B\in\widetilde\Omega$ with $A\subseteq B$, by \eqref{e:wimu} we have
\begin{equation}\label{e:temu}
\widetilde\mu(A)\leq \widetilde\mu(B).
\end{equation}
Let now $(A_n)_{n\in\NN}$ be a sequence of subsets in $\widetilde\Omega$ such that 
$A_n\cap A_k=\emptyset$ whenever $n\neq k$. We prove
\begin{equation}\label{e:temucup}
\widetilde\mu(\bigcup_{n\in\NN} A_n)=\sum_{n=1}^\infty \widetilde\mu(A_n).
\end{equation}
Indeed,
\begin{align*}
\widetilde\mu(\bigcup_{n=1}^\infty A_n) 
& =\sup_{\lambda\in\Lambda} \mu_\lambda(\bigcup_{n=1}^\infty
A_n\cap X_\lambda) = \sup_{\lambda\in\Lambda} \mu_\lambda(\bigcup_{n=1}^\infty
(A_n\cap X_\lambda)) \\
\intertext{and taking into account that $\mu_\lambda$ is $\sigma$-additive we get}
& = \sup_{\lambda\in\Lambda} \sum_{n=1}^\infty \mu_\lambda(A_n\cap X_\lambda) \\
& \leq \sum_{n=1}^\infty \sup_{\lambda\in\Lambda} \mu_\lambda(A_n\cap X_\lambda) =
\sum_{n=1}^\infty \widetilde\mu(A_n),
\end{align*}
hence one of the inequalities in \eqref{e:temucup} is proven.

In order to prove the second inequality in \eqref{e:temucup}, 
let us observe that, without any loss of generality, we can
assume $\widetilde\mu(\bigcup_{n=1}^\infty A_n)<\infty$ and hence, by \eqref{e:temu}, it follows
that $\widetilde\mu(A_n)<\infty$ for all $n\in\NN$. Then, for any $\epsilon>0$ and any $k\in\NN$
there exists $\lambda_k\in\Lambda$ such that
\begin{equation}\label{e:emu}
\widetilde\mu(A_k)-\frac{\epsilon}{2^k}<\mu_{\lambda_k}(A_k\cap X_{\lambda_k})\leq \widetilde
\mu(A_k).
\end{equation}
Fix $n\in\NN$. Then there exists $\lambda\in\Lambda$ such that $\lambda_1,\ldots,\lambda_n\leq
\lambda$. As a consequence of the properties (i)--(iii), it follows that
\begin{equation}\label{e:mulak}
\mu_{\lambda_k}(A_k\cap X_{\lambda_k})\leq \mu_\lambda(A_k\cap X_\lambda),\quad 
k=1,\ldots,n.
\end{equation}
Then, taking into account \eqref{e:emu} and twice of \eqref{e:mulak}, we have
\begin{align*}
\sum_{k=1}^n \widetilde\mu(A_k)-\epsilon & < \sum_{k=1}^n 
\bigl(\widetilde\mu(A_k)-\frac{\epsilon}{2^k}\bigr) \\
& < \sum_{k=1}^n \mu_{\lambda_k}(A_k\cap X_{\lambda_k}) \\
& \leq \sum_{k=1}^n \mu_\lambda(A_k\cap X_\lambda) \\
\intertext{which, taking into account the mutual disjointness of the sets $A_k$, $k=1,\ldots,n$,}
& = \mu_\lambda(\bigcup_{k=1}^n A_k\cap X_\lambda) \leq \widetilde\mu(\bigcup_{k=1}^n A_k)
\leq \widetilde\mu(\bigcup_{k=1}^\infty A_k).
\end{align*}
Consequently, we have
\begin{equation*}
\sum_{k=1}^n \widetilde\mu(A_k)< \widetilde\mu(\bigcup_{k=1}^\infty A_k) +\epsilon, \quad n\in\NN,
\end{equation*}
hence, letting $n$ approach $+\infty$ we get
\begin{equation*}
\sum_{k=1}^\infty \widetilde\mu(A_k)\leq \widetilde\mu(\bigcup_{k=1}^\infty A_k) +\epsilon,
\end{equation*}
and then, since $\epsilon>0$ is arbitrary, it follows that
\begin{equation*}
\sum_{k=1}^\infty \widetilde\mu(A_k)\leq\widetilde\mu(\bigcup_{k=1}^\infty A_k).
\end{equation*}
This finishes the proof of \eqref{e:temucup}.
\end{proof}

\subsection{Concrete Representing Locally Hilbert Spaces}\label{ss:crlhs}
With notation and assumptions as in the previous subsection, 
for each $\lambda\in\Lambda$ we consider the Hilbert space $L^2(X_\lambda,\mu_\lambda)$ viewed as a space 
of functions on $X$ in the following way.
\begin{align}
L^2(X_\lambda,\mu_\lambda) & :=\{f\colon X\ra \CC\mid 
\supp(f)\subseteq X_\lambda,\ f|_{X_\lambda}\mbox{ is } \Omega_\lambda\mbox{-measurable, }\nonumber \\
& \phantom{:=\ \ \ \ }
\int_{X_\lambda} |f(x)|^2\de\mu_\lambda(x)<\infty,\mbox{ identified } \mu_\lambda\mbox{-a.e.}\},
\label{e:leltwolambda}
\end{align}
endowed with the quadratic norm 
\begin{equation*} \|f\|_{L^2(X_\lambda,\mu_\lambda)}^2:= \int_{X_\lambda} |f(x)|^2\de\mu_\lambda(x),
\quad f\in L^2(X_\lambda,\mu_\lambda).
\end{equation*}
It is clear from this definition that $f\mapsto f|_{X_\lambda}$ is an isometric identification with the usual 
$L^2(X_\lambda,\mu_\lambda)$ space. This notation is useful for the definition of a concrete locally Hilbert 
space.

\begin{lemma}\label{l:letexem}
$(L^2(X_\lambda,\mu_\lambda))_{\lambda\in\Lambda}$, when these Hilbert spaces are viewed as spaces of 
functions defined on the whole set $X$ as before, is a representing strictly inductive system
of Hilbert spaces. \end{lemma}

\begin{proof} Indeed, let $\lambda,\nu\in\Lambda$ be such that $\lambda\leq \nu$. For any 
$\phi\in L^2(X_\lambda,\mu_\lambda)$ we have $\supp(\phi)\subseteq X_\lambda\subseteq X_\nu$. Also,
let us denote $\phi_\lambda=\phi|_{X_\lambda}$ and $\phi_\nu=\phi|_{X_\nu}$. We show that $\phi_\nu$ is 
$\Omega_\nu$-measurable.
For any Borel subset of $\CC$, in view of the properties by (sim1)--(sim2), we have
\begin{align*}\phi_\nu^{-1}(S) & =\phi_\nu^{-1}((S\setminus\{0\})\cup \{0\}) \\
& =\phi_\nu^{-1}(S\setminus\{0\})\cup \phi_\nu^{-1}(\{0\}) \\
& =\phi_\lambda^{-1}(S\setminus\{0\}) \cup \phi^{-1}_\lambda(\{0\})\cup (X_\nu\setminus X_\lambda)\\
& = \phi_\lambda^{-1}(S)\cup (X_\nu\setminus X_\lambda)\in \Omega_\nu. 
\end{align*}
 
Further, by (sim3) it follows that, when considering the equivalence class of $\phi_\nu$ with respect to
$\mu_\nu$, it follows that
$\phi_\nu\in L^2(X_\nu,\mu_\nu)$ extends $\phi_\lambda$ and
\begin{equation*}
\int_{X_\nu} |\phi_\nu(x)|^2\de\mu_\nu(x)=\int_{X_\lambda}|\phi(x)|^2\de\mu_\lambda(x).
\end{equation*}
In this way, we see that $L^2(X_\lambda,\mu_\lambda)$ is isometrically included in 
$L^2(X_\nu,\mu_\nu)$.

We now prove that the strictly inductive system of Hilbert spaces 
$(L^2(X_\lambda,\mu_\lambda))_{\lambda\in\Lambda}$ is representing. To this end, let 
$\lambda,\nu,\epsilon\in\Lambda$ be such that $\lambda,\nu\leq \epsilon$. Then we have 
$X_\lambda,X_\nu\subseteq X_\epsilon$ and all these are elements in $\Omega_\epsilon$. We have the 
decomposition of $X_\epsilon$ with mutually disjoint measurable sets
\begin{equation*} X_\epsilon = (X_\lambda\setminus (X_\lambda\cap X_\nu))\cup (X_\nu \setminus 
(X_\lambda\cap X_\nu))\cup (X_\lambda\cap X_\nu)\cup (X_\epsilon \setminus (X_\lambda\cup X_\nu)).
\end{equation*}
From here, with the identification of $L^2(X_\lambda,\mu_\lambda)$ with $L^2(X_\lambda,\mu_\lambda)$ 
explained before, we have the following decomposition
\begin{align*}
L^2(X\epsilon,\mu_\epsilon) & = L^2(X_\lambda\setminus (X_\lambda\cap X_\nu),\mu_\lambda)\oplus 
L^2(X_\nu \setminus (X_\lambda\cap X_\nu),\mu_\nu)\oplus \\
& \phantom{=}\,\, \oplus L^2(X_\lambda\cap X_\nu,\mu_\lambda)\oplus L^2(X_\epsilon \setminus (X_\lambda\cup X_\nu),\mu_\epsilon).
\end{align*}
From here, it follows that the orthogonal projection of $L^2(X_\epsilon,\mu_\epsilon)$ onto 
$L^2(X_\lambda,\mu_\lambda)$ and the orthogonal projection of $L^2(X_\epsilon,\mu_\epsilon)$ onto 
$L^2(X_\nu,\mu_\nu)$ can be explicitly written and proven that they commute. In view of Lemma~\ref{l:repof},
we conclude that the strictly inductive system of Hilbert spaces 
$(L^2(X_\lambda,\mu_\lambda))_{\lambda\in\Lambda}$ is representing.
\end{proof}

As a consequence of Lemma~\ref{l:letexem}, it makes perfectly sense that, when the Hilbert spaces 
$L^2(X_\lambda,\mu_\lambda)$ are viewed as spaces of functions defined on the whole set $X$, to define 
\begin{equation}\label{e:eltwoloc}
L^2_{\mathrm{loc}}(X,\mu):=\varinjlim_{\lambda\in\Lambda}L^2(X_\lambda,\mu_\lambda),
\end{equation}
the inductive limit of the strictly inductive system of Hilbert spaces
$(L^2(X_\lambda,\mu_\lambda))_{\lambda\in\Lambda}$, which is a representing locally Hilbert space. 
The representing locally Hilbert space $L^2_{\mathrm{loc}}(X,\mu)$ induced by the
inductive limit of a strictly inductive system of measure spaces 
$(X_\lambda,\Omega_\lambda,\mu_\lambda)_{\lambda\in\Lambda}$
is called \emph{concrete}.
 
 \begin{remark}\label{r:eltwoloc}
We identify the locally
Hilbert space $L^2_{\mathrm{loc}}(X,\mu)$ with the collection of $\mu$-a.e.\ equivalence classes of functions
$\phi\colon X\ra \CC$ such that for some $\lambda\in\Lambda$ we have 
$\phi|_{X_\lambda}\in L^2(X_\lambda,\mu_\lambda)$ and $\phi|_{(X\setminus X_\lambda)}=0$, $\mu$-a.e.
\end{remark}

As pointed out in Subsection~\ref{ss:gp}, each locally Hilbert space $\cH$ is isometrically and densely 
embedded in its completion to a Hilbert denoted by $\widetilde \cH$, where uniqueness is modulo a certain 
unitary operator, which 
is the inductive limit of the underlying strictly inductive system of Hilbert spaces within the category of Hilbert 
spaces with isometric morphisms, see Remark~\ref{r:lch}.
In the following we show that the above abstract completion can be made concrete for the case
of the locally Hilbert space $L^2_{\mathrm{loc}}(X,\mu)$: its completion 
can be realised as the Hilbert space $L^2(X,\widetilde\mu)$, where $\widetilde\mu$ is defined
as in \eqref{e:wimu}.

\begin{proposition}\label{p:eltwo}
The locally Hilbert 
space $L^2_{\mathrm{loc}}(X,\mu)$ is dense in the Hilbert space $L^2(X,\widetilde\mu)$.
\end{proposition}

\begin{proof} Let $\phi\in L^2(X,\widetilde\mu)$ be arbitrary. For any 
$\lambda\in\Lambda$ let $\phi_\lambda=\phi|_{X_\lambda}$. We show that $\phi_\lambda$ is
$\Omega_\lambda$-measurable. Indeed, for any Borel subset $B$ of the complex plain we have
\begin{equation*}
\phi_\lambda^{-1}(B)=\{x\in X_\lambda\mid \phi_\lambda(x)\in B\}
= \{x\in X_\lambda\mid \phi(x)\in B\}=\phi^{-1}(B)\cap X_\lambda,
\end{equation*}
hence, taking into account that $\phi$ is $\widetilde\Omega$-measurable and the definition of 
$\widetilde\Omega$, see \eqref{e:tomega}, it follows that $\phi_\lambda^{-1}(B)\in\Omega_\lambda$.

Let $\widetilde\phi_\lambda\colon X\ra\CC$ be the canonical extension of $\phi_\lambda$ defined by
$\widetilde\phi_\lambda|_{X_\lambda}=\phi_\lambda$ and 
$\widetilde\phi_\lambda|_{(X\setminus X_\lambda)}=0$. Then $\widetilde\phi_\lambda$ is 
$\Omega$-measurable. Since $\widetilde\Omega\supseteq \Omega$ it follows that $\widetilde\phi_\lambda$
is $\widetilde\Omega$-measurable as well.
Then,
\begin{equation}\label{e:intefi}
\int_X |\widetilde\phi_\lambda(x)|^2\de\widetilde\mu(x) = \int_X |\phi(x)|^2 \chi_{X_\lambda}(x)\de
\widetilde\mu(x)\leq \int_X |\phi(x)|^2\de\widetilde\mu(x)<+\infty.
\end{equation}
In view of Remark~\ref{r:eltwoloc}, it
follows that $\widetilde\phi_\lambda\in L^2_{\mathrm{loc}}(X,\mu)$, for all $\lambda\in\Lambda$.

On the other hand, since the net of functions 
$(\chi_{X\setminus X_\lambda})_{\lambda\in\Lambda}$ converges pointwise to $0$ and $\phi$ is
finite $\widetilde\mu$-a.e., in view
of \eqref{e:intefi} and the Dominated Convergence Theorem, it follows that
\begin{equation*}
\|\phi-\widetilde\phi_\lambda\|_2^2 = \int_X |\phi(x)-\widetilde\phi_\lambda(x)|^2\de\widetilde\mu(x)
= \int_X |\phi(x)|^2 \chi_{X\setminus X_\lambda}(x)\de\widetilde\mu(x)
\xrightarrow[\lambda\in\Lambda]{ } 0.
\end{equation*}
Thus, we have proven that the space $L^2_{\mathrm{loc}}(X,\mu)$ is dense in the Hilbert space 
$L^2(X,\widetilde\mu)$.
\end{proof}

\begin{remark} By Proposition~\ref{p:eltwo}, one can view
$(L^2(X,\widetilde\mu),(L^2(X_\lambda,\mu_\lambda))_{\lambda\in\Lambda},L^2_{\mathrm{loc}}(X,\mu))$ as a 
commutative quantised domain in the sense of \cite{Dosi2}. Also, this allows us to view Lemma~\ref{l:letexem}
in a simpler fashion, namely to view the Hilbert spaces $L^2(X_\lambda,\mu_\lambda)$ as subspaces of the 
Hilbert space $L^2(X,\widetilde \mu)$, a fact which was specified in the paragraph before Lemma~\ref{l:letexem} 
in a rather more involved fashion.
\end{remark}

We need a result slightly more general than Lemma~\ref{l:letexem}. The proof follows the same lines and we 
omit it.

\begin{proposition}\label{p:tensorprod} Let $\cG$ be a Hilbert space, 
$(X_\lambda,\Omega_\lambda,\mu_\lambda)_{\lambda\in\Lambda}$ be a strictly inductive system of measure 
spaces and let $(X,\Omega,\mu)$ be its inductive limit. If $\cG$ is a Hilbert space, for each 
$\lambda\in\Lambda$ we denote by $L^2_\cG(X_\lambda,\mu_\lambda)$ the collection of all functions 
$f\colon X\ra \cG$ such that $\supp(f)\subseteq X_\lambda$, $f|_{X_\lambda}$ is measurable with respect to
$(X_\lambda,\Omega_\lambda)$, $\int_{X_\lambda} \|f(x)\|_\cG^2\de\mu_\lambda(x)<\infty$, and identified 
$\mu_\lambda$-a.e.

Then, the net $(L^2_\cG(X_\lambda,\mu_\lambda))_{\lambda\in\Lambda}$ is a representing 
strictly inductive system of Hilbert spaces and, consequently, its inductive limit
\begin{equation}
L^2_{\cG,\mathrm{loc}}(X,\mu):=\varinjlim_{\lambda\in\Lambda} L_\cG^2(X_\lambda,\mu_\lambda),
\end{equation}
exists and it is a representing locally Hilbert space.
\end{proposition}

\begin{remark}\label{r:tensorprod}  With notation as in Proposition~\ref{p:tensorprod}, for each 
$\lambda\in\Lambda$ there is a natural identification of the Hilbert spaces $L^2_\cG(X_\lambda,\mu_\lambda)$ 
and $\cG\otimes L^2(X_\lambda,\mu_\lambda)$. On elementary tensors this identification acts by 
$\cG\otimes L^2(X_\lambda,\mu_\lambda)\ni g\otimes f\mapsto f g\in L^2_\cG(X_\lambda,\mu_\lambda)$.
\end{remark}

\subsection{An Example: The Hata Tree-Like self-similar Set}\label{ss:hata}

If $(M;d)$ is a metric space, a map $f\colon M\ra M$ is called a \emph{contraction} if there exists 
$\alpha \in [0,1)$ such that $d(f(x_1),f(x_2))\leq \alpha\, d(x_1,x_2)$ for all
$x_1,x_2\in M$. The least $\alpha\in [0,1)$ for which this is true is called the \emph{Lipschitz constant} of $f$.
Also, $f$ is called a \emph{similitude} if $d(f(x_1),f(x_2))= \alpha\, d(x_1,x_2)$ for some $\alpha> 0$ and 
all $x_1,x_2\in M$. Note that similitudes are injective functions.
Given $N\in\NN$, a set $\{f_j\mid j=1,\ldots,N\}$ of contractions on the metric space $(M;d)$ 
is called an \emph{iterated function system} on $M$.  
A subset $X\subseteq M$ is called a \emph{self-similar set} with 
respect to the iterated function system $\{f_j\mid j=1,\ldots,N\}$ if
\begin{equation}\label{e:self-similar}
X=\bigcup_{j=1}^N f_j(X).
\end{equation}

\begin{theorem}[J.E.~Hutchinson \cite{Hutchinson}]\label{t:hutchinson}
 Let $(M;d)$ be a complete metric space and let 
$\{f_j\mid j=1,\ldots,N\}$ be an iterated function 
system on $M$. Then there exists a unique nonempty compact subset $X\subseteq M$ 
which is self-similar with respect to this iterated function system.
\end{theorem}

It is useful to recall the idea of the proof, which is constructive to a large extent. Let $\cC(M)$ be the collection
of all nonempty compact subsets of $M$ and endow it with the \emph{Pompeiu-Hausdorff metric}
\begin{equation*}
d_{\mathrm{PH}}(X,Y):=\inf \{r>0\mid X\subseteq \bigcup_{x\in Y} \overline{B}_r(x),\  
Y\subseteq \bigcup_{x\in X} \overline{B}_r(x)\},
\end{equation*}
where $\overline{B}_r(x):=\{y\in M\mid d(x,y)\leq r\}$ is the \emph{closed ball} of center $x$ and radius $r$. It is
proven that, with respect to the metric $d_{\mathrm{PH}}$ the space $\cC(M)$ is complete. Then, since any 
contraction map is uniformly continuous, hence continuous, one can define 
$F\colon \cC(M)\ra \cC(M)$ by
\begin{equation}\label{e:fex}
F(X):=\bigcup_{j=1}^N f_j(X).
\end{equation}
One can prove that $F$ is a contraction map as well, more precisely, if $\alpha_j$ is the Lipschitz constant 
of $f_j$, $j=1,\ldots, N$, then
\begin{equation*}
d_{\mathrm{PH}} (F(X),F(Y))\leq \max\{\alpha_j\mid j=1,\ldots,N\} d_{\mathrm{PH}}(X,Y),\quad X,Y\in\cC(M).
\end{equation*}
Finally, by the Contraction Theorem, $F$ has a unique fixed point $X$, and this is the self-similar set
with respect to the iterated function system $\{f_j\mid j=1,\ldots,N\}$. By the proof of the 
Contraction Theorem, one can start with any nonempty 
compact subset $X_0\subseteq M$, called the \emph{seed}, and then the sequence of compact 
sets $(F^n(X_0))_n$, where $F^n=F\circ \cdots \circ F$, the composition of $n$ copies of $F$, converges to $X$
with respect to the metric $d_{\mathrm{PH}}$.

Consider the functions $f_1,f_2\colon \CC\ra \CC$ defined by
\begin{equation}\label{e:fef}
f_1(z)=c\bar{z},\quad f_2(z)=(1-|c|^2)\bar z+|c|^2,\quad z\in \CC,
\end{equation}
where $c\in \CC$ is given. Note that, for any $z,\zeta\in\CC$ we have
\begin{equation*}
|f_1(z)-f_1(\zeta)|=|c| |z-\zeta|,\quad |f_2(z)-f_2(\zeta)|=|1-|c|^2| |z-\zeta|.
\end{equation*}
This shows that, both $f_1$ and $f_2$ are similitudes and, assuming that $0<|c|,|1-c|<1$, then both $f_1$ and 
$f_2$ are contractions, with 
Lipschitz constants $|c|$ and $|1-|c|^2|$, respectively. In order to avoid 
trivial cases, we assume that $c\not\in\RR$. We will assume these throughout this subsection.  By 
Theorem~\ref{t:hutchinson}, let $X$ denote the self-similar set
with respect to the iterated function system $\{f_1,f_2\}$, with definitions as in \eqref{e:fef}, and which is
called the \emph{Hata tree-like self-similar set}, \cite{Hata}. See Figure~\ref{f:hata14}, in order to have a first idea
of what this self-similar set looks like (the algorithm used to generate it is described below). Let us observe that 
the latter condition, $0<|1-c|<1$, is not necessary in order to guarantee the fact that both $f_1$ and $f_2$ are 
contractions but, without it, the geometry of the self-similar set $X$ would be slightly different.

\hspace{-10pt}\begin{figure}[h] 
\includegraphics[trim= 70 230 70 230,clip]{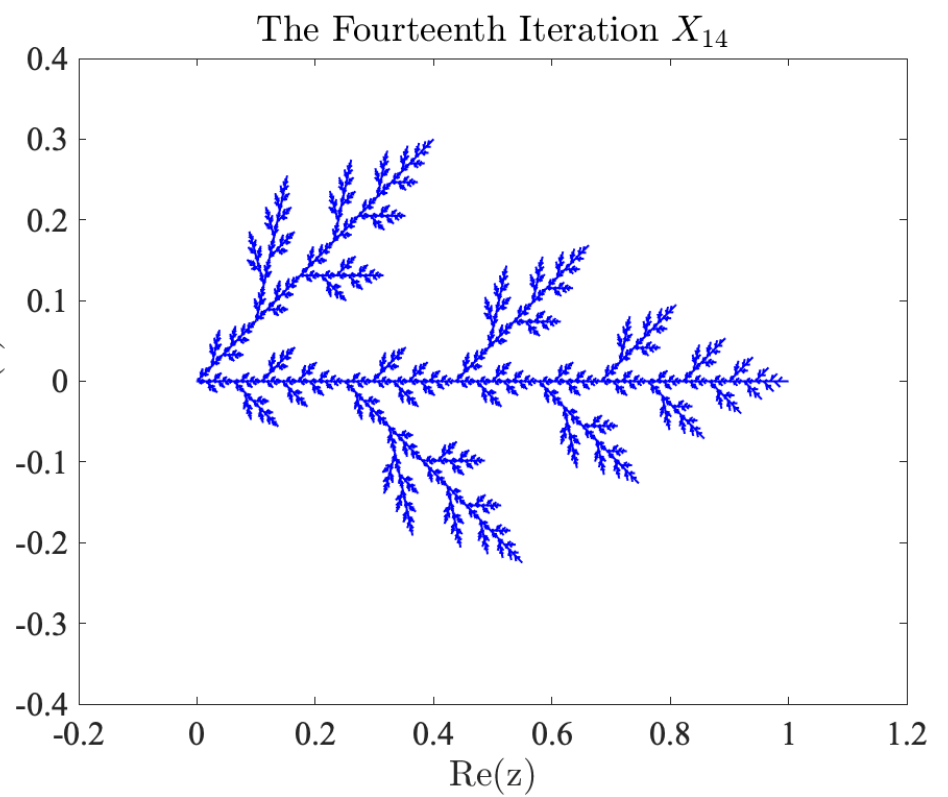}
\caption{The $14$th approximation in generating the Hata tree-like self-similar set, for $c=0.3+0.4\iac$
and the seed $X_0=[0,1]\cup c[0,1]$, provides 
an idea of the tree form of this self-similar set. The picture is obtained with a MATLAB code based on 
Algorithm 1.}
\label{f:hata14}
\end{figure}

Also, in this special case, the function 
$F\colon \cC(\CC)\ra \cC(\CC)$ from \eqref{e:fex} becomes 
\begin{equation}\label{e:fexa}
F(A)=f_1(A)\cup f_2(A),\quad A\in \cC(\CC).
\end{equation}
In the following we briefly describe the generation of the Hata tree-like self-similar set by approximation with an 
ascending sequence of compact sets. As a by-product, this 
will eventually help us to produce images that provide an idea of the set and why 
it is named like this.

Let $X_0=[0,1]\cup c[0,1]$. If $z,\zeta\in\CC$, we denote 
\begin{equation}\label{e:zeze}
[z,\zeta]=\{(1-t)z+t\zeta\mid t\in [0,1]\},\end{equation}
the segment of straight line in the complex plane with endpoints $z$ and $\zeta$. This segment inherits 
the orientation from $[0,1]$, so $z$ is the beginning and $\zeta$ is the end.
We observe that
\begin{equation}\label{e:exone}
f_1(X_0)=c[0,1]\cup [0,|c|^2]\mbox{ and } f_2(X_0)=[|c|^2,1]\cup [|c|^2,(1-|c|^2)\bar c+|c|^2], 
\end{equation}
which shows that, with notation as in \eqref{e:fexa}, we have
\begin{equation}\label{e:xez}
X_0\subseteq [0,1]\cup c[0,1]\cup [|c|^2,(1-|c|^2)\bar c+|c|^2]=
 f_1(X_0)\cup f_2(X_0)=F(X_0)=:X_1.
\end{equation}
The sequence of approximation of the Hata tree-like set $X$, as in the proof of Theorem~\ref{t:hutchinson}, is 
defined recursively by $X_{n+1}:=F(X_n)$, $n\geq 0$, hence, from \eqref{e:xez} we can prove by induction that
\begin{equation}\label{e:xena}
X_n\subseteq F(X_n)=X_{n+1},\quad n\geq 0,
\end{equation}
and then, for each fixed $n\in\NN$, we have
\begin{equation}\label{e:xenasub}
X_n\subseteq X_m,\quad m\geq n.
\end{equation}
For any $\epsilon>0$, since $(X_m)_m$ converges to $X$,
by the definition of the Pompeiu-Hausdorff metric, it follows that there exists $m\geq n$ and 
sufficiently large such that
\begin{equation*}
X_n\subseteq X_m \subseteq \bigcup_{x\in X}\overline{B}_\epsilon(x).
\end{equation*}
Since $\epsilon>0$ is arbitrary and $X$ is compact, this shows that
\begin{equation}\label{e:xenes}
X_n\subseteq X,\quad n\geq 0.
\end{equation}

From here, one can conclude that
\begin{equation}\label{e:oveb}
\overline{\bigcup_{n\geq 0} X_n}=X,
\end{equation}
where we denote by $\overline{E}$ the closure of the subset $E$ in $\CC$. 

We need some terminology from symbolic dynamics, following \cite{Kigami}. 
To keep things to a minimum, we confine our 
considerations to the special case related to the Hata tree-like self-similar set, so we consider the 
\emph{set of two symbols} $S=\{1,2\}$. For each $m\in\NN$ we denote by $W_m$ the set of all 
words of \emph{length} $m$ with symbols in $S$,
\begin{equation*}
W_m:=\{w_1\ldots w_m\mid w_1,\ldots, w_m\in S\},
\end{equation*}
and let $W_0=\{\emptyset\}$, where $\emptyset$ is the \emph{empty word}. Then $W$ is the 
\emph{set of all finite words}
\begin{equation*}
W:=\bigcup_{m\geq 0} W_m.
\end{equation*}
If $w\in W$ then we denote by $|w|$ its \emph{length}. By definition, the empty word has length $0$.

For any finite word $w=w_1w_2\ldots w_m$ we denote 
\begin{equation}\label{e:fewe}
f_w:=f_{w_1}\circ f_{w_2}\circ \cdots \circ f_{w_m},
\end{equation}  
where $f_1$ and $f_2$ are defined at \eqref{e:fef}. With this notation, let us observe that, for each $m\in \NN$ 
we have
\begin{equation}\label{e:xeme}
X_m=\bigcup_{w\in W_m} f_w(X_0),
\end{equation}
and then, from \eqref{e:oveb}, we get
\begin{equation*}
\overline{\bigcup_{w\in W}f_w(X_0)}=X.
\end{equation*}
Also, note that $W_m$ has exactly $2^m$ words and that the size of $W_m$ is doubled at each step.

For further reference, we calculate some of the first compositions of the functions $f_1$ and $f_2$ defined at
\eqref{e:fef}. First, $f_w$ for $|w|=2$.
\begin{align}
f_{11}(z) & = |c|^2 z, & f_{22}(z) & = (1-|c|^2)^2z+|c|^2(2-|c|^2),\label{e:few2}\\
f_{12}(z) & = c(1-|c|^2)z+c|c|^2, & f_{21}(z) & = (1-|c|^2)\bar c z+|c|^2,\nonumber
\end{align}
Secondly, $f_w$ for $|w|=3$.
\begin{align}
f_{111}(z) & = c|c|^2\bar z, & f_{122}(z) & = c(1-|c|^2)^2\bar z+c|c|^2(2-|c|^2),\label{e:few3}\\
f_{112}(z) & = c^2(1-|c|^2)\bar z +c^2|c|^2, & f_{121}(z) & = |c|^2(1-|c|^2)\bar z+c|c|^2,\nonumber\\
f_{222}(z) & = (1-|c|^2)^3\bar z +|c|^2(3-3|c|^2+|c|^4), & f_{211}(z) & = |c|^2(1-|c|^2)\bar z+|c|^2,\nonumber\\
f_{212}(z) & = \bar c(1-|c|^2)^2\bar z+(1-|c|^2)\bar c|c|^2+|c|^2, & f_{221}(z) & = (1-|c|^2)^2c\bar z+|c|^2(2-|c|^2).
\nonumber
\end{align}

Based on the discussion we had so far, we made an algorithm
described in Algorithm~\ref{a:hata}, which we can use in order to generate 
the sequence of approximations $(X_m)_{m\geq 0}$ and plot as many elements of the sequence as we want.
We briefly describe the setting of the algorithm.
Given arbitrary $c\in\CC\setminus\RR$, 
subject to the conditions $0<|c|,|1-c|<1$, let $N$ denote the maximal number of 
iterations, and let $s$ denote the number of samplings on each straight segment. We provide the functions 
$f_1$ and $f_2$ defined at \eqref{e:fef}. The seed $X_0=[0,1]\cup c[0,1]$ is made up by two complex 
vectors, $X01$ containing the sampling of $c[0,1]$ and $X02$ containing the sampling of $[0,1]$. Also, 
$m_1,\ldots,m_k\leq N$ denote  indices of the approximation subsets $X_{m_1},\ldots,X_{m_k}$ which we want 
to be plotted. The algorithm will construct the subsets of each approximating subset $X_m$ as rows of a matrix 
$X$, more precisely, $X_m$ is made by plots of
the rows $X(2^{m+1}-1,:)$ through $X(2^{m+2}-2,:)$. This is formally described in Algorithm~\ref{a:hata}.

\begin{algorithm}\label{a:alg1}
\caption{Generating the approximating sequence of subsets of the Hata tree-like self-similar set and plotting
some of these.}
\begin{algorithmic}[1]
\REQUIRE $c$, $N$, $s$, the functions $f_1$ and $f_2$ defined at \eqref{e:fef}, the seed $X_0$ made up by two complex vectors, $X01$  and $X02$, and 
$m_1,\ldots,m_k\leq N$.
\STATE Initiate by $0$ the matrix $X$ with $2^{N+3}$ rows and $s$ columns. \STATE Initiate the first row $X(1,:)\leftarrow X01$ and the second row $X(2,:)\leftarrow X02$.
\FOR{$m=1$ to $N$}
\FOR{$l=1$ to $2^m$}
\FOR{$k=1$ to $2$}
\IF{$k=1$} 
\STATE $X(l+2^{m+1}-2,:)\leftarrow  f_1(X(l+2^m-2,:))$
\ELSIF{$k=2$} 
\STATE $X(l+2^{m}+2^{m+1}-2,:) \leftarrow  f_2(X(l+2^m-2,:))$
\ENDIF
\ENDFOR
\ENDFOR
\ENDFOR
\RETURN Plot the subsets $X_{m_1},\ldots,X_{m_k}$.
\end{algorithmic}
\label{a:hata}
\end{algorithm}

To have a first idea on how the Hata tree-like self-similar set "grows" by the approximating sequence described 
above, we used this algorithm to 
make a MATLAB code by which we plot the seed $X_0$ and the approximation sets $X_1$ through $X_5$, see 
Figure~\ref{f:hata}. However, we can get a better idea on how the Hata tree-like self-similar set looks like by 
visualising the $14$th approximation set, see Figure \ref{f:hata14}.

Now a few theoretical facts. The following lemma is known but we include it for the reader's convenience.

\begin{lemma}\label{l:hatacon} With notation as before, for each $m\geq 0$ the set $X_m$ is connected and the 
Hata tree-like set $X$ is connected.
\end{lemma}

\begin{proof} Clearly the set $X_0$ is connected. We prove by induction that $X_m$ is connected for all 
$m\geq 0$. For arbitrary $m\in\NN_0$, assume that we know that $X_m$ is connected. Then
\begin{equation*}
X_{m+1}=F(X_m)=f_1(X_m)\cup f_2(X_m).
\end{equation*}
Since $X_m$ is connected 
and both $f_1$ and $f_2$ are continuous, it follows that both $f_1(X_m)$ and $f_2(X_m)$ are 
connected. In order to conclude that $X_{m+1}$ is connected we only have to prove that 
$f_1(X_m)\cap f_2(X_m)\neq \emptyset$. Indeed, using \eqref{e:xenes} and \eqref{e:exone}, we have
\begin{equation*}
f_1(X_m)\cap f_2(X_m)\supseteq f_1(X_0)\cap f_2(X_0)=\{|c|^2\}\neq\emptyset.
\end{equation*}
From here we conclude that $\bigcup_{n\geq 0}X_n$ is a connected set and then $X$ is connected due to
\eqref{e:oveb}.
\end{proof}

Actually, more is true. The Hata tree-like self-similar set is arcwise connected, but this comes from a more 
general result in \cite{Hata}: if $K$ is any self-similar set then $K$ is connected if and only if it is arcwise 
connected, see also Theorem~1.6.2 in \cite{Kigami}.

There are at least two more facts that we have to explain about the Hata tree-like self-similar set. The first fact 
explains the tree structure of the Hata self-similar set. Graphically this can be seen if 
we plot the subsets $X_0$ through $X_n$ for $n$ sufficiently large (see
Figure~\ref{f:hata} which is obtained by an implementation of Algorithm~\ref{a:hata}), but, of course, we need an 
analytic proof of this fact. The second fact that we need is a 
parametrisation of the set $X$ that can be used to describe the topological structure and the measure space 
structure of $X$. These two facts are somehow related, more precisely, the first one will help us getting the latter 
one. In order to prove analytically the tree structure of the Hata set, we first need some definitions.

 Let $S$ be a connected set in $\CC$. A \emph{branch} 
of $S$ is a segment (of a straight line) contained in $S$ such that it is maximal, with respect to set inclusion, 
among all segments contained in $S$. Also, assume that the set $S$ of $\CC$ is obtained as a 
$2$-dimensional directed graph, more precisely, it is obtained from a finite set of points, the nodes, 
and between some pairs of them there are oriented edges that are segments. 
$S$ is called a \emph{binary tree} if: there is one node, the starting node, having no incoming edge and exactly 
two outgoing edges, some nodes, called the ending nodes, that have only one incoming edge and no outgoing 
edges, and all the other nodes have only one incoming edge and exactly two outgoing edges.

\hspace{-10pt}\begin{figure}[h]
\includegraphics[trim= 70 230 70 230,clip]{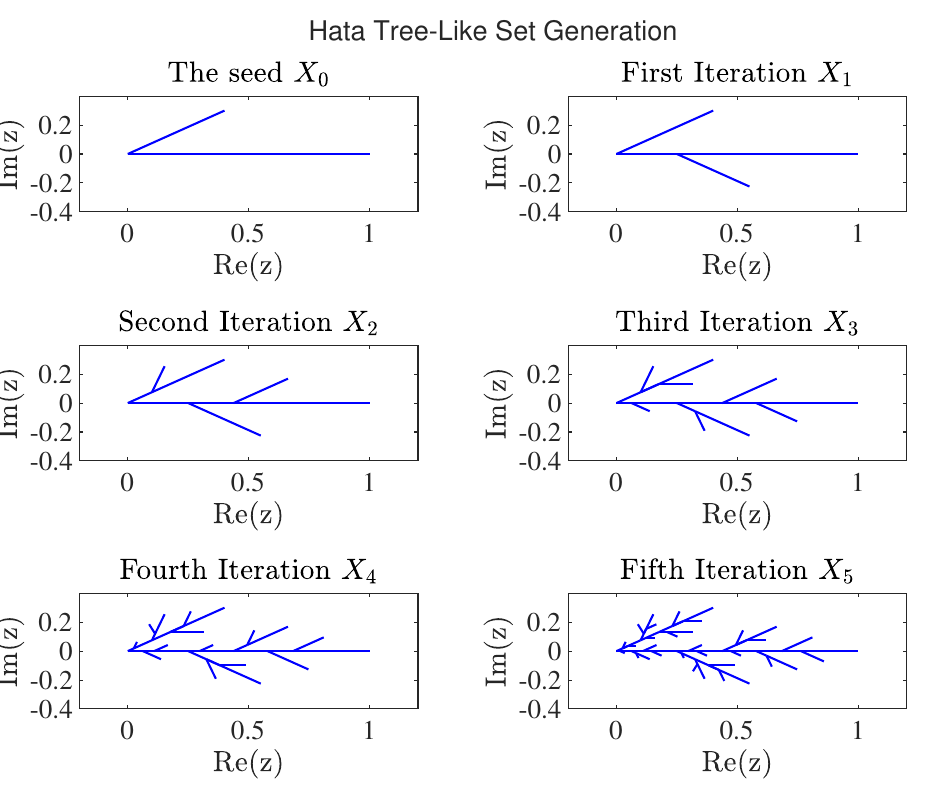}
\caption{The seed $X_0=[0,1]\cup c[0,1]$ 
and the first five approximation steps in generating the Hata tree-like self-similar set,
 for $c=0.3+0.4\iac$. $X_0$ has $2=2^0+1$ branches, $X_1$ has $3=2^1+1$ branches, $X_2$ has $5=2^2+1$ 
 branches, $X_3$ has $9=2^3+1$ branches, $X_4$ has $17=2^4+1$ branches, and $X_5$ has $33=2^5+1$ 
 branches. The pictures are obtained with a MATLAB code based on Algorithm~1.}
\label{f:hata}
\end{figure}

In the next proposition, for arbitrary $n\in\NN_0$, the approximating set $X_n$ is 
considered as an oriented graph, with orientation 
provided by \eqref{e:zeze} and the nodes provided by the intersections of pairs of segments, see 
Figure~\ref{f:hata}.

\begin{proposition}\label{p:hata}
 For each $n\in\NN_0$, the approximating set $X_n$ is a binary tree with the starting node at $0$, it has a 
number of exactly $2^n+1$ ending nodes and, consequently, when 
considered simply a subset of $\CC$, it has a number of exactly $2^n+1$ branches.

In addition, for each $n\in \NN$, the approximating set $X_n$ can be 
parameterised as the union of all its $2^n+1$ branches as follows
\begin{equation}\label{e:xenaz}
X_n=[0,1]\cup c(0,1]\cup \bigcup_{k=0}^{n-1} \bigcup_{w\in W_k} (f_{w2}(0),f_{w2}(c)].
\end{equation}
\end{proposition}

\begin{proof} We prove the statement by induction. We split the set $X_0$ into its two branches $X_{00}=[0,1]$
and $X_{01}=c[0,1]$. For $n=0$ it is clear from the definition that $X_0$ is a binary tree and 
consists of $2=2^0+1$ branches, the segment $[0,1]$ and the segment $c[0,1]$, both of them having in common 
the point $0$. We will examine two more steps in order to get the general idea of proof.

 For $n=1$ we have
\begin{align*}
f_1(X_{00}) & =[0,c], & f_1(X_{01}) & =[0,|c|^2],\\
f_2(X_{00}) & =[|c|^2,1], & f_2(X_{01}) & = [|c|^2,(1-|c|^2)\bar c+|c|^2],
\end{align*} 
and then we see from here that $X_1$ is the union of four segments, three of them make $X_0$, more 
precisely, $f_1(X_{00})=X_{01}$ and $f_1(X_{01})\cup f_2(X_{00})=X_{00}$, while $f_2(X_{01})$ is the new 
branch added to $X_1$ which is not in $X_0$. In addition, the new branch $f_2(X_{01})$ is a subbranch of
the main branch $X_{00}$ which starts at the point $|c|^2$.
So,  $X_1$ has $3=2^1+1$ branches and it is a binary tree. 

For $n=2$, with notation as in \eqref{e:fewe}, from \eqref{e:few2} we have the following segments.
\begin{align}
f_{11}(X_{00}) & = |c|^2[0,1], & f_{11}(X_{01}) & = |c|^2c[0,1],\label{e:femulte2} \\
f_{21}(X_{00}) & = [|c|^2,(1-|c|^2\bar c+|c|^2], &  f_{21}(X_{01}) & =[|c|^2,2|c|^2-|c|^4],\nonumber\\
f_{22}(X_{00}) & = [2|c|^2-|c|^4,1],  &  f_{22}(X_{01}) & = [2|c|^2-|c|^4,(1-|c|^2)^2c +2|c|^2-|c|^4],\nonumber\\
f_{12}(X_{00}) & = c[|c|^2,1], & f_{12}(X_{01}) & = [c|c|^2,(1-|c|^2)c^2+c|c|^2].\nonumber
\end{align}
From here we see that $X_2$ is made of the three branches of $X_1$ as follows: the main branch 
$X_{00}=[0,1]=f_{11}(X_{00})\cup f_{21}(X_{01})\cup f_{12}(X_{00})$, the  second branch 
$X_{01}=c[0,1]=f_{11}(X_{01})\cup f_{12}(X_{00})$, and the third branch $f_2(X_{01})=f_{21}(X_{00})$. 
There are two more branches in $X_2$ which do not belong to $X_1$ and these are: $f_{22}(X_{01})$, which
is a subbranch of the main branch $X_{00}$ and starts at $2|c|^2-|c|^4$, and $f_{12}(X_{01})$ which is a 
subbranch of the branch $X_{01}$ and starts at $c|c|^2$. So, $X_2$ has $5=2^2+1$ branches and 
is a  binary tree.

For $n=3$, with notation as in \eqref{e:fewe}, from \eqref{e:few3} we have the following segments.
\begin{align}
f_{111}(X_{00}) & = [0, c|c|^2],\label{e:femulte30}\\
f_{122}(X_{00}) & = [c|c|^2(2-|c|^2),c(1-|c|^2)^2+c|c|^2(2-|c|^2)],\nonumber\\
f_{112}(X_{00}) & = [c^2|c|^2,c^2(1-|c|^2) +c^2|c|^2],\nonumber\\
f_{121}(X_{00}) & = [c|c|^2,|c|^2(1-|c|^2)+c|c|^2],\nonumber\\
f_{222}(X_{00}) & = [|c|^2(3-3|c|^2+|c|^4),(1-|c|^2)^3 +|c|^2(3-3|c|^2+|c|^4)],\nonumber\\
f_{211}(X_{00}) & = [|c|^2,|c|^2(1-|c|^2)+|c|^2],\nonumber\\
f_{212}(X_{00}) & = [(1-|c|^2)\bar c|c|^2+|c|^2,\bar c(1-|c|^2)^2+(1-|c|^2)\bar c|c|^2+|c|^2],\nonumber\\ 
f_{221}(X_{00}) & = [|c|^2(2-|c|^2),(1-|c|^2)^2c+|c|^2(2-|c|^2)].
\nonumber
\end{align}
and
\begin{align}
f_{111}(X_{01}) & = [0, |c|^4],\label{e:femulte31}\\
f_{122}(X_{01}) & = [c|c|^2(2-|c|^2),|c|^2(1-|c|^2)^2+c|c|^2(2-|c|^2)],\nonumber\\
f_{112}(X_{01}) & = [c^2|c|^2,c|c|^2(1-|c|^2) +c^2|c|^2],\nonumber\\
f_{121}(X_{01}) & = [c|c|^2,|c|^2(1-|c|^2)\bar c+c|c|^2],\nonumber\\
f_{222}(X_{01}) & = [|c|^2(3-3|c|^2+|c|^4),(1-|c|^2)^3\bar c +|c|^2(3-3|c|^2+|c|^4)],\nonumber\\
f_{211}(X_{01}) & = [|c|^2,|c|^2(1-|c|^2)\bar c+|c|^2],\nonumber\\
f_{212}(X_{01}) & = [(1-|c|^2)\bar c|c|^2+|c|^2,(\bar c)^2(1-|c|^2)^2+(1-|c|^2)(\bar c)^2|c|^2+|c|^2],\nonumber\\ 
f_{221}(X_{01}) & = [|c|^2(2-|c|^2),(1-|c|^2)^2|c|^2+|c|^2(2-|c|^2)].
\nonumber
\end{align}
A careful examination of \eqref{e:femulte30} and \eqref{e:femulte31} shows that only the segments 
$f_{w2}(X_{01})$ bring new branches in $X_3$ and these new branches start at points on the branches of 
$X_2$. So, $X_3$ has $9=2^3+1$ branches and is a binary tree.

For the general step of induction, assume that the statement holds for $X_0$ through $X_n$. First, let us 
observe that both $f_1$ and $f_2$ are affine functions in $\bar z$ and, consequently, for any $w\in W$, the 
function $f_w$ is an affine function in either $z$, if $|w|$ is even, or $\bar z$, if $|w|$ is odd.
Since $X_0$ is a connected union of two segments, $X_{00}$ and $X_{01}$, 
for any $w\in W$ the set  $f_w(X_0)$ is a connected 
finite union of at most $2^{|w|+1}$ segments.
Then, 
$X_{n+1}=F(X_n)=f_1(X_n)\cup f_2(X_n)$ is a connected set, by Lemma~\ref{l:hatacon}, a union of at most 
$2^{n+2}$ segments. The set 
$X_{n+1}$ contains $X_n$, see \eqref{e:xena}, which has $2^n+1$ branches. We represent 
$2^n+1=(2^{n-1}+1)+2^{n-1}$ in order to point out that $X_{n-1}$ has $2^{n-1}+1$ branches while 
$X_n$ has $2^{n-1}$ more branches which are not in $X_{n-1}$. Since 
\begin{equation*}X_n\subseteq X_{n+1}=F(X_n)=f_1(X_n)\cup f_2(X_n),\end{equation*} we observe that 
only the $2^{n-1}$ branches which belong to $X_n$ but not to $X_{n-1}$ yield new branches for $X_{n+1}$ and 
these are exactly $2^{n-1}+2^{n-1}=2\cdot 2^{n-1}=2^n$, hence $X_{n+1}$ has exactly $2^n+1+2^n=2^{n+1}+1$ 
branches.

It remains to prove that for each $n\in\NN_0$ the approximating set $X_{n+1}$ is a binary tree, for 
any $n\geq 1$. For 
this, we claim that, within the reasoning from the previous paragraphs, we have
\begin{equation}\label{e:xenab}
X_n=\bigcup_{w\in W_n}f_{w1}(X_{00})\cup \bigcup_{w\in W_n}f_{w1}(X_{01})\cup 
\bigcup_{w\in W_n}f_{w2}(X_{00}),
\end{equation}
and
\begin{equation}\label{e:xenaplus}
X_{n+1}=X_n \cup \bigcup_{w\in W_n}f_{w2}(X_{01}),
\end{equation}
where the latter term on the right hand side of \eqref{e:xenaplus} 
consists of branches of $X_{n+1}$ which do not belong to $X_n$ and which are mutually 
disjoint, modulo the starting point of one of them.
Indeed, for $n=1$ this follows from \eqref{e:femulte2}. For $n\geq 2$, 
this follows from the argument in the previous 
paragraph, with the observation that the set $\{f_{w2}(X_{01})\mid w\in W_n\}$ consists in mutually disjoint 
segments, each one representing a new branch of $X_{n+1}$ and which starts on a branch of $X_n$. This 
concludes the proof that $X_{n+1}$ is a binary tree.

We use \eqref{e:xenaplus} and get that, for any $k=1,\ldots,n$, modulo a finite number of points, the starting 
points of the new branches, we have
\begin{equation}\label{e:xeka}
X_k\setminus X_{k-1} = \bigcup_{w\in W_{k-1}} f_{w2}(X_{01})
\end{equation}
then, letting $k=1,\ldots,n,$ we get
\begin{align*}
X_n& =X_{0}\cup \bigcup_{k=0}^n (X_k\setminus X_{k-1}) =X_0\cup \bigcup_{k=0}^n \bigcup_{w\in W_{k-1}} 
f_{w2}(X_{01})\\
\intertext{and then, taking into account that $f_{w}$ is an affine function in either $z$ or $\bar z$, 
that $X_{01}$ is a segment,  and that the set $\{f_{w2}(X_{01})\mid w \in W_1\cup \cdots \cup W_n\}$ is made up 
by segments which have only the left end common with another segment, this 
equals}
& =[0,1]\cup c(0,1]\cup \bigcup_{k=0}^{n-1} \bigcup_{w\in W_k} (f_{w2}(0),f_{w2}(c)].\qedhere
\end{align*}
\end{proof}

Finally, we can tackle some variants in which we can organise the Hata tree-like self-similar set as a strictly 
inductive system of measure spaces. 

\begin{example}\label{ex:hata1}
First, the parametrisation \eqref{e:xenaz} is very helpful for this purpose 
since it provides a computable tool for the topologies and the $\sigma$-algebras on the sets $X_n$.
For each $m\in\NN_0$, let $\Omega_m$ be the Borel $\sigma$-algebra generated by the induced topology of 
the complex plane $\CC$ on $X_m$. Since each $X_m$ is a connected subset of $\CC$ which is a  tree having 
$2^m+1$ branches, we easily see that, for each $m\in\NN_0$, we 
have $\Omega_m=\{A\cap X_m\mid A\in\Omega_{m+1}\}\subseteq \Omega_{m+1}$, hence 
$(X_m,\Omega_m)_{m\in\NN_0}$ is a strictly inductive system of measurable spaces. Letting, for each 
$m\in\NN_0$, $\mu_m$ denote by Borel measure on $X_m$, we see that this is actually generated by the length 
of subsets made up as a finite union of intervals, and then $(X_m,\Omega_m,\mu_m)_{m\in\NN_0}$ becomes
a strictly inductive system of measure spaces.
\end{example}

Starting with the explicit description of the set $X_n$ as in Proposition~\ref{p:hata}, other variants of organising 
the Hata tree-like self-similar set as an iductive system of measure spaces can be obtained in such a way that the 
index set is not linear. This can be seen graphically in Figure~\ref{f:hata} but it can be made in at least two 
more variants.

\begin{example}\label{ex:hata2} Starting from the previous example, we consider the sequence 
$(X_n)_{n\geq 0}$, which is explicitly described in \eqref{e:xenaz}. We can define a different strictly inductive 
system of measure spaces which, in view of a spectral theory on the Hata tree-like self-similar set, has the 
advantage of containing information about each branch. We define
\begin{equation*}
\Lambda=\{(n,k)\mid n,k\in \NN_0,\ 0\leq k\leq 2^n\}, 
\end{equation*}
on which we define the following preorder: if $(n,k)$ and $(m,l)$ are in $\Lambda$ then $(n,k)\leq (m,l)$ if either
$n< m$ or $n=m$ and $k=0$.

In order to define the system of sets, for each $n\in\NN$ we consider the set $W_n$ of words of length $n$ and 
make an indexation of its elements $w_1^n,w_2^n,w_3^n,\ldots,w_{2^n}^n$. Then, in view of the parametrisation 
\eqref{e:xeka}, define the sets $X_{(n,k)}$ by
\begin{equation}
X_{(n,k)}:=\begin{cases}X_n\cup f_{w_k^n2}(X_{01}), & k\neq 0,\\
X_n, & k=0.
\end{cases}
\end{equation}
The topology on each $X_{(n,k)}$ is the usual topology induced from $\CC$ 
and the measure is the Borel measure, which is uniquely defined by the length of each segment component.
It is easy to see that $(X_{(n,k)})_{(n,k)\in\Lambda}$, with the underlying measure space structure, is a strictly 
inductive system of measure spaces. 

The first 
observation in the previous example is that the order is not linear. 
One of the advantages of this strictly inductive system of measure spaces is that it singles out each branch of the 
Hata tree-like self-similar set and hence, in view of a tentative spectral theory, it catches more accurately its 
topology.
\end{example}

\begin{remark} With notation as in the previous example,
let $(\epsilon_m)_{m\in \NN}$ be the 
sequence in $\Lambda$ defined by $\epsilon_m=(m,0)$. Then, observe that the following hold.
\begin{itemize}
\item[(c1)] For each $m\in\NN$ we have $\epsilon_m\leq \epsilon_{m+1}$.
\item[(c2)] For each $\lambda\in\Lambda$ there exists $m\in\NN$ such that $\lambda\leq \epsilon_m$.
\item[(c3)] For each $m\in\NN$ the set of $\lambda\in\Lambda$ such that $\lambda\leq \epsilon_m$ is finite.
\end{itemize}
\end{remark}

\begin{example}\label{ex:hata3} Let $\Lambda$ denote the set of all subsets $\lambda$ of the Hata tree-like set 
subject to the following conditions.
\begin{itemize}
\item[(x1)] $\lambda$ is connected.
\item[(x2)] $\lambda$ is a finite union of branches of $X_n$, for some $n\in\NN_0$.
\end{itemize}
With respect to order defined by set inclusion, clearly $\Lambda$ is a directed set. 

We consider the system $(X_\lambda)_{\lambda\in\Lambda}$ defined by 
$X_\lambda=\lambda$, for all $\lambda\in\Lambda$. The topology on each $X_\lambda$ is induced from the 
topology of $\CC$ and the measure on $X_\lambda$ is the Borel measure defined by the length function of 
segments on each of its 
branches. It is easy to see that $(X_\lambda)_{\lambda\in\Lambda}$ is a strictly inductive system of measure 
spaces. In addition, let the sequence $(X_n)_{n\in\NN_0}$ be the sequence defined recursively 
by $X_{n+1}:=F(X_n)$, $n\geq 0$, where $X_0=[0,1]\cup c[0,1]$, $F$ is defined at \eqref{e:fexa} and the 
functions $f_1$ and $f_2$ are defined at \eqref{e:fef}. Then, defining $\epsilon_m=X_m$, for all $m\in\NN$,
it easy to see that the strictly inductive system of measure spaces $(X_\lambda)_{\lambda\in\Lambda}$ 
satisfies the conditions (c1) through (c3) from Example~\ref{ex:hata2}.

This strictly inductive system of measure spaces is reacher than the one in Example~\ref{ex:hata2} because
it catches with a finer precision the structure of the Hata tree-like self-similar set. However, in order to use it for a 
spectral analysis, one should find a parametrisation of its elements, similar to what we get in \eqref{e:xenaz}.
\end{example}

\begin{remark} The parametrisation of a self-similar set is important in view of a spectral analysis on fractal sets. 
The approach used in \cite{Kigami} makes use of Dirichlet forms on networks that can be 
expressed as limits of networks of finite sets. But the whole picture is much more powerful since, 
to each Dirichlet form $(\cE,\dom(\cE))$ 
defined on $L^2(X,\mu)$, for a locally compact space $X$ endowed with a regular Borel measure $\mu$, there corresponds a 
strongly continuous semigroup on $L^2(X,\mu)$ with the Markov property, see Appendix B.3 in \cite{Kigami}. 
Although regularity of Borel measures on locally compact sets is not guaranteed, this is guaranteed on a 
compact metric space $X$, see e.g.\ Theorem~5.2.1 in H.~Geiss, S.~Geiss \cite{Geiss}. For the case of the 
Hata tree-like self-similar set, we should define a net of Dirichlet forms on each 
$L^2(X_\lambda,\mu_\lambda)$, since we have
a Borel space $(X_\lambda,\mu_\lambda)$ 
on the compact metric space $X_\lambda$, in which the Borel measure $\mu_\lambda$ is automatically regular, 
and then investigate the limit to the whole set $X$, that is, on the locally Hilbert space 
$\varinjlim_{\lambda\in\Lambda} L^2(X_\lambda,\mu_\lambda)$. This approach requires a further study of strongly continuous 
semigroups on locally Hilbert spaces and their correspondence with Dirichlet forms on these spaces.
\end{remark}

\section{Spectral Theorems}\label{s:st}

\subsection{Resolvent and Spectrum}\label{ss:ras}
With notation as in Subsection~\ref{ss:lcsab}, if $T\in\Lloc(\cH)$, its \emph{resolvent set} is
\begin{equation}\label{e:res} \rho(T)=\{z\in\mathbb{C}\mid zI-T\mbox{ is invertible in }\Lloc(\cH)\},
\end{equation}
and its \emph{spectrum} is
\begin{equation}\label{e:spec} \sigma(T)=\mathbb{C}\setminus\rho(T)=
\{z\in \mathbb{C}\mid zI-T\mbox{ is not invertible in }\Lloc(\cH)\}.
\end{equation}
For the proof of the next result see Lemma~2.1 in \cite{Gheondea3}.

\begin{lemma}\label{l:inv} Let $T\in\Lloc(\cH)$ and assume that there exists $S\colon \cH\ra\cH$ a 
linear operator such that $TS=ST=I_\cH$. Then $S\in\Lloc(\cH)$.
\end{lemma}

For the proof of the next result see Proposition~2.2 in \cite{Gheondea3}. The first statement is actually a simple 
consequence of Lemma~\ref{l:inv}.

\begin{proposition}\label{p:spec} Let $T=\varprojlim_{\lambda\in\Lambda} T_\lambda$ be a 
locally bounded operator in $\Lloc(\cH)$. Then
\begin{equation}\label{e:rotez}
\rho(T)=\{z\in\mathbb{C}\mid zI-T\mbox{ is invertible in }\cL(\cH)\},
\end{equation} and
\begin{equation}\label{e:roteb} 
\rho(T)=\bigcap_{\lambda\in\Lambda}\rho(T_\lambda),\quad
\sigma(T)=\bigcup_{\lambda\in\Lambda} \sigma(T_\lambda).
\end{equation}
\end{proposition}

As a consequence of the previous proposition, we see that the spectrum of a locally 
bounded operator is always nonempty but it may be neither closed nor bounded.
On the other hand, if $\Lambda$ is countable, then the resolvent set of any operator $T\in\Lloc(\cH)$
is a $G_\delta$-set in $\CC$, while its spectrum is always an $F_\sigma$-subset of $\CC$.

\subsection{Locally Normal Operators}\label{ss:lno}
A linear operator $T\colon\cH\ra\cH$ is called
\emph{locally normal} if $T\in\Lloc(\cH)$ and $TT^*=T^*T$, respectively, \emph{locally selfadjoint} if
$T\in\Lloc(\cH)$ and $T=T^*$. In a similar way, $T\in\Lloc(\cH)$ is \emph{locally positive}, also 
denoted by $T\geq 0$, if 
$\langle Th,h\rangle \geq 0$. As usual, this opens the
possibility of introducing an order relation between locally selfadjoint operators: $A\geq B$ if 
$A-B\geq 0$.

Let $T\in\Lloc(\cH)$, $T=\varprojlim_{\lambda\in\Lambda}T_\lambda$.
Then $T$ is locally normal if and only if $T_\lambda$ is normal for all $\lambda\in\Lambda$. 
Similarly, $T$ is locally selfadjoint (locally positive) if and only if $T_\lambda$ 
is selfadjoint (positive) for all $\lambda\in\Lambda$. In particular, 
taking into account Proposition~\ref{p:spec}, it follows that, if $T$ is a locally normal operator
then it is locally selfadjoint (locally positive) if and only if 
$\sigma(T)\subseteq\RR$ ($\sigma(T)\subseteq\RR_+$).

\begin{example}\label{ex:lnop} With notation as in
Example~\ref{ex:locli}, given a function $\phi\in L^\infty_{\mathrm{loc}}(X,\mu)$ we can define the
operator $M_\phi\colon L^2_{\mathrm{loc}}(X,\mu)\ra L^2_{\mathrm{loc}}(X,\mu)$ of multiplication
by $\phi$, 
\begin{equation}\label{e:mop}
M_\phi f:=\phi f,\quad f\in L^2_{\mathrm{loc}}(X,\mu).
\end{equation}
Indeed,
we first observe that, by this definition, 
$M_\phi\in\Lloc(L^2_{\mathrm{loc}}(X,\mu))$. In order to see this, 
for each $\lambda\in\Lambda$ we consider the operator 
$M_{\phi|_{X_\lambda}}\colon L^2(X_\lambda,\mu_\lambda)\ra
L^2(X_\lambda,\mu_\lambda)$ of multiplication with 
$\phi|_{X_\lambda}\in L^\infty(X_\lambda,\mu_\lambda)$. Then 
$M_{\phi|_{X_\lambda}}$ is a bounded linear operator and it is
easy to see that the net $(M_{\phi|_{X_\lambda}})_{\lambda\in\Lambda}$ satisfies the properties
as in \eqref{e:temuha} hence, by Proposition~\ref{p:lbo2}, it follows that there exists a unique 
operator $M_\phi\colon L^2_{\mathrm{loc}}(X,\mu)\ra L^2_{\mathrm{loc}}(X,\mu)$ such that
\begin{equation}\label{e:teloc} 
M_\phi|_{L^2(X_\lambda,\mu_\lambda)}=J_\lambda M_{\phi|_{X_\lambda}},\quad \lambda\in\Lambda,
\end{equation}
where $J_\lambda$ denotes the canonical embedding of $L^2(X_\lambda,\mu_\lambda)$ into
$L^2_\mathrm{loc}(X,\mu)$. So, technically, we can write
\begin{equation}\label{e:mefivar}
M_\phi=\varprojlim_{\lambda\in\Lambda} M_{\phi|_{X_\lambda}}.
\end{equation}
Moreover, $M_\phi$ is a locally normal operator and
\begin{equation}\label{e:alno}
M_\phi^*=M_{\overline \phi}.
\end{equation}
In addition, 
\begin{equation}\label{e:samo}
\sigma(M_\phi)=\essran(\phi):=\bigcup_{\lambda\in\Lambda} \essran(\phi_\lambda),
\end{equation}
where $\essran(f)$ denotes the essential range of a function $f$ from $L^\infty$.
In particular, $M_\phi$ is locally selfadjoint (locally
positive) if and only if $\phi$ has real essential range (nonnegative essential range).

On the other hand, the map 
$L^\infty_{\mathrm{loc}}(X,\mu)\ni \phi \mapsto M_\phi\in\Lloc(L^2_{\mathrm{loc}}(X,\mu))$
is a coherent $*$-representation of the locally $C^*$-algebra $L^\infty_{\mathrm{loc}}(X,\mu)$. 
\end{example}

The following result is a counter-part of the celebrated Fuglede-Putnam Theorem for locally normal operators,
see Theorem~2.10 in \cite{Gheondea3}.

\begin{theorem}\label{t:fp}
Let $N\in\Lloc(\cH)$ and $M\in\Lloc(\cK)$ be locally normal operators, with the locally Hilbert spaces
$\cH$ and $\cK$ modelled on the same directed set $\Lambda$, and let $B\in\Lloc(\cK,\cH)$ be such
that $NB=BM$. Then $N^*B=BM^*$.
\end{theorem}

With respect to the definition of a locally normal operator we should observe that, if 
$(T_\lambda)_{\lambda\in\Lambda}$ is a net of normal operators, 
$T_\lambda\in\cB(\cH_\lambda)$ for all $\lambda$, then it is sufficient to impose the condition (lbo1), 
see Subsection~\ref{ss:lbo}, and then the condition (lbo2) follows as a consequence of the classical 
Fuglede-Putnam Theorem, in order for the net to define a locally normal operator $T=\varprojlim_{\lambda\in\Lambda}T_\lambda$.

\subsection{Locally Spectral Measures}\label{ss:lsm}
Let $\cH=\varinjlim_{\lambda\in\Lambda}\cH_\lambda$ be a locally Hilbert space.
A linear operator $E\colon\cH\ra\cH$ is called a \emph{local projection} if $E\in\Lloc(\cH)$ 
and $E^2=E=E^*$. Clearly, $E$ is a local projection if and only if $I-E$ is a local projection.
As a consequence of Proposition~\ref{p:lbo2} we have the following
characterisation of local projections.

\begin{lemma}\label{e:loproj} Given $E\in\Lloc(\cH)$, the following assertions are equivalent:
\begin{itemize}
\item[(i)] $E$ is a local projection.
\item[(ii)] $E=\varprojlim_{\lambda\in\Lambda}E_\lambda$ with $E_\lambda\in\cB(\cH_\lambda)$
projections (that is, $E_\lambda^2=E_\lambda^*=E_\lambda$) for all $\lambda\in\Lambda$ and,
for all $\lambda,\nu\in\Lambda$ with $\lambda\leq\nu$ we have
\begin{equation}\label{e:emuhal}
E_\nu|\cH_\lambda=J_{\nu,\lambda}E_\lambda\mbox{ and }E_\nu P_{\lambda,\nu}
=P_{\lambda,\nu}E_\nu,
\end{equation}
where
$P_{\lambda,\nu}=J_{\nu,\lambda}^*$ denotes the projection of $\cH_\nu$ onto $\cH_\lambda$.
\end{itemize}
\end{lemma}

Proposition~2.5 in \cite{Gheondea3} provides equivalent characterisations for the ranges of local projections.
Proposition~2.6 in \cite{Gheondea3} shows that the geometry of local projections is close to that of projections 
in Hilbert spaces.
Also, let us recall that for each $\lambda\in\Lambda$ a Hermitian projection $P_\lambda\in\cL(\cH)$ is 
uniquely defined such that $\ran(P_\lambda)=\cH_\lambda$, cf.\ Lemma~\ref{l:perp} and the 
definition thereafter. In Proposition~2.8  in \cite{Gheondea3} 
there are characterised the local projections within the larger class of Hermitian projections.

In the following we consider a 
strictly inductive system of measurable spaces 
$((X_\lambda,\Omega_\lambda))_{\lambda\in\Lambda}$, see Subsection~\ref{ss:sisms}.
We consider a
 \emph{projective system of spectral measures} $(E_\lambda)_{\lambda\in\Lambda}$
with respect to $(X_\lambda,\Omega_\lambda,\cH_\lambda)_{\lambda\in\Lambda}$, that is:
\begin{itemize}
\item[(psm1)] For each $\lambda\in\Lambda$, $E_\lambda$ is a spectral measure with respect to 
$(X_\lambda,\Omega_\lambda,\cH_\lambda)$.
\item[(psm2)] For any $\lambda,\nu\in\Lambda$ with $\lambda\leq\nu$ we have
\begin{equation*}
E_\nu(A)|\cH_\lambda =J_{\nu,\lambda}E_\lambda(A),\quad A\in\Omega_\lambda,
\end{equation*}
and
\begin{equation*}
E_\nu(A) P_{\lambda,\nu}=P_{\lambda,\nu}E_\nu(A),\quad A\in\Omega_\nu.
\end{equation*}
\end{itemize}

In the following we consider the inductive limit $(X,\Omega)$ of the strictly inductive system
of measurable spaces $((X_\lambda,\Omega_\lambda))_{\lambda\in\Lambda}$,
see \eqref{e:lms}, as well as its extension to a measurable space $(X,\widetilde\Omega)$ as in 
\eqref{e:tomega}. For the following result see Lemma~3.2 in \cite{Gheondea3}.

\begin{lemma}\label{l:ext} There exists a unique mapping $E\colon\widetilde\Omega\ra\Lloc(\cH)$ such that
\begin{equation}\label{e:exta}
E(A)|\cH_\lambda =J^\cH_\lambda E_\lambda(A\cap X_\lambda),\quad A\in\widetilde\Omega,\ 
\lambda\in\Lambda.
\end{equation}
In addition, the mapping $E$ has the following properties.
\begin{itemize}
\item[(i)] $E(A)=E(A)^*$ for all 
$A\in\widetilde\Omega$.
\item[(ii)] $E(\emptyset)=0$ and $E(X)=I$.
\item[(iii)] $E(A_1\cap A_2)=E(A_1)E(A_2)$ for all $A_1,A_2\in\widetilde\Omega$.
\item[(iv)] For any sequence $(A_n)_n$ of mutually disjoint sets in $\widetilde\Omega$ we have
\begin{equation*}
E(\bigcup_{n=1}^\infty A_n)=\sum_{n=1}^\infty E(A_n).
\end{equation*}
\end{itemize}
\end{lemma}

Let us recall that 
the mapping $E\colon\widetilde\Omega\ra\Lloc(\cH)$ as in Lemma~\ref{l:ext} is defined as follows:
for each $\lambda\in\Lambda$ we define
\begin{equation}\label{e:ext}
E_\lambda(A):=E_\lambda(A\cap X_\lambda),\quad A\in\widetilde\Omega,
\end{equation}
which makes perfectly sense since $A\cap X_\lambda\in\Omega_\lambda$. We observe that, for each
$A\in\widetilde\Omega$, the net $(E_\lambda(A))_{\lambda\in\Lambda}$ satisfies 
the conditions \eqref{e:temuha}, which follows by the assumption (psm2). Then, by Proposition~\ref{p:lbo2},
it follows that there exists $E(A)\in\Lloc(\cH)$ such that \eqref{e:exta} holds.

We call the mapping $E\colon\widetilde\Omega\ra\Lloc(\cH)$ obtained in Lemma~\ref{l:ext} a
\emph{locally spectral measure}.
On the other hand, when restricted to the locally $\sigma$-ring 
$\Omega\subseteq\widetilde\Omega$, the mapping $E|\Omega$ has the following properties.
\begin{itemize}
\item[(i)] $E(A)=E(A)^*$ for all 
$A\in\Omega$.
\item[(ii)] $E(\emptyset)=0$.
\item[(iii)] $E(A_1\cap A_2)=E(A_1)E(A_2)$ for all $A_1,A_2\in\Omega$.
\item[(iv)] For any sequence $(A_n)_n$ of mutually disjoint sets in $\Omega$ such that $\bigcup_{n=1}^\infty A_n\in\Omega$, we have
\begin{equation*}
E(\bigcup_{n=1}^\infty A_n)=\sum_{n=1}^\infty E(A_n).
\end{equation*}
\end{itemize}
We call $(E,X,\Omega,\cH)$
the \emph{projective limit} of the projective system of spectral measures $((E_\lambda,X_\lambda,
\Omega_\lambda,\cH_\lambda))_{\lambda\in\Lambda}$ and use the notation
\begin{equation}\label{e:plsm}
(E,X,\Omega,\cH)=\varprojlim_{\lambda\in\Lambda} (E_\lambda,X_\lambda,
\Omega_\lambda,\cH_\lambda).
\end{equation}

A locally spectral measure $E$, as we defined it, cannot be considered as a spectral measure, due to the fact
that its codomain is $\Lloc(\cH)$. However, see Proposition~3.3 in \cite{Gheondea3}, 
it can be lifted to a spectral measure. 
More precisely, taking into account that
any local projection is actually a bounded operator on the pre-Hilbert space $\cH$, the range of
$E$ is actually contained in its bounded part $b\Lloc(\cH)$ which is canonically embedded in 
$\cB(\widetilde\cH)$. Therefore, for each $A\in\widetilde\Omega$, since 
$E(A)\leq I$, it follows that $E(A)$ is bounded hence it has a unique extension to an operator 
$\widetilde E(A)\in\cB(\widetilde\cH)$.

In the following we recall how we can produce locally normal operators by integration of locally bounded
$\widetilde\Omega$-measurable functions with respect to a locally spectral measure. First, one has to 
clarify the meaning of $\widetilde\Omega$-measurability. For the proof of the following result see Lemma~3.4
in \cite{Gheondea3}.

\begin{lemma}\label{l:lmf} Let $\phi\colon X\ra \CC$. The following assertions are equivalent:
\begin{itemize}
\item[(i)] $\phi$ is $\widetilde\Omega$-measurable.
\item[(ii)] For each $\lambda\in\Lambda$ the function $\phi|X_\lambda$ is $\Omega_\lambda$-measurable.
\end{itemize}
\end{lemma}

In the following let us denote
\begin{equation}\label{e:bel}
B_{\mathrm{loc}}(X,\widetilde\Omega):=\{\phi\colon X\ra\CC\mid 
\phi|X_\lambda\mbox{ is bounded and }\Omega_\lambda\mbox{-measurable for all }\lambda\in\Lambda\}.
\end{equation}
It is easy to see that $B_{\mathrm{loc}}(X,\widetilde\Omega)$ is a $*$-algebra of complex functions, with
usual algebraic operations and involution. Letting, for each $\lambda\in\Lambda$,
\begin{equation}\label{e:bes}
p_\lambda(\phi)=\sup_{x\in X_\lambda}|\phi(x)|\quad \phi\in B_{\mathrm{loc}}(X,\widetilde\Omega),
\end{equation}
we obtain a family of $C^*$-seminorms $\{p_\lambda\}_{\lambda\in\Lambda}$
with respect to which 
$B_{\mathrm{loc}}(X,\widetilde\Omega)$ becomes a locally $C^*$-algebra.
More precisely, letting 
$B(X_\lambda,\Omega_\lambda)$ denote the $C^*$-algebra of all bounded
and $\Omega_\lambda$-measurable functions $f\colon X_\lambda\ra\CC$, we have
\begin{equation}\label{e:belal}
B_{\mathrm{loc}}(X,\widetilde\Omega)=\varprojlim_{\lambda\in\Lambda}B(X_\lambda,\Omega_\lambda).
\end{equation}

Let now $\phi\in B_{\mathrm{loc}}(X,\widetilde\Omega)$ be fixed. By Proposition~9.4 in \cite{Conway}, 
for each $\lambda\in\Lambda$ there exists a unique normal operator $N_\lambda\in\cB(\cH_\lambda)$
such that
\begin{equation}\label{e:nel}
N_\lambda(\phi) =\int_{X_\lambda} \phi(x)\de E_\lambda(x),
\end{equation}
in the following sense: for each $\epsilon>0$ and $\{A_1,\ldots,A_n\}$ an 
$\Omega_\lambda$-partition of $X_\lambda$ such that, $\sup\{|\phi(x)-\phi(y)|<\epsilon\mid x,y\in A_k\}$
for all $k=1,\ldots,n$, and for any choice of points $x_k\in A_k$, for all $k=1,\ldots,n$, we have
\begin{equation}\label{e:nelas}
\| N_\lambda(\phi) -\sum_{k=1}^n \phi(x)E_\lambda(A_k)\|<\epsilon,
\end{equation}
where $\|\cdot\|$ denotes the operator norm in $\cB(\cH_\lambda)$. For the proof of the following result
see Proposition~3.5 in \cite{Gheondea3}.

\begin{proposition}\label{p:nes}
For each $\phi\in B_{\mathrm{loc}}(X,\widetilde\Omega)$ there exists a unique locally normal operator $N(\phi)
\in\Lloc(\cH)$ such that $N(\phi)=\varprojlim_{\lambda\in\Lambda}N_\lambda(\phi)$. In addition, 
the mapping $B_{\mathrm{loc}}(X,\widetilde\Omega)\ni \phi\mapsto N(\phi)\in \Lloc(\cH)$ is a coherent 
$*$-representation. 
\end{proposition}

In view of Proposition~\ref{p:nes}, for any $\phi\in B_{\mathrm{loc}}(X,\widetilde\Omega)$ we denote
\begin{equation}\label{e:isin}
\int_{\sigma(N)} \phi(x)\de E(x):= N(\phi)=\varprojlim_{\lambda\in\Lambda} N_\lambda(\phi)
=\varprojlim_{\lambda\in\Lambda} \int_{\sigma(N_\lambda)} \phi(x)\de E_\lambda (x).
\end{equation}

\subsection{The First Spectral Theorem}\label{ss:fst}
Let $N\in\Lloc(\cH)$ be a locally normal operator, for some locally Hilbert space 
$\cH=\varinjlim_{\lambda\in\Lambda}\cH_\lambda$. By definition, see Subsection~\ref{ss:pllcs}, there exists
uniquely the net $(N_\lambda)_{\lambda\in\Lambda}$ with $N=\varprojlim_{\lambda\in\Lambda}N_\lambda$,
in the sense of the conditions (lbo1) and (lbo2). By Proposition~\ref{p:spec} we have 
$\sigma(N)=\bigcup_{\lambda\in\Lambda}\sigma(N_\lambda)$. For each $\lambda\in\Lambda$ the spectrum
$\sigma(N_\lambda)$ is a compact nonempty subset of $\CC$ and $(\sigma(N_\lambda),\cB_\lambda)$ is
a measurable space, where $\cB_\lambda$ denotes the $\sigma$-algebra of all Borel subsets of 
$\sigma(N_\lambda)$. Then, $((\sigma(N_\lambda),\cB_\lambda))_{\lambda\in\Lambda}$ 
is a strictly inductive system of 
measurable spaces, as in Subsection~\ref{ss:sisms}, and letting $(\sigma(N),\cB)$ be its inductive limit in the 
sense of \eqref{e:lms}, it is easy to see that, with respect to \eqref{e:tomega}, $\widetilde\cB$ is
the $\sigma$-algebra of all Borel subsets of $\sigma(N)$.

Further on, by the Spectral Theorem for normal operators in Hilbert spaces, e.g.\ see Theorem~10.2 
in \cite{Conway}, for each $\lambda\in\Lambda$, let $E_\lambda$ denote the spectral measure of 
$N_\lambda$ with respect to the triple $(\sigma(N_\lambda),\cB_\lambda,\cH_\lambda)$. In particular,
\begin{equation}\label{e:nexus}
N_\lambda=\int_{\sigma(N_\lambda)} z\de E_\lambda(z),\quad I_\lambda=\int_{\sigma(N_\lambda)} 
\de E_\lambda(z),
\end{equation}
where $I_\lambda$ denotes the identity operator on $\cH_\lambda$. By Lemma~3.6 in \cite{Gheondea3},
$((E_\lambda,\sigma(N_\lambda),\cB_\lambda,\cH_\lambda))_{\lambda\in\Lambda}$ is a projective system
of spectral measures. For the proof of the following theorem, that we call The First Spectral Theorem for locally 
normal operators, see Theorem~3.7 in \cite{Gheondea3}.

\begin{theorem}\label{t:st1}
For any locally normal operator $N\in\Lloc(\cH)$ there exists a unique locally spectral measure $E$, with
respect to the Borel measurable space $(\sigma(N),\widetilde\cB)$ 
and the locally Hilbert space $\cH$, with the following properties.
\begin{itemize}
\item[(i)] $N=\int_{\sigma(N)} z\de E(z)$ and $I=\int_{\sigma(N)} \de E(z)$, in the sense of \eqref{e:isin}.
\item[(ii)] For any nonempty relatively open subset $G$ of $\sigma(N)$ we have $E(G)\neq 0$.
\item[(iii)] If $T\in\Lloc(\cH)$ then $TN=NT$ if and only if $TE(A)=E(A)T$ for all Borel subset of $\sigma(A)$.
\end{itemize}
\end{theorem}

As a consequence of Theorem~\ref{t:st1} we can obtain the functional calculus with locally bounded Borel
functions, see Theorem~3.8 in \cite{Gheondea3}.

\begin{theorem}\label{t:bfc} Let $N\in\Lloc(\cH)$ be a locally normal operator and $E$ its locally spectral 
measure. Then the mapping
\begin{equation}\label{e:bfc}
B_{\mathrm{loc}}(\sigma(N))\ni \phi\mapsto \phi(N)=\int_{\sigma(N)} \phi(z)\de E(z)\in \Lloc(\cH)
\end{equation}
is a coherent $*$-representation of the locally $C^*$-algebra $B_{\mathrm{loc}}(\sigma(N))$, 
and it is the unique coherent $*$-morphism of locally $C^*$-algebras such that:
\begin{itemize}
\item[(i)] $\zeta(N)=N$ and $1(N)=I$, where $\zeta(z)=z$ and $1(z)=1$ for all $z\in\CC$.
\item[(ii)] For any net $(\phi_j)_{j\in\cJ}$ such that $\phi_j\ra 0$ pointwise and 
$\sup_{j\in\cJ}\sup_{x\in \sigma(N_\lambda)}|\phi_j(x)|<+\infty$ for each $\lambda\in\Lambda$, it follows that 
$\phi_j(N)\ra 0$ strongly operatorial.
\end{itemize}
\end{theorem}

\begin{example}\label{ex:fm} With notation as in
Example~\ref{ex:locli} and Example~\ref{ex:lnop}, for $\phi\in L^\infty_{\mathrm{loc}}(X,\mu)$ consider
the locally normal operator $M_\phi\in\Lloc(L^2_\mathrm{loc}(X,\mu))$ of multiplication with $\phi$ as in
\eqref{e:mop}. We use the same symbol $\phi\colon X\ra \CC$ to denote a representative of the coset
of $\phi$. By Proposition~10.4 in \cite{Conway}, for each $\lambda\in\Lambda$ the spectral measure
$E_\lambda$ of $M_{\phi_\lambda}\in\cB(L^2(X_\lambda,\mu_\lambda))$ is defined by
\begin{equation}\label{e:elam}
E_\lambda(A)=M_{\chi_{\phi_\lambda^{-1}(A)}},\quad A\mbox{ is any Borel subset of }
\sigma(M_{\phi_\lambda})=
\essran(\phi_\lambda),
\end{equation}
where $\phi_\lambda=\phi|X_\lambda$. Then it can be shown that 
$((E_\lambda,\essran(\phi_\lambda),\cH_\lambda))_{\lambda\in\Lambda}$ is a projective system
of spectral measures and letting $E$ denote the corresponding
$\Lloc(\cH)$-valued locally spectral measure, defined on the $\sigma$-algebra of all Borel subsets
of
\begin{equation*}
\sigma(M_\phi)=\essran(\phi)
=\bigcup_{\lambda\in\Lambda}\essran(\phi_\lambda),
\end{equation*} it follows that
\begin{equation}\label{e:smac}
E(A)=M_{\chi_{\phi^{-1}(A)}},\ A\mbox{ is any Borel subset of }\sigma(M_\phi).
\end{equation}
On the other hand, as shown in Example~\ref{ex:lnop},
\begin{equation*}
M_\phi=\varprojlim_{\lambda\in\Lambda} M_{\phi_\lambda},
\end{equation*}
hence, in view of Theorem~\ref{t:st1}, $E$ defined 
at \eqref{e:smac} is the locally spectral measure of $M_\phi$.

In addition, if $\psi\in B_\mathrm{loc}(\sigma(M_\phi))$ then, in view of Theorem~\ref{t:bfc}, we have
$\psi(M_\phi)=M_{\psi\circ\phi}$.
\end{example}

\subsection{The Second Spectral Theorem.}\label{ss:tsst}
We need to perform certain operations with locally Hilbert spaces, such as orthogonal sums. 
Let $\cJ$ be an arbitrary set and consider a family
of Hilbert spaces $\bH=\{\cH_j\}_{j\in \cJ}$. We consider the bundle of Hilbert spaces
\begin{equation*}\cJ\ast \bH := \bigsqcup_{j\in\cJ}\cH_j=\{(j,h)\mid j\in\cJ,\ h\in \cH_j\},
\end{equation*}
and the canonical map $\pi\colon \cJ\ast \bH\ra \cJ$, $\pi(j,h)=j$ for all $(j,h)\in\cJ\ast \bH$. A \emph{section} of
the bundle $\cJ\ast \bH$ is a left inverse of $\pi$, that is, $f\colon \cJ\ra \cJ\ast \bH$ subject to the condition that,
for any $j\in\cJ$ we have $f(j)=(j,h)$, with $h\in\cH_j$,
Let $\cF_{\bH}(\cJ)$ denote the vector space of all sections of the bundle $\cJ\ast\bH$ and let 
$\cH=\bigoplus_{j\in \cJ}\cH_j$ denote the collection of all sections $(h_j)_{j\in \cJ}\in\cF_\bH(\cJ)$ 
subject to the condition that
\begin{equation*}\sum_{j\in \cJ}\|h_j\|_{\cH_j}^2<\infty.
\end{equation*}
A standard argument of Hilbert spaces shows that if $(h_j)_{j\in J}\in\cH$ then $h_j=0$ 
for all less an at most countable set of $j\in \cJ$. Then, if $h=(h_j)_{j\in \cJ}$ and 
$k=(k_j)_{j\in \cJ}$ are arbitrary sections in $\cH$ then, letting
\begin{equation}\label{e:laheka}
\langle h,k\rangle_0=\sum_{j\in J}\langle h_j,k_j\rangle_{\cH_j},
\end{equation}
correctly defines an inner product in $\cH$ with respect to which it becomes a Hilbert 
space and the summation is actually an absolutely convergent series. The Hilbert space 
$\cH=\bigoplus_{j\in \cJ}\cH_j$ defined before is called the \emph{special orthogonal direct sum} of $\bH$. This 
construction can be put in more general perspective by using reproducing kernel Hilbert spaces associated to 
operator valued kernels as in \cite{GheondeaTilki}, with details provided in \cite{GheondeaKulkarniPamula}, but 
for our purposes this is sufficient.
With these definitions, the proof of the next proposition is straightforward.

\begin{proposition}\label{p:dirsum} Let $\Lambda$ be a directed set and let $\cJ$ be an arbitrary index set. 
Assume that for each $j\in\cJ$ there is given a strictly inductive system of Hilbert spaces 
$(\cH_{\lambda,j})_{\lambda\in\Lambda}$. Then, the net of special orthogonal direct sums 
$(\bigoplus_{j\in\cJ} \cH_{\lambda,j})_{\lambda\in\Lambda}$
is a strictly inductive system of Hilbert spaces. In addition, if for any $j\in\cJ$ the strictly inductive system of
Hilbert spaces $(\cH_{\lambda,j})_{\lambda\in\Lambda}$ is representing then 
$(\bigoplus_{j\in\cJ} \cH_{\lambda,j})_{\lambda\in\Lambda}$
is a representing strictly inductive system of Hilbert spaces as well.
\end{proposition}

The next example describes a special representing locally Hilbert space, that is slightly more general than
the concrete representing locally Hilbert space described in Example~\ref{ss:crlhs}, 
which turns out to be suitable for the representation of any locally normal operator on a representing locally 
Hilbert space in such a way that it contains information about the spectral multiplicity as well.

\begin{example}\label{ex:psconc}
For $\Lambda$ a given directed set and each $n\in\NN\cup\{\infty\}$, let 
$((X_{\lambda,n},\Omega_{\lambda,n},\mu_{\lambda,n}))_{\lambda\in\Lambda}$
be a strictly inductive system of measure spaces. In the following, for any natural number 
$n$ we let $\ell^2_n$ denote
the space $\CC^n$ organised as a Hilbert space in the natural way and let $\ell^2_\infty$ denote the Hilbert
space of all square summable complex sequences indexed on $\NN$. Also, with the interpretation of the 
spaces $ L^2(X_{\lambda,n},\mu_{\lambda,n})$ as in Subsection~\ref{ss:crlhs}, let
\begin{equation}\label{e:halen} 
\cH_{\lambda,n}:=\ell^2_n\otimes L^2(X_{\lambda,n},\mu_{\lambda,n}),\quad 1\leq n\leq 
\infty.
\end{equation} 
Since $L^2(X_{\lambda,n},\mu_{\lambda,n})$ is isometrically included in $L^2(X_{\nu,n},\mu_{\nu,n})$
whenever $\lambda,\nu\in \Lambda$ are such that $\lambda\leq \nu$, see Lemma~\ref{l:letexem} and Proposition~\ref{p:tensorprod}, 
it follows that $\cH_{\lambda,n}$
is isometrically embedded in $\cH_{\nu,n}$ and hence $(\cH_{\lambda,n})_{\lambda\in\Lambda}$ is a
strictly inductive system of Hilbert space for each $1\leq n\leq \infty$, see also Remark~\ref{r:tensorprod}.

Consequently, by Proposition~\ref{p:dirsum}, letting
\begin{equation}\label{e:hal}
\cH_\lambda:=\cH_{\lambda,\infty}\oplus\bigoplus_{n\geq 1} \cH_{\lambda,n},
\end{equation}
we get a representing strictly inductive system of Hilbert spaces $(\cH_\lambda)_{\lambda\in\Lambda}$ and 
let
\begin{equation}\label{e:haz}
\cH=\varinjlim_{\lambda\in\Lambda} \cH_\lambda.
\end{equation}

For each $n\in\NN\cup\{\infty\}$, as in Subsection~\ref{ss:sisms} and Subsection~\ref{ss:sismes},
we consider
\begin{equation}\label{e:xenaa}
(X_n,\Omega_n,\mu_n)=\varinjlim_{\lambda\in\Lambda} (X_{\lambda,n},\Omega_{\lambda,n},
\mu_{\lambda,n}), 
\end{equation}
that is,
\begin{equation}\label{e:xenas}
X_n=\bigcup_{\lambda\in\Lambda}X_{\lambda,n},\quad \Omega_n=\bigcup_{\lambda\in\Lambda}\Omega_{\lambda,n},
\end{equation}
and $\mu_n\colon\Omega_n\ra [0,\infty]$, defined by,
\begin{equation}\label{e:menas}
\mu_n(A)=\mu_{\lambda,n}(A),\quad A\in\Omega,
\end{equation}
where $\lambda\in\Lambda$ is any index such that $A\in\Omega_{\lambda,n}$. Also, as in \eqref{e:tomega},
let
\begin{equation}\label{e:tomegas} \widetilde\Omega_n:=\{A\subseteq X_n\mid A\cap 
X_{\lambda,n}\in\Omega_{\lambda,n}\mbox{ for all }\lambda\in\Lambda\}
\end{equation}
and recall that $\widetilde\Omega_n$, by Proposition~3.1 in \cite{Gheondea3},
is a $\sigma$-algebra on $X_n$ which extends $\Omega_n$. Also, as
in \eqref{e:wimu},
define $\widetilde\mu_n\colon\widetilde\Omega_n\ra [0,+\infty]$ by
\begin{equation}\label{e:wimun}
\widetilde\mu_n(A):=\sup_{\lambda\in\Lambda} \mu_{\lambda,n}(A\cap X_{\lambda,n}),\quad 
A\in\widetilde\Omega_n,
\end{equation}
and recall, see Proposition~\ref{p:wimu}, that it is a measure.

Further on, as in Example~\ref{ex:locli}, for each integer $n\in\NN\cup\{\infty\}$, 
let $\phi_n\in L^\infty_\mathrm{loc}(X_n,\mu_n)$, that is, 
$\phi_n\colon X_n\ra [0,\infty]$ is a function, identified $\widetilde\mu_n$-almost everywhere, subject to the
property that $\phi_{\lambda,n}:=\phi_n|_{X_{\lambda,n}}\in L^\infty(X_{\lambda,n},\mu_{\lambda,n})$ 
for all $\lambda\in \Lambda$. In addition, we assume that, for each $\lambda\in\Lambda$ we have
\begin{equation}\label{e:bnd}
\sup_{n\geq 1}\|\phi_{\lambda,n}\|_\infty<\infty.
\end{equation}
Then, for each $n\in\NN\cup\{\infty\}$ we consider the representing locally Hilbert space
\begin{equation}\label{e:eltwolocn}
L^2_{\mathrm{loc}}(X_n,\mu_n)=\varinjlim_{\lambda\in\Lambda}L^2(X_{\lambda,n},\mu_{\lambda,n}),
\end{equation}
the inductive limit of the strictly inductive system of Hilbert spaces
$(L^2(X_{\lambda,n},\mu_{\lambda,n}))_{\lambda\in\Lambda}$ and then,
as in Example~\ref{ex:lnop}, we consider the operator $M_{\phi_n}
\colon L^2_\mathrm{loc}(X_n,\mu_n)\ra L^2_\mathrm{loc}(X_n,\mu_n)$ defined by
\begin{equation}\label{e:mefin}
M_{\phi_n}f:=\phi_n f,\quad f\in L^2_\mathrm{loc}(X_n,\mu_n).
\end{equation}
Recall that
\begin{equation}\label{e:samon}
\sigma(M_{\phi_n})=\essran(\phi_n):=\bigcup_{\lambda\in\Lambda} \essran(\phi_{\lambda,n}),
\end{equation}
and that the operator $M_{\phi_n}$ is locally normal. 

In addition, the mapping 
$L^\infty_\mathrm{loc}(X_n,\mu_n)\ni \phi \mapsto M_\phi \in\cB_\mathrm{loc}(L^2_\mathrm{loc}(X_n,\mu_n))$ 
is an injective $*$-morphism of locally $C^*$-algebras and, because of this, 
in the following, for simplicity, we will denote by 
$L^\infty_\mathrm{loc}(X_n,\mu_n)$ also the range of this injective $*$-morphism. Thus, 
$L^\infty_\mathrm{loc}(X_n,\mu_n)$ is viewed as a 
concrete locally $C^*$-algebra of locally bounded operators
acting on the locally Hilbert space $L^2_\mathrm{loc}(X_n,\mu_n)$.

For simplicity, as before, for each $\lambda\in\Lambda$ and each $n\in\NN\cup\{\infty\}$, we will identify
the $C^*$-algebra $L^\infty(X_{\lambda,n},\mu_{\lambda,n})$ with the range of the faithful $*$-representation
$L^\infty(X_{\lambda,n},\mu_{\lambda,n})\ni \phi\mapsto M_\phi\in\cB(L^2(X_{\lambda,n},\mu_{\lambda,n}))$.
In particular, with this conventional notation, we have
\begin{equation}
L^\infty_\mathrm{loc}(X_n,\mu_n)=\varprojlim_{\lambda\in\Lambda} L^\infty(X_{\lambda,n},\mu_{\lambda,n}).
\end{equation}

Letting $\phi=(\phi_n)_{n\in\NN\cup\{\infty\}}$, we observe that the net
of operators
\begin{equation}
(I_\infty\otimes M_{\phi_{\lambda,\infty}})\oplus 
\bigoplus_{n\geq 1} (I_n\otimes M_{\phi_{\lambda,n}})\in \cB(\cH_\lambda),\quad \lambda\in\Lambda,
\end{equation}
where $\cH_\lambda$ is defined in \eqref{e:hal}, satisfies the conditions in Proposition~\ref{p:lbo2} and,
consequently, this enables us to define the locally bounded operator $M_\phi\in \cB_\mathrm{loc}(\cH)$,
with notation as in \eqref{e:haz}, by
\begin{equation}
M_\phi:=\varprojlim_{\lambda\in\Lambda}\bigl((I_\infty\otimes M_{\phi_{\lambda,\infty}})\oplus 
\bigoplus_{n\geq 1} (I_n\otimes M_{\phi_{\lambda,n}})\bigr).
\end{equation}
Taking into account the assumption \eqref{e:bnd}, 
it is easy to see that $M_\phi$ is a locally normal operator on $\cH$ and
\begin{equation}\label{e:mafi}
\sigma(M_\phi)=\bigcup_{n\in\NN\cup\{\infty\}} \sigma(M_{\phi_n})= \bigcup_{n\in\NN\cup\{\infty\}}
\bigcup_{\lambda\in\Lambda} \essran(\phi_{\lambda,n}).
\end{equation} 

In the following we calculate the spectral measure of the locally normal operator $M_\phi$ and show how the 
representing locally Hilbert space defined at \eqref{e:hal} contains information on the spectral multiplicity of 
$M_\phi$. Let us fix $n\in \NN\cup\{\infty\}$. Taking into account Example~\ref{ex:fm}, the spectral measure
of the operator $M_{\phi_n}$, when viewed as a locally normal operator acting on the locally Hilbert space
$L^2_{\mathrm{loc}}(X_n,\mu_n)$, see \eqref{e:eltwolocn}, is obtained in the following way. For each 
$\lambda\in\Lambda$, as in \eqref{e:elam}, let
\begin{equation}\label{e:elaman}
E_{\lambda,n}(A)=M_{\chi_{\phi_{\lambda,n}^{-1}(A)}},\quad A\mbox{ is any Borel subset of }
\sigma(M_{\phi_{\lambda,n}})=
\essran(\phi_{\lambda,n}),
\end{equation}
where $\phi_{\lambda,n}=\phi_n|X_\lambda$. Then we have that
$((E_{\lambda,n},\essran(\phi_\lambda),\cH_{\lambda,n}))_{\lambda\in\Lambda}$ is a projective system
of spectral measures and letting $E_n$ denote the corresponding
$\Lloc(\cH_n)$-valued locally spectral measure, defined on the $\sigma$-algebra of all Borel subsets
of
\begin{equation*}
\sigma(M_{\phi_n})=\essran(\phi_n)
=\bigcup_{\lambda\in\Lambda}\essran(\phi_{\lambda,n}),
\end{equation*} it follows that
\begin{equation}\label{e:smacan}
E_n(A)=M_{\chi_{\phi_n^{-1}(A)}},\ A\mbox{ is any Borel subset of }\sigma(M_{\phi_n}).
\end{equation}
Also,
\begin{equation*}
M_{\phi_n}=\varprojlim_{\lambda\in\Lambda} M_{\phi_{\lambda_n}},
\end{equation*}
hence, in view of Theorem~\ref{t:st1}, $E_n$ defined 
at \eqref{e:smacan} is the locally spectral measure of $M_{\phi_n}$.

From here, it follows that the locally normal operator $I_n\otimes M_{\phi_n}$ defined as
\begin{equation}\label{e:iano}
I_n\otimes M_{\phi_n}=\varprojlim_{\lambda\in\Lambda} (I_n\otimes M_{\phi_{\lambda,n}})
\end{equation}
and acting on the locally Hilbert space $L^2_{\mathrm{loc}}(X_n,\mu_n)$ defined as in \eqref{e:eltwolocn},
has its spectral measure exactly $I_n\otimes E_n$, with the locally spectral measure $E_n$ defined as
in \eqref{e:smacan}. In particular, this shows that the spectral multiplicity of the locally normal operator 
$I_n\otimes M_{\phi_n}$ is exactly $n$. For example, if $z\in\CC$ is an eigenvalue of 
$I_n\otimes M_{\phi_n}$ then $z$ has 
spectral multiplicity $n$. In view of \eqref{e:mafi}, it follows that $z$ is an eigenvalue of $M_\phi$ of multiplicity 
$n$. Actually, more is true, any eigenvalue of $M_\phi$ of multiplicity $n$ should live in the locally Hilbert space
$\cH_n$ defined as in \eqref{e:halen} and viewed as a locally Hilbert subspace of $\cH$, see \eqref{e:haz}.
\end{example}

A representing strictly inductive system of Hilbert spaces $(\cH_\lambda)_{\lambda\in\Lambda}$ as in 
Example~\ref{ex:psconc} is called \emph{pseudo-concrete} and the corresponding
representing locally Hilbert space $\cH=\varinjlim_{\lambda\in\Lambda}\cH_\lambda$ is called 
\emph{pseudo-concrete} as well.

In the following we will use 
some basic knowledge on Abelian von Neumann algebras which can be seen, for 
example, in the classical texts of J.~Dixmier \cite{Dixmier},  
\c S.~Str\u atil\u a and L.~Zsido \cite{StratilaZsido} or the more modern books of R.V. Kadison and J.R. Ringrose
\cite{KadisonRingrose} or that 
of R. B.~Blackadar \cite{Blackadar}.
Let us recall that, given a Hausdorff topological space $X$, a measure space $(X,\Omega,\mu)$ is called
\emph{locally finite} if any open set of $X$ is measurable and for any $x\in X$ there exists $V$ an open 
neighbourhood of $x$ such that $\mu(V)$ is finite. We will intensively use the following result, see
Corollary III.1.5.18 in \cite{Blackadar}.

\begin{theorem}\label{t:abvn}
Let $\mathcal{Z}$ 
be an Abelian von Neumann algebra on the Hilbert space $\mathcal{H}$. Then, there exists a 
unique decomposition
\begin{equation*}
\mathcal{H}=\mathcal{H}_\infty\oplus\bigoplus_{n\geq 1} \mathcal{H}_n,
\end{equation*}
such that, for each $n\in \mathbb{N}\cup \{\infty\}$, $\mathcal{H}_n$ is invariant under $\mathcal{Z}$ and there 
exists a locally finite measure space $(X_n,\Omega_n,\mu_n)$ such that the 
compression of $\mathcal{Z}$ to $\mathcal{H}_n$, denoted by  $\mathcal{Z}_{\mathcal{H}_n}:=
P_{\cH_n}\mathcal{Z}|_{\cH_n}$, when viewed as a von Neumann algebra in $\cB(\cH_n)$,
is spatially isomorphic to the $n$-fold amplification 
$I_n\otimes L^\infty(X_n,\mu_n)$ acting on $\ell^2_n\otimes L^2(X_n,\mu_n)$, that is, there exists a unitary 
operator 
\begin{equation*}U_n\colon \cH_n\ra \ell^2_n\otimes L^2(X_n,\mu_n),\end{equation*} 
such that the map
\begin{equation*}
\mathcal{Z}_{\mathcal{H}_n}\ni T\mapsto  U_n TU_n^*\in I_n\otimes L^\infty(X_n,\mu_n)
\end{equation*}
is an isomorphism of von Neumann algebras.
\end{theorem}

The first step in our enterprise is to obtain the spectral theorem for locally normal operators on 
representing locally Hilbert spaces as multiplication operators with functions in $L^\infty_{\mathrm{loc}}$ in a 
form which contains also the spectral multiplicity aspect  encoded 
in the structure of the pseudo-concrete representing locally Hilbert space, as in Example~\ref{ex:psconc}. 
Note that in Theorem~\ref{t:abvn} there is nothing said on the uniqueness of the topological-measure spaces
$(X_n,\Omega_n,\mu_n)$, although a certain uniqueness property holds in a weak sense. 
The uniqueness property is one of the obstructions that prevent us from
obtaining the Second Spectral Theorem for locally normal operators without imposing additional 
assumptions on the directed set $\Lambda$. A second obstruction is related to the fact that we want a functional 
model modulo a locally unitary operator which has to be built in some way from bottom-up and this requires an 
induction process.
For these technical reasons, we will restrict our results to the special case when the directed set $\Lambda$ 
satisfies the following conditions.

\begin{itemize}
\item[(c1)] There exists a sequence $(\epsilon_m)_{m\geq 1}$ of elements of $\Lambda$ such that for 
each $m\geq 1$ we have $\epsilon_m\leq \epsilon_{m+1}$.
\item[(c2)] For each $\lambda\in \Lambda$ there exists $m\geq 1$ such that $\lambda\leq\epsilon_m$.
\item[(c3)]  For each $m\geq 1$ the set of all $\lambda\in\Lambda$ such that $\lambda\leq\epsilon_m$ 
is finite.
\end{itemize}
Because this set of conditions will play an important role in the main theorems, a directed set $\Lambda$ 
satisfying properties (c1) through (c3) is called \emph{sequentially finite}.

In the following, for each $\epsilon\in\Lambda$ we use the notation $\Lambda_\epsilon$ 
as defined in \eqref{e:lal}. It follows by (c2) and (c3) that for every $\lambda_0\in\Lambda$ the set of all 
$\lambda\in\Lambda$ with $\lambda\leq\lambda_0$ is finite.
Consequently, if the directed set $\Lambda$ satisfies conditions (c1) through (c3), it is necessarily countable.
Also, in each of the examples \ref{ex:hata1} through \ref{ex:hata3}, the strictly inductive systems associated to 
the Hata tree-like set satisfy conditions (c1) through (c3).

\begin{remark}\label{r:repas}
The next two theorems, which are the main results of this article, are stated for locally normal operators
acting on representing locally Hilbert spaces. On the one hand, as we will see in the next 
theorem, the assumption on the underlying locally Hilbert space to be representing allows us to obtain $*$-
isomorphisms implemented by locally unitary operators. On the other hand, given an arbitrary locally normal 
operator $N$, such that it acts on a locally Hilbert space that might not be representing, and considering the 
locally $C^*$-algebra $\cA$ generated by it, according to Theorem~5.1 in \cite{Inoue} and Example~\ref{e:rlhs}, 
we can always find a representing locally Hilbert space $\cH$ 
such that $\cA$ is $*$-isomorphic to a locally $C^*$-subalgebra of $\Lloc(\cH)$. Consequently,
if we drop the assumption on the underlying 
locally Hilbert space to be representing, we still can get a functional model in a more general sense.
\end{remark}

\begin{theorem}\label{t:st2} Assume that the directed set $\Lambda$ is sequentially finite.
Let $N\in\cB_\mathrm{loc}(\cH)$ be a locally normal operator on the representing locally Hilbert space
$\cH=\varinjlim_{\lambda\in\Lambda}\cH_\lambda$. Then, for each $n\in\NN\cup\{\infty\}$,
there exists a strictly inductive system of locally finite 
measure spaces $((X_{\lambda,n},\Omega_{\lambda,n},\mu_{\lambda,n})\}_{\lambda\in\Lambda}$ and,
with respect to 
\begin{equation*}
(X_n,\Omega_n,\mu_n)=\varinjlim_{\lambda\in\Lambda}(X_{\lambda,n},\Omega_{\lambda,n},\mu_{\lambda,n}),
\end{equation*}
defined as in \eqref{e:ilsm}, there exists $\phi_n\in L^\infty_\mathrm{loc}(X_n,\mu_n)$, 
see Example~\ref{ex:locli} for definition, such that, for each $\lambda\in\Lambda$,
\begin{equation*}
\sup_{n\geq 1}\|\phi_n|_{X_{\lambda,n}}\|_\infty<\infty,
\end{equation*}
and there exists a locally unitary operator
\begin{equation*}
U\colon \cH\ra \varinjlim_{\lambda\in\Lambda}\biggl(\bigl(\ell^2_\infty\otimes L^2(X_{\lambda,\infty},
\mu_{\lambda,\infty})\bigr)
\oplus \bigoplus_{n\geq 1} \bigl(\ell^2_n\otimes L^2(X_{\lambda,n},\mu_{\lambda,n})\bigr)\biggr),
\end{equation*}
where the spaces $L^2(X_{\lambda,n},\mu_{\lambda,n})$ are viewed as in Subsection~\ref{ss:crlhs},
such that
\begin{equation*}
U^*NU=\varprojlim_{\lambda\in\Lambda}\bigl((I_\infty\otimes M_{\phi_{\infty}|_{X_{\lambda,\infty}}})\oplus
\bigoplus_{n\geq 1} (I_n\otimes M_{\phi_{n}|_{X_{\lambda,n}}})\bigr).
\end{equation*}
\end{theorem}

\begin{proof} 
We start with a construction that will be used several times during the proof. 
For any $\lambda\in\Lambda_\epsilon$ we consider the orthogonal projection 
$P_{\lambda,\epsilon}\in\cB(\cH_\epsilon)$ with $\ran(P_{\lambda,\epsilon})=\cH_\lambda$. 
Let $\cM_\epsilon$ be the von Neumann algebra in $\cB(\cH_\epsilon)$ generated by $N_\epsilon$ and
$P_{\lambda,\epsilon}$, for all $\lambda\in\Lambda_\epsilon$. Formally, 
\begin{equation}\label{e:wstare}
\cM_\epsilon:=W^*(\{N_\epsilon\}\cup\{ P_{\lambda,\epsilon}\mid \lambda\in\Lambda_\epsilon\}).
\end{equation} 
The von Neumann algebra $\cM_\epsilon$ is commutative. To see this, by 
Proposition~\ref{p:lbo2}, we have the commutation of $N_\epsilon$ and $N_\epsilon^*$ 
with all orthogonal projections $P_{\lambda,\epsilon}$, $\lambda\in\Lambda_\epsilon$. 
Then, by Lemma~\ref{l:repof} we have the mutual commutation 
of the family of orthogonal projections $\{P_{\lambda,\epsilon}\}_{\lambda\leq\epsilon}$ because the locally 
Hilbert space $\cH$ is representing. Finally, $N_\epsilon$ commutes with $N_\epsilon^*$ since $N$ is locally 
normal.

We construct inductively, following $m\in \NN$, where $(\epsilon_m)_{m\geq 1}$ 
is the sequence of elements in $\Lambda$ subject to the conditions (c1) through (c3),
the strictly inductive systems of locally finite measure spaces 
$((X_{\lambda,n},\Omega_{\lambda,n},\mu_{\lambda,n})\}_{\lambda\in\Lambda}$, for all $n\in\NN\cup\{\infty\}$, 
that satisfy the following. \medskip

\textbf{Statement S.} \emph{For each $m\in \NN$ and for each $\lambda\in \Lambda_{\epsilon_m}$,
there exists a unique decomposition 
\begin{equation}\label{e:hepsi}
\cH_\lambda=\cH_{\lambda,\infty}\oplus\bigoplus_{n\geq 1}\cH_{\lambda,n}
\end{equation}  
such that, for each $n\in\NN\cup\{\infty\}$,
\begin{itemize}
\item[(i)] $\cH_{\lambda,n}$ is invariant under $\cM_\lambda$ and, consequently, 
the compression of $\cM_\lambda$ to $\cH_{\lambda,n}$,
\begin{equation*}
\cM_{\lambda,\cH_n}:=\{P_{\cH_{\lambda,n}}T|_{\cH_{\lambda,n}}\mid T\in\cM_\lambda\}\subseteq \cB(\cH_{\lambda,n}),
\end{equation*}
is a commutative von Neumann algebra;
\item[(ii)] there exists a unitary operator $U_{\lambda,n}\colon \cH_{\lambda,n}\ra \ell^2_n\otimes 
L^2(X_{\lambda,n},\mu_{\lambda,n})$ that implements the spatial isomorphism of von Neumann algebras
\begin{equation}\label{e:phie}
I_n\otimes L^\infty(X_{\lambda,n},\mu_{\lambda,n})\ni(I_n\otimes \phi)\mapsto
U_{\lambda,n}^* (I_n\otimes M_\phi)U_{\lambda,n}\in \cM_{\lambda,\cH_n}.
\end{equation}
\end{itemize}
}

In order to avoid trivial cases, without any loss of generality, 
we assume that for each $m\in\NN$ we have $\epsilon_m\neq \epsilon_{m+1}$ and that $\cH_{\epsilon_m}$ is a 
proper subspace of $\cH_{\epsilon_{m+1}}$. 

For $m=1$, with notation as in \eqref{e:lal}, we consider the finite set 
$\Lambda_{\epsilon_1}=\{\lambda_{0,1},\ldots,\lambda_{0,k_0}\}$ and 
note that, with notation as in \eqref{e:wstare}, the commutative von Neumann algebra $\cM_{\epsilon_1}$ is 
generated by $N_{\epsilon_1},P_{\lambda_{0,1},\epsilon_1},\ldots, P_{\lambda_{0,k_0},\epsilon_1}$. Recall that 
all these operators act in the Hilbert space $\cH_{\epsilon_1}$. For example, $P_{\lambda_{0,1},\epsilon_1}$ is 
the orthogonal projection of $\cH_{\epsilon_1}$ onto $\cH_{\lambda_{0,1}}$ and so on.
Due to the mutual commutation of the projections 
$P_{\lambda_{0,1},\epsilon_1},\ldots, P_{\lambda_{0,k_0},\epsilon_1}$, by employing ranges of all nonzero 
projections which are finite products of 
projections of type $P_{\lambda_{0,i},\epsilon_1}$ and $I_{\cH_{\epsilon_1}}-P_{\lambda_{0,j},\epsilon_1}$, for
$i,j\in\{1,\ldots,k_0\}$,
the Hilbert space $\cH_{\epsilon_1}$ 
can be split into an orthogonal sum of subspaces $\cH_{\epsilon_1,1},\ldots,\cH_{\epsilon_1,l_1}$
\begin{equation}
\cH_{\epsilon_1}=\cH_{\epsilon_1,1}\oplus\cdots\oplus\cH_{\epsilon_1,l_1},
\end{equation} 
all of them
invariant under $\cM_{\epsilon_1}$ and such that, for each $j\in \{1,\ldots,k_0\}$, the Hilbert space 
$\cH_{\lambda_{0,j}}$ is an orthogonal sum of some of these subspaces, 
$\cH_{\epsilon_1,p_1},\ldots,\cH_{\epsilon_1,p_{s_j}}$, that is
\begin{equation}\label{e:haloj}
\cH_{\lambda_{0,j}}= \cH_{\epsilon_1,p_1}\oplus \cdots \oplus \cH_{\epsilon_1,p_{s_j}}.
\end{equation}

For each $l=1,\ldots,l_1$, we apply Theorem~\ref{t:abvn} to the commutative reduced von Neumann algebra 
${\cM_{\epsilon_1}}_{\cH_{\epsilon_1,l}}$ and get a unique decomposition
\begin{equation}
\cH_{\epsilon_1,l}= \cH_{\epsilon_1,l,\infty}\oplus \bigoplus_{n\geq 1}\cH_{\epsilon_1,l,n},
\end{equation} 
such that, for each $n\in \NN\cup\{\infty\}$,
\begin{itemize}
\item[(i)] $\cH_{\epsilon_1,l,n}$ is invariant under $\cM_{\epsilon_1}$ and, consequently, 
$\cM_{\epsilon_1,\cH_{\epsilon_1,l,n}}$, the compression of $\cM_{\epsilon_1}$ to $\cH_{\epsilon_1,l,n}$, 
is a commutative von Neumann algebra in $\cB(\cH_{\epsilon_1,l,n})$;
\item[(ii)] there exists a locally finite measure space 
$(X_{\epsilon_1,l,n},\Omega_{\epsilon_1,l},\mu_{\epsilon_1,l,n})$  and a unitary operator $U_{\epsilon_1,l,n}
\colon \cH_{\epsilon_1,l,n}\ra \ell^2_n\otimes L^2(X_{\epsilon_1,l,n},\mu_{\epsilon_1,l,n})$ such that, the mapping
\begin{equation*}
I_n\otimes L^\infty(X_{\epsilon_1,l,n},\mu_{\epsilon_1,l,n})\ni I_n\otimes \phi \mapsto U_{\epsilon_1,l,n}^* 
(I_n\otimes M_\phi) U_{\epsilon_1,l,n}\in \cM_{\epsilon_1,\cH_{\epsilon_1,l,n}}
\end{equation*}
is an isomorphism of von Neumann algebras.
\end{itemize}

For each $j\in \{1,\ldots,k_0\}$, correspondingly to the decomposition \eqref{e:haloj}, and for each 
$n\in\NN\cup\{\infty\}$, we define
\begin{equation}\label{e:xala}
X_{\lambda_{0,j},n}:= \bigsqcup_{s=1}^{s_j} X_{\epsilon_1,s,n},
\end{equation}
$\Omega_{\lambda_{0,j},n}$ is the $\sigma$-algebra generated by $\Omega_{\epsilon_n,s,n}$ for all
$s=1,\ldots,s_j$, and the measure $\mu_{\lambda_{0,j},n}$ is defined by 
\begin{equation*}\mu_{\lambda_{0,j},n}(A):=\sum_{s=1}^{s_j} \mu_{\epsilon_1,s,n}(A\cap X_{\epsilon_1,s,n}),
\quad A\in \Omega_{\lambda_{0,j},n}.\end{equation*}
It is easy to see that, in this fashion, for each $n\in\NN\cup\{\infty\}$, the locally finite measure spaces 
$(X_{\lambda_{0,j},n},\Omega_{\lambda_{0,j},n},\mu_{\lambda_{0,j},n})$, for $j\in \{1,\ldots,k_0\}$, 
make a strictly inductive system. In view of \eqref{e:xala}, we define the unitary operators $U_{\lambda_{0,j},n}$
as the direct sum of the unitary operators $U_{\epsilon_1,s,n}$ for $s=1,\ldots,s_j$, for each $j=1,\ldots,k_0$.

Let us assume that, for a fixed $m\geq 1$ and for all $n\in\NN\cup\{\infty\}$, 
we have a strictly inductive system of measure spaces 
$(X_{\lambda,n},\Omega_{\lambda,n},\mu_{\lambda,n})_{\lambda \in\Lambda_{\epsilon_m}}$ as in Statement S. With 
notation as in \eqref{e:lal}, by 
assumption, $\Lambda_{\epsilon_{m+1}}$ is finite and contains $\Lambda_{\epsilon_m}$. We split
\begin{equation*}
\Lambda_{\epsilon_{m+1}}=\Lambda_{\epsilon_m}\sqcup (\Lambda_{\epsilon_{m+1}}\setminus \Lambda_{\epsilon_m}),
\end{equation*}
hence
\begin{equation*}
\Lambda_{\epsilon_m}=\{\lambda_{m,1},\ldots,\lambda_{m,k_m}\},\quad \Lambda_{\epsilon_{m+1}}\setminus 
\Lambda_{\epsilon_m}=\{\lambda_{m,k_m+1},\ldots,\lambda_{m,k_m+k_{m+1}}\},
\end{equation*}
for some natural numbers $k_m$ and $k_{m+1}$.
Then, the von Neumann algebra $\cM_{\epsilon_{m+1}}$ is generated by $N_{\epsilon_{m+1}}$
and the projections $P_{\lambda_{m,1},\epsilon_{m+1}}$ through 
$P_{\lambda_{m,k_m+k_{m+1}},\epsilon_{m+1}}$, where,
for example, $P_{\lambda_{m,1},\epsilon_{m+1}}$
is the projection of $\cH_{\epsilon_{m+1}}$ onto $\cH_{\lambda_{m,1}}$, and so on. Recall that all these 
operators act in $\cH_{\epsilon_{m+1}}$ and mutually commute, in particular the von Neumann algebra 
$\cM_{\epsilon_{m+1}}$ is commutative. Then, similarly to what we did for the case $m=1$, 
by employing ranges of all nonzero 
projections which are finite products of 
projections $P_{\lambda_{m,i},\epsilon_{m+1}}$ and $I_{\cH_{\epsilon_{m+1}}}-P_{\lambda_{m,j},\epsilon_{m+1}}$, for all 
$i,j\in\{1,\ldots,k_m\}$,
the Hilbert space $\cH_{\epsilon_{m+1}}$ can be split into 
mutually orthogonal subspaces $\cH_{\epsilon_{m+1},1},\ldots,\cH_{\epsilon_{m+1},l_m},
\cH_{\epsilon_{m+1},l_m+1},\ldots,\cH_{\epsilon_{m+1},l_m+l_{m+1}}$, all of them invariant under 
$\cM_{\epsilon_{m+1}}$,
\begin{equation}\label{e:hememu}
\cH_{\epsilon_{m+1}}=\bigoplus_{j=1}^{l_m} \cH_{\epsilon_{m+1},j}\oplus \bigoplus_{j=1}^{l_{m+1}} 
\cH_{\epsilon_{m+1},l_m+j},
\end{equation}
such that,
\begin{equation}\label{e:hemem}
\cH_{\epsilon_m} = \bigoplus_{j=1}^{l_m} \cH_{\epsilon_{m+1},j},
\end{equation}
for each $k=1,\ldots,k_m$, there exist indices $j_1,\ldots,j_{p_k}\in\{1,\ldots,l_m\}$, $p_k\leq l_m$, such that
\begin{equation}\label{e:celam}
\cH_{\lambda_{m,k}}=\bigoplus_{i=1}^{p_k} \cH_{\epsilon_{m+1},j_{p_i}},
\end{equation}
and, for each $k=k_m+1,\ldots,k_m+k_{m+1}$, there exist indices $j_1,\ldots,j_{p_k}\in \{1,\ldots,l_m+l_{m+1}\}$, 
$p_k\leq l_m+l_{m+1}$, such that \eqref{e:celam} holds.

With respect to the components $\cH_{\epsilon_{m+1},j}$, for $j=1,\ldots,l_m$, 
in the first orthogonal sum from the right hand side of
\eqref{e:hememu} and the components $\cH_{\epsilon_{m+1},l_m+j}$, for $j=1,\ldots,l_{m+1}$, in the latter orthogonal sum, we have to treat 
separately each of these cases.

\emph{First case:} $j\in \{1,\ldots,l_m\}$. 
By assumption, Statement S holds for $\lambda=\epsilon_m$ hence, there exists a unique decomposition
\begin{equation}\label{e:hepsiem}
\cH_{\epsilon_m}=\cH_{\epsilon_m,\infty}\oplus\bigoplus_{n\geq 1}\cH_{\epsilon_m,n}
\end{equation}  
and, for each $n\in\NN\cup\{\infty\}$, there exist the locally finite measure space 
$(X_{\epsilon_m,n},\Omega_{\epsilon_m,n},\mu_{\epsilon_m,n})$ and the unitary operator 
$U_{\epsilon_m,n}\colon \cH_{\epsilon_m,n}\ra \ell^2_n\otimes 
L^2(X_{\epsilon_m,n},\mu_{\epsilon_m,n})$ such that \eqref{e:phie} holds for $\lambda=\epsilon_m$.
Then,
in view of \eqref{e:hemem} and the fact that all the spaces there involved are invariant under 
$\cM_{\epsilon_{m}}$, 
by considering the projection of $\cH_{\epsilon_{m+1}}$ onto $\cH_{\epsilon_{m+1},j}$,
for each $n\in\NN\cup\{\infty\}$ we can get a locally finite measure space 
$(X_{\epsilon_{m+1},j,n},\Omega_{\epsilon_{m+1},j,n},\mu_{\epsilon_{m+1},j,n})$ which is obtained by restriction 
from the locally finite measure space $(X_{\epsilon_m,n},\Omega_{\epsilon_m,n},\mu_{\epsilon_m,n})$, and a 
unitary operator $U_{\epsilon_{m+1},j,n}\colon \cH_{\epsilon_{m+1},j,n}\ra \ell^2_n\otimes 
L^2(X_{\epsilon_{m+1},j,n},\mu_{\epsilon_{m+1},j,n})$, where we have the decomposition 
\begin{equation}
\cH_{\epsilon_{m+1},j}=\cH_{\epsilon_{m+1},j,\infty}\oplus\bigoplus_{n=1}^\infty \cH_{\epsilon_{m+1},j,n},
\end{equation}
which is obtained from \eqref{e:hepsiem}, such that the unitary operator $U_{\epsilon_{m+1},j,n}$ implements 
the spatial isomorphism of the reduced von Neumann algebra $\cM_{\epsilon_{m+1},\cH_{\epsilon_{m+1},j,n}}$ 
with the $n$-times amplification von Neumann algebra 
$I_n\otimes L^\infty(X_{\epsilon_{m+1},j,n},\mu_{\epsilon_{m+1},j,n})$.

\emph{Second case:} $j\in \{l_m+1,\ldots,l_{m}+l_{m+1}\}$. In this case we have to play differently and apply 
Theorem~\ref{t:abvn} and get the unique decomposition
\begin{equation}
\cH_{\epsilon_{m+1},j}=\cH_{\epsilon_{m+1},j,\infty}\oplus\bigoplus_{n=1}^\infty \cH_{\epsilon_{m+1},j,n},
\end{equation}
with all these spaces invariant under $\cM_{\epsilon_{m+1}}$, and, for each $n\in\NN\cup\{\infty\}$,
we can get a locally finite measure space 
$(X_{\epsilon_{m+1},j,n},\Omega_{\epsilon_{m+1},j,n},\mu_{\epsilon_{m+1},j,n})$
and a 
unitary operator $U_{\epsilon_{m+1},j,n}\colon \cH_{\epsilon_{m+1},j,n}\ra \ell^2_n\otimes 
L^2(X_{\epsilon_{m+1},j,n},\mu_{\epsilon_{m+1},j,n})$ which implements 
the spatial isomorphism of the reduced von Neumann algebra $\cM_{\epsilon_{m+1},\cH_{\epsilon_{m+1},j,n}}$ 
with the $n$-times amplification von Neumann algebra 
$I_n\otimes L^\infty(X_{\epsilon_{m+1},j,n},\mu_{\epsilon_{m+1},j,n})$.

Further on, let $\lambda\in \Lambda_{\epsilon_{m+1}}\setminus \Lambda_{\epsilon_m}$, hence $\lambda=\lambda_{m,k_m+k}$ 
for some $k\in \{1,\ldots,k_{m+1}\}$. Then there exist indices $j_1,\ldots,j_{p_k}\in \{1,\ldots,l_m+l_{m+1}\}$, 
$p_k\leq l_m+l_{m+1}$, such that \eqref{e:celam} holds. For each $n\in \NN\cup\{\infty\}$ 
we define  
\begin{equation}\label{e:xalam}
X_{\lambda_{m,k_m+k,n}} = \bigsqcup_{i=1}^{p_k} X_{\epsilon_{m+1},p_i,n},
\end{equation}
let $\Omega_{\lambda_{m,k_m+k,n}}$ be the $\sigma$-algebra generated by 
$\Omega_{\epsilon_{m+1},p_i,n}$ for $i=1,\ldots,p_k$, and define the measure 
$\mu_{\epsilon_{m+1},k_m+k,n}\colon\Omega_{\epsilon_{m+1},k_m+k,n}\ra [0,+\infty]$ by
\begin{equation}
\mu_{\epsilon_{m+1},k_m+k,n}(A):= \sum_{i=1}^{p_k} \mu_{\epsilon_{m+1},p_i,n}(A\cap X_{\epsilon_{m+1},p_i,n}),\quad A\in \Omega_{\epsilon_{m+1},k_m+k,n}.
\end{equation}
The unitary operators $U_{\lambda_{m,k_m+k,n}}$ are defined, in view of  \eqref{e:xalam}, as the direct sum of
operators $U_{\epsilon_{m+1},p_i,n}$ for $i=1,\ldots,p_k$.
It is now easy to see that Statement S holds for $\epsilon_{m+1}$ and the general induction step is proven, which concludes the proof of Statement S for all $m\in\NN$.

Once the Statement S is proven, in view of the technical condition (c2), it follows that, for each $n\in\NN\cup\{\infty\}$ we have the strictly inductive system of locally finite measure spaces
$((X_{\lambda,n},\Omega_{\lambda,n},\mu_{\lambda,n})\}_{\lambda\in\Lambda_{\epsilon_m}}$,
such that, for each $\lambda\in \Lambda$, the statements (i) and (ii) as in Statement S hold.
Then, as in Example~\ref{ex:psconc},
we can define the pseudo-concrete representing locally Hilbert space
\begin{equation}
\varinjlim_{\lambda\in\Lambda}\biggl(\bigl(\ell^2_\infty\otimes L^2(X_{\lambda,\infty})\bigr)\oplus
\bigoplus_{n\geq 1}\bigl(\ell^2_n\otimes L^2(X_{\lambda,n})\bigr)\biggl).
\end{equation}

Further on, for any $\lambda\in\Lambda$ we consider the unitary operator
\begin{equation}\label{e:ulam}
U_\lambda:=U_{\lambda,\infty}\oplus\bigoplus_{n\geq 1}U_{\lambda,n}\colon
\cH_\lambda \ra \bigl(\ell^2_\infty\otimes L^2(X_{\lambda,\infty})\bigr)\oplus
\bigoplus_{n\geq 1}\bigl(\ell^2_n\otimes L^2(X_{\lambda,n})\bigr),
\end{equation}
and then, by Proposition~\ref{p:lbo2}, we observe that the net $(U_\lambda)_{\lambda\in\Lambda}$
gives rise to a locally unitary operator 
\begin{equation}\label{e:uvar}
U:=\varprojlim_{\lambda\in\Lambda} U_\lambda\colon \cH=\varinjlim_{\lambda\in\Lambda} \cH_\lambda
\ra \varinjlim_{\lambda\in\Lambda} \bigl(\ell^2_\infty\otimes L^2(X_{\lambda,\infty})\bigr)\oplus
\bigoplus_{n\geq 1}\bigl(\ell^2_n\otimes L^2(X_{\lambda,n})\bigr).
\end{equation}

Finally, $N=\varprojlim_{\lambda\in\Lambda} N_\lambda$ where $N_\lambda\in\cM_\lambda\subset
\cB(\cH_\lambda)$ is normal for each $\lambda \in\Lambda$. In view of \eqref{e:hepsi}, for each 
$\lambda\in\Lambda$ and each $n\in\NN\cup\{\infty\}$, let $Q_{\lambda,n}$ be the projection of 
$\cH_\lambda$ onto $\cH_{\lambda,n}$.
Then, for each $\lambda\in\Lambda$ and
each $n\in\NN\cup\{\infty\}$ the normal operator $N_{\lambda,n}:=Q_{\lambda,n}N_\lambda |_{\cH_{\lambda,n}}
\in \cM_{\lambda,n}\subset \cB(\cH_{\lambda,n})$ hence, there exists uniquely 
$\phi_{\lambda,n}\in L^\infty(X_{\lambda,n},\mu_{\lambda,n})$ such that
\begin{equation}\label{e:nela}
N_{\lambda,n}=U^*_{\lambda,n}(I_n\otimes M_{\phi_{\lambda,n}}) U_{\lambda,n}.
\end{equation}
Then,
\begin{equation}\label{e:fela}
\sup_{n\geq 1}\|\phi_{\lambda,n}\|_\infty=\sup_{n\geq 1}\|N_{\lambda,n}\|
=\sup_{n\geq 1}\|Q_{\lambda,n}N_\lambda |_{\cH_{\lambda,n}}\|\leq \|N_\lambda\|<\infty.
\end{equation}
From \eqref{e:nela} we get
\begin{align}
N_\lambda & = N_{\lambda,\infty}\oplus\bigoplus_{n\geq 1} N_{\lambda,n}\nonumber \\
& = U^*_{\lambda,\infty}(I_\infty\otimes M_{\phi_{\lambda,n}}) U_{\lambda,\infty} \oplus
\bigoplus_{n\geq 1} U^*_{\lambda,n}(I_n\otimes M_{\phi_{\lambda,n}}) U_{\lambda,n}\nonumber \\
& = U^*_{\lambda}\bigl((I_\infty\otimes M_{\phi_{\lambda,\infty}}) \oplus \bigoplus_{n\geq 1}
(I_n\otimes M_{\phi_{\lambda,n}})\bigr) U_{\lambda}\label{e:nulae}
\end{align}
and then, by taking the projective limits on both sides we get
\begin{equation}\label{e:neu}
N= U^*\varprojlim_{\lambda\in\Lambda}\bigl((I_\infty\otimes M_{\phi_{\lambda,\infty}})\oplus
\bigoplus_{n\geq 1}(I_n\otimes M_{\phi_{\lambda,n}})\bigr)U,
\end{equation}
where the locally unitary operator $U$ is defined at \eqref{e:uvar}.
\end{proof}

The previous theorem can be transferred to a spectral theorem for locally normal operators on representing 
locally Hilbert spaces in terms of projective limits of multiplication operators on concrete representing 
locally Hilbert spaces.

\begin{theorem}\label{t:st3}
Let $N\in\cB_\mathrm{loc}(\cH)$ be a locally normal operator on the representing locally Hilbert space
$\cH=\varinjlim_{\lambda\in\Lambda}\cH_\lambda$ and assume that the directed set $\Lambda$ is sequentially 
finite. Then, there exists a strictly inductive system of locally finite 
measure spaces $((X_{\lambda},\Omega_{\lambda},\mu_{\lambda}))_{\lambda\in\Lambda}$ and,
with respect to 
\begin{equation*}
(X,\Omega,\mu)=\varinjlim_{\lambda\in\Lambda}(X_{\lambda},\Omega_{\lambda},\mu_{\lambda}),
\end{equation*}
defined as in \eqref{e:ilsm}, there exists $\phi\in L^\infty_\mathrm{loc}(X,\mu)$, 
see Example~\ref{ex:locli} for definition, and there exists a locally unitary operator
\begin{equation*}
V\colon \cH\ra \varinjlim_{\lambda\in\Lambda} L^2(X_{\lambda},\mu_\lambda),
\end{equation*}
such that, letting $M_\phi=\varprojlim_{\lambda\in\Lambda}M_{\phi|_{X_{\lambda}}}$ as in
 Example~\ref{ex:lnop}, we have
\begin{equation*}
N=V^*M_\phi V.
\end{equation*}
\end{theorem}

\begin{proof} With notation as in Theorem~\ref{t:st2}, fix $\lambda\in\Lambda$ arbitrary. For each $n\in\NN$ and 
each $k=1,\ldots,n$, let $(X_{\lambda,n,k},\Omega_{\lambda,n,k},\mu_{\lambda,n,k})$ be a copy of 
$(X_{\lambda,n},\Omega_{\lambda,n},\mu_{\lambda,n})$. For each $k\in\NN$ 
let $(X_{\lambda,\infty,k},\Omega_{\lambda,\infty,k},\mu_{\lambda,\infty,k})$ be a copy of
$(X_{\lambda,\infty},\Omega_{\lambda,\infty},\mu_{\lambda,\infty})$. Then let $X_\lambda$ be defined by disjoint 
union of all these sets,
\begin{equation}\label{e:xela} X_\lambda:= \bigsqcup_{k=1}^\infty X_{\lambda,\infty,k}\sqcup
\bigsqcup_{n=1}^\infty\bigsqcup_{k=1}^n X_{\lambda,n,k}.
\end{equation} 
On $X_\lambda$ we define the $\sigma$-algebra $\Omega_\lambda$ in the following way: an arbitrary set 
$S\subseteq X_\lambda$ belongs to $\Omega_\lambda$ if and only if 
$S\cap X_{\lambda,n,k}\in \Omega_{\lambda,n,k}$, for all $n\in \NN$ and all $k=1,\ldots,n$, and 
$S\cap X_{\lambda,\infty,k}\in \Omega_{\lambda,\infty,k}$ for all $k\in\NN$. On $X_\lambda$ we can also define 
the topology in a similar fashion: an arbitrary set $S\subseteq X$ is open if and only if  
$S\cap X_{\lambda,n,k}$ is open in $X_{\lambda,n,k}$, for all $n\in \NN$ and all $k=1,\ldots,n$, and 
$S\cap X_{\lambda,\infty,k}$ is open in $X_{\lambda,\infty}$ for all $k\in\NN$. Then it can be seen that any open 
subset of $X_\lambda$ is measurable.
The measure $\mu_\lambda$ on $(X_\lambda,\Omega_\lambda)$ is defined in the natural fashion: 
for any $S\in \Omega_\lambda$ let
\begin{equation}\label{e:mela}
\mu_\lambda(S):=\sum_{k=1}^\infty \mu_{\lambda,\infty,k}(S\cap X_{\lambda,\infty,k})
+ \sum_{n=1}^\infty \sum_{k=1}^n \mu_{\lambda,n,k}(S\cap X_{\lambda,n,k}).
\end{equation}
Then we can see that the measure space $(X_\lambda,\Omega_\lambda,\mu_\lambda)$ is locally finite.

It is easy to see that $(X_\lambda,\Omega_\lambda,\mu_\lambda)_{\lambda\in\Lambda}$ is a strictly inductive
system of measure spaces and let $(X,\Omega,\mu)$ be its inductive limit in the sense of 
definitions \eqref{e:lms} and \eqref{e:lmsa}. Then let $L^2_{\mathrm{loc}}(X,\mu)$ be the inductive limit of the
strictly inductive system of Hilbert spaces $(L^2(X_\lambda,\mu_\lambda))_{\lambda\in\Lambda}$ in the sense of
\eqref{e:eltwoloc}.  Also, from the definition of the measure space 
$(X_\lambda,\Omega_\lambda,\mu_\lambda)$,
see \eqref{e:xela} and \eqref{e:mela}, it follows that there is a natural identification (unitary operator)
\begin{equation}
L^2(X_\lambda,\mu_\lambda)\simeq \bigoplus_{k=1}^\infty L^2(X_{\lambda,\infty,k},\mu_{\lambda,\infty,k})\oplus
\bigoplus_{n=1}^\infty\bigoplus_{k=1}^n L^2(X_{\lambda,n,k},\mu_{\lambda,n,k}),
\end{equation}
and then, it is easy to see that there exists a canonical unitary operator
\begin{equation}
T_\lambda\colon  \bigl(\ell^2_\infty\otimes L^2(X_{\lambda,\infty},
\mu_{\lambda,\infty})\bigr)
\oplus \bigoplus_{n\geq 1} \bigl(\ell^2_n\otimes L^2(X_{\lambda,n},\mu_{\lambda,n})\bigr)\ra L^2(X_\lambda,\mu_\lambda).
\end{equation}
In addition, the net $(T_\lambda)_{\lambda\in\Lambda}$ is a projective system and 
hence, by Proposition~\ref{p:lbo2}, its projective limit 
\begin{equation}T=\varprojlim_{\lambda\in\Lambda}T_\lambda\colon 
\varinjlim_{\lambda\in\Lambda} \bigl(\ell^2_\infty\otimes L^2(X_{\lambda,\infty},
\mu_{\lambda,\infty})\bigr)
\oplus \bigoplus_{n\geq 1} \bigl(\ell^2_n\otimes L^2(X_{\lambda,n},\mu_{\lambda,n})\bigr)
\ra L^2_{\mathrm{loc}}(X,\mu)
\end{equation} 
is a locally unitary operator. 

We now define the functions $\phi_\lambda$, for an arbitrary fixed $\lambda\in\Lambda$. For each 
$n\in\NN\cup\{\infty\}$ consider the function $\phi_{\lambda,n}\in L^\infty(X_{\lambda,n},\mu_{\lambda,n})$ 
as in \eqref{e:nela}. In view of  \eqref{e:xela}, for any $n\in\NN$ and $k=1,\ldots,n$, let 
$\phi_\lambda|_{X_{\lambda,n,k}}:=\phi_{\lambda,n}$ and for any $k\in\NN$ let 
$\phi_\lambda|_{X_{\lambda,\infty,k}}:=\phi_{\lambda,\infty}$. In view of \eqref{e:fela}, 
$\phi_\lambda\in L^\infty(X_\lambda,\mu_\lambda)$. In addition, because $(N_\lambda)_{\lambda\in\Lambda}$ 
is a projective system of normal operators, it follows that $(\phi_\lambda)_{\lambda\in\Lambda}$ is a projective 
system of functions and hence there exists uniquely its projective limit
$\phi\in L^\infty_{\mathrm{loc}}(X,\mu)$ as in Example~\ref{ex:locli}. Then, for each $\lambda\in\Lambda$, we
have
\begin{equation}\label{e:mefix}
M_{\phi|_{X_\lambda}} =T_\lambda \bigl((I_\infty\otimes M_{\phi_{\lambda,\infty}}) \oplus \bigoplus_{n\geq 1}
(I_n\otimes M_{\phi_{\lambda,n}})\bigr)
T_\lambda^*.
\end{equation}

Finally, we define $V_\lambda=T_\lambda U_\lambda$, for all $\lambda\in\Lambda$, where the locally unitary 
operator $U_\lambda$ is defined at \eqref{e:ulam}  and we notice that $(V_\lambda)_{\lambda \in\Lambda}$ is a projective system of unitary operators hence, by Proposition~\ref{p:lbo2}, there exists uniquely the locally
unitary operator $V=\varprojlim_{\lambda\in\Lambda} V_\lambda$.
By \eqref{e:nulae}  and \eqref{e:mefix} we have
\begin{align*}
N & =\varprojlim_{\lambda\in\Lambda} N_\lambda  = \varprojlim_{\lambda\in\Lambda}
U^*_{\lambda}\bigl((I_\infty\otimes M_{\phi_{\lambda,\infty}}) \oplus \bigoplus_{n\geq 1}
(I_n\otimes M_{\phi_{\lambda,n}})\bigr) U_{\lambda} \\
& = \varprojlim_{\lambda\in\Lambda} U_\lambda^* T_\lambda^* M_{\phi|_{X_\lambda}} T U_\lambda 
= U^*T^* M_\phi TU= V^* M_\phi V,
\end{align*}
and the proof is finished. 
\end{proof}

\subsection{The Third Spectral Theorem. Direct Integral Representations.}\label{ss:ttst}

As a direct consequence of the fact that Theorem~\ref{t:st2} contains information on the spectral multiplicity, 
we can show that from here one can obtain a variant of the Third Spectral Theorem for locally normal operators, 
that is, in terms of direct integrals of locally Hilbert spaces. The idea is to follow the lines of the proof of 
Theorem~\ref{t:st3} but in such a way that the spectral multiplicity be preserved. In this article, it is not our 
intention to dive into the 
conceptualisation of the theory of direct integrals for locally Hilbert spaces in this 
subsection (an attempt is already made in \cite{KulkarniPamula} for the case of Abelian 
locally von Neumann algebras) and for this reason we keep this subsection to a minimum.

With notations and assumptions as in Theorem~\ref{t:st2} and its proof, assume that all Hilbert 
spaces $\cH_\lambda$ are separable. Since $\Lambda$ is countable, it follows that both the locally Hilbert 
space $\cH$ and its completion $\widetilde \cH$ are separable. For arbitrary but 
fixed $\lambda\in \Lambda$ and any $n\in\NN\cup\{\infty\}$, consider the measure space
$(X_{\lambda,n},\Omega_{\lambda,n},\mu_{\lambda_n})$ constructed during the proof Statement S in the proof 
of Theorem~\ref{t:st2} and note that, due to the separability of 
$\cH_\lambda$, e.g.\ see 
Theorem~4.4.3 in \cite{BratelliRobinson}, it is a \emph{standard measure space}, that is, modulo a subset of null 
measure, $X_{\lambda,n}$ is a Polish space (complete, separable, metric space) and $\Omega_{\lambda,n}$ is 
a Borel $\sigma$-algebra (that is, it contains all open subsets of $X_{\lambda,n}$). 
Moreover, let $(X_\lambda,\Omega_\lambda,\mu_\lambda)$ be the disjoint sum of the family of measure spaces
$\{(X_{\lambda,n},\Omega_{\lambda,n},\mu_{\lambda_n})\}_{n\in\NN\cup\{\infty\}}$, as we explain in the 
following. First, let $X_\lambda$ be defined by the disjoint union of all these sets,
\begin{equation}\label{e:xelas} X_\lambda:= X_{\lambda,\infty}\sqcup
\bigsqcup_{n=1}^\infty X_{\lambda,n}.
\end{equation} 
On $X_\lambda$ we define the $\sigma$-algebra $\Omega_\lambda$ in the following way: an arbitrary set 
$S\subseteq X_\lambda$ belongs to $\Omega_\lambda$ if and only if 
$S\cap X_{\lambda,n}\in \Omega_{\lambda,n}$, for all $n\in \NN\cup\{\infty\}$. 
On $X_\lambda$ we can also define the metric, modulo a subset of null measure, by the Fr\'echet construction. 
Then it can be seen that any open subset of $X_\lambda$ is measurable.
The measure $\mu_\lambda$ on $(X_\lambda,\Omega_\lambda)$ is defined in the natural fashion: 
for any $S\in \Omega_\lambda$ let
\begin{equation}\label{e:melas}
\mu_\lambda(S):= \mu_{\lambda,\infty}(S\cap X_{\lambda,\infty})
+ \sum_{n=1}^\infty \mu_{\lambda,n}(S\cap X_{\lambda,n}).
\end{equation}
Then we can see that the measure space $(X_\lambda,\Omega_\lambda,\mu_\lambda)$ is a standard measure 
space.

In view of Definition~4.4.1A and Theorem~4.4.2 in \cite{BratelliRobinson}, viewing  the spaces 
$L^2(X_{\lambda,n},\mu_{\lambda,n})$ as in Subsection~\ref{ss:crlhs},
letting
\begin{equation}\cG_\lambda(x):= \ell^2_n\otimes L^2(X_{\lambda,n},\mu_{\lambda,n}),\quad x\in X_\lambda,
\end{equation}
where $n\in \NN\cup[\{\infty\}$ is uniquely determined such that $x\in X_{\lambda,n}$,
the Hilbert space 
\begin{equation}\label{e:dirint}
\int_{X_\lambda}^\oplus \cG_\lambda(x)\de\mu_\lambda(x):= \bigl(\ell^2_\infty\otimes L^2(X_{\lambda,\infty},
\mu_{\lambda,\infty})\bigr)
\oplus \bigoplus_{n\geq 1} \bigl(\ell^2_n\otimes L^2(X_{\lambda,n},\mu_{\lambda,n})\bigr)
 \end{equation}
is a direct integral of Hilbert spaces (compare with 
\cite{BratelliRobinson} pp.~ 443). Let us observe that the net 
$(X_\lambda,\Omega_\lambda,\mu_\lambda)_{\lambda\in\Lambda}$ is a strictly inductive system of standard
measure spaces and, consequently, we can define its inductive limit $(X,\Omega,\mu)$ as in 
Subsection~\ref{ss:sismes}.
Also, we observe that, for each $x\in X$, $(\cG_\lambda(x))_{\lambda\in\Lambda}$ is a strictly inductive system of Hilbert spaces and 
hence we can define the locally Hilbert spaces
\begin{equation}
\cG(x):=\varinjlim_{\lambda\in\Lambda} \cG_\lambda(x),\quad x\in X.
\end{equation}

Since the net of the Hilbert spaces from 
the right hand side of \eqref{e:dirint}, indexed by $\lambda\in\Lambda$, is a strictly inductive system of Hilbert spaces, 
it follows that we can define the \emph{direct integral locally Hilbert 
space} induced by this system
\begin{align}
\int_{X}^{\oplus,\mathrm{loc}} \cG(x)\de\mu(x) & := \varinjlim_{\lambda\in\Lambda} 
\int_{X_\lambda}^\oplus \cG_\lambda(x)\de\mu_\lambda(x) \nonumber\\
& = \varinjlim_{\lambda\in\Lambda}
\bigl(\ell^2_\infty\otimes L^2(X_{\lambda,\infty},
\mu_{\lambda,\infty})\bigr)
\oplus \bigoplus_{n\geq 1} \bigl(\ell^2_n\otimes L^2(X_{\lambda,n},\mu_{\lambda,n})\bigr).\label{e:dirintloc}
\end{align}

Now, for each $\lambda\in\Lambda$ and $n\in \NN\cup\{\infty\}$, let 
$\phi_{\lambda,n}\in L^\infty(X_{\lambda,n},\mu_{\lambda,n})$ be defined by \eqref{e:nela}. Then, in view 
of \eqref{e:fela} it follows that, for each $\lambda\in\Lambda$, 
we can get $\phi_\lambda\in L^\infty(X_\lambda,\mu_\lambda)$ such that 
$\phi_\lambda|_{X_{\lambda,n}}=\phi_{\lambda,n}$ for all $n\in\NN\cup\{\infty\}$. This enables us to define the 
scalar operator $S_{\phi_\lambda}(x)$ on $\cG_{\lambda}(x)$, for all $x\in X_\lambda$, of multiplication with
$\phi_\lambda(x)$. Then, we get the net of \emph{diagonalisable} operators on the direct integral
Hilbert space defined at \eqref{e:dirint}
\begin{equation*}
\int_{X_\lambda}^\oplus S_{\phi_\lambda}(x)\de\mu_\lambda(x),\quad \lambda\in\Lambda,
\end{equation*}
which turns out to define, via Proposition~\ref{p:lbo2}, a \emph{diagonalisable} locally bounded operator
on the direct integral locally Hilbert space defined at \eqref{e:dirintloc}
\begin{equation}\label{e:diag}
\int_X^{\oplus,\mathrm{loc}} S_\phi(x)\de\mu(x):= \varprojlim_{\lambda\in\Lambda}
\int_{X_\lambda} ^\oplus S_{\phi_\lambda}(x)\de\mu_\lambda(x).
\end{equation}
Finally, an argument similar to the one 
used in order to obtain \eqref{e:neu} shows that we can get a locally unitary operator 
\begin{equation*}
U\colon \cH\ra \int_{X}^{\oplus,\mathrm{loc}} \cG(x)\de\mu(x)
\end{equation*}
such that
\begin{equation*}
N=U^* \int_X^{\oplus,\mathrm{loc}} S_\phi(x)\de\mu(x)\,U.
\end{equation*}

Summing up the reasoning and the formulae from above we have.

\begin{theorem}\label{t:st4}
Assume that the directed set $\Lambda$ is sequentially finite.
Let $N\in\cB_\mathrm{loc}(\cH)$ be a locally normal operator on the representing locally Hilbert space
$\cH=\varinjlim_{\lambda\in\Lambda}\cH_\lambda$ subject to the condition that all Hilbert spaces $\cH_\lambda$ 
are separable. Then there exists a strictly inductive net 
$(X_\lambda,\Omega_\lambda,\mu_\lambda)_{\lambda\in\Lambda}$ of 
standard measure spaces with respect to which $N$ is locally unitarily equivalent to a diagonalisable locally 
bounded operator described at \eqref{e:dirintloc} on a direct integral locally Hilbert space described at 
\eqref{e:diag}.
\end{theorem}

\subsection{Final Comments.}\label{ss:last}
 We have now an almost complete picture of the spectral theory for locally normal operators. The First Spectral 
Theorem, in terms of spectral measures, see Theorem~\ref{t:st1}, 
is the most general, without any assumption on the underlying locally 
Hilbert space or the directed set. The Second Spectral Theorem, in terms of multiplication operators with 
functions which are locally of $L^\infty$ type, see Theorem~\ref{t:st2} and Theorem~\ref{t:st3},
is proven under additional assumptions on the underlying Hilbert space as a representing 
one and the assumption on the directed set $\Lambda$ to be sequentially finite. The first assumption is natural 
and not that restrictive, if we want locally unitarily equivalence, see Remark~\ref{r:repas}. 
But the most disturbing one is the latter assumption on the 
directed set to be locally finite. The obstruction that prevents us from removing this assumption comes from the 
representations of Abelian von Neumann algebras, see Theorem~\ref{t:abvn}. 
To be more precise, a careful examination of the proof of 
Theorem~\ref{t:st2} shows that, by the mathematical induction process following $m\in\NN$ which is
used to prove Statement S, at each step we keep all the locally finite measure spaces 
$(X_{\lambda,n},\Omega_{\lambda,n}, \mu_{\lambda,n})$ built at the previous steps and only construct the new 
ones. If we apply Theorem~\ref{t:abvn} and construct the locally finite measure spaces 
$(X_{\lambda,n},\Omega_{\lambda,n}, \mu_{\lambda,n})$ for arbitrary $\lambda\in\Lambda$ and $n\in\NN$, 
we have at least two problem. First, we have to show that we get a strictly inductive system of locally finite 
measure spaces and second we have to build the locally unitary operator that implements the functional model.
Work remains to be done on the question whether the latter assumption can be weakened or not. 

The Third Spectral Theorem, in terms of direct integrals of locally Hilbert spaces, see Theorem~\ref{t:st4}, 
is obtained under all the previous assumptions plus the separability of the underlying locally Hilbert spaces. The 
assumption on separability is natural, having in mind the restriction required by measure theory. However, the 
assumption on local finiteness of the directed set is imposed in view of the fact that we obtain the 
Third Spectral Theorem as a consequence of the Second Spectral Theorem. In view of the recent attempt to get 
direct integral representations of Abelian locally von Neumann algebras, see \cite{KulkarniPamula}, 
we think that the idea of the proof of 
Theorem~\ref{t:st4} can be used for such a purpose. The advantage of using Theorem~\ref{t:abvn} is that the 
measurability is implicit and what remains to be done is to transfer this construction to measurable bundles of 
Hilbert spaces in a way compatible with the locally space structure. We will make this explicit in a future paper \cite{GheondeaKulkarniPamula}. 

\end{document}